\documentclass[a4paper,11pt, leqno]{amsart}
\usepackage{amssymb}
\usepackage{amsthm}
\usepackage{mathrsfs, amsfonts, amsmath}

\newcommand{\rk}{{\mathbb R}^2}
\newcommand{\R}{{\mathbb R}}

\newcommand{\haus}{{\mathcal H}}
\newcommand{\dhaus}{{\, d \mathcal H}}
\newcommand{\hausk}{{\mathcal H}^2}
\newcommand{\dhausk}{{\,d \mathcal H}^2}
\newcommand{\modu}{\operatorname{mod}}

\newtheorem{theorem}{\textbf{THEOREM}}[section]
\newtheorem{lemma}[theorem]{\textsc{Lemma}}
\newtheorem{proposition}[theorem]{\textsc{Proposition}}
\newtheorem{corollary}[theorem]{\textsc{Corollary}}
\theoremstyle{definition}
\newtheorem{definition}[theorem]{\textsc{Definition}}
\newtheorem{question}[theorem]{\textsc{Question}}
\newtheorem{example}[theorem]{\textsc{Example}}
{\theoremstyle{remark} \newtheorem{remark}[theorem]{Remark}}
\def\charfn_#1{{\raise1.2pt\hbox{$\chi_{\kern-1pt\lower3pt\hbox{{$\scriptstyle#1$}}}$}}}
\def\leq{\leqslant }
\def\geq{\geqslant }

\def\XXint#1#2#3{{\setbox0=\hbox{$#1{#2#3}{\int}$}
\vcenter{\hbox{$#2#3$}}\kern-.5\wd0}}

\begin{document}

\title{Uniformization of two-dimensional metric surfaces}
\author{Kai Rajala} \thanks{Research supported by the Academy of Finland, project number 257482. Parts of this research were 
carried out while the author was visiting the University of Michigan. He thanks the Department of Mathematics for hospitality. 
\newline {\it 2010 Mathematics Subject Classification.} Primary 30L10, Secondary 30C65, 28A75, 51F99, 52A38. 
}
\date{}

\begin{abstract} 
We establish uniformization results for metric spaces that are homeomorphic to the Euclidean plane or sphere and have locally finite Hausdorff $2$-measure. Applying the geometric definition of quasiconformality, we give a necessary and sufficient condition for such spaces to be QC equivalent to the Euclidean plane, disk, or sphere. Moreover, we show that if such a QC parametrization exists, then the dilatation can be bounded by $2$. As an application, we show that the Euclidean upper bound for measures of balls is a sufficient condition for the existence of a $2$-QC parametrization. This result gives a new approach to the Bonk-Kleiner theorem on parametrizations of Ahlfors $2$-regular spheres by quasisymmetric maps. \end{abstract}
\maketitle

\renewcommand{\baselinestretch}{1.2}

\tableofcontents

\section{Introduction}

\subsection{Background}

One of the main problems in Analysis in Metric Spaces is to find conditions under which a metric space can be mapped to a 
Euclidean space by homeomorphisms with good geometric and analytic properties. In particular, non-smooth versions of 
the classical uniformization theorem have found applications in several different areas of mathematics. This problem is very 
difficult in general, and many basic questions remain open. 

Without the presence of smoothness, the parametrizations one looks for are usually quasiconformal (QC) or quasisymmetric (QS) homeomorphisms, which distort shapes in a controlled manner (see Sections \ref{defsection} and \ref{quasisym} 
for definitions), or bi-Lipschitz homeomorphisms, which also distort distances in a controlled manner.  We will here concentrate on QC and QS maps. Concerning the existence of bi-Lipschitz parametrizations, we only briefly note that interesting sufficient conditions and counterexamples have been found both in the $2$-dimensional (\cite{BL}, \cite{Fu}, \cite{La}, \cite{MS}, \cite{To1}, \cite{To2})  and higher-dimensional cases (\cite{ABT}, \cite{HeiKe}, \cite{HeiRi}, \cite{HeiSu}, \cite{Se4}). 

Uniformization problems concerning QC and QS maps have received considerable attention in recent years, and they have found significant applications in geometry, complex dynamics, geometric topology and geometric measure theory, among other areas. In particular, several problems in the theory of hyperbolic groups can be interpreted as uniformization problems concerning boundaries of the groups in question, cf. \cite{Bo}, \cite{Bon}, \cite{BK2}, \cite{BonMer}, \cite{HP}, \cite{Kl}. 

The theory of QC and QS maps $f:Y \to \R^n$ can be roughly divided into two parts depending on 
the metric space $Y$. If $Y$ ``has dimension $n$", meaning that $Y$ shares some metric or geometric properties with $\R^n$ and in particular is not a fractal, then analytic methods can be used to study such maps $f$. On the other hand, if $Y$ has fractal-like behavior, then one mainly has to rely on weaker methods. Also, the infinitesimal or analytic definitions of quasiconformality do not give a good theory in this case, and one has to concentrate on QS maps. See \cite{TV} and \cite{Hei} for the basic properties of QS maps in metric spaces, and \cite{BM}, \cite{DT}, \cite{Me2}, \cite{Me3} for results on QS parametrizations of fractal spaces. 

In this paper we consider the first part described above. A theory of QC maps between metric spaces 
equipped with the Hausdorff $Q$-measure has been established by Heinonen and Koskela \cite{HeKo}, based on two assumptions: \emph{Ahlfors regularity} and the \emph{Loewner condition}. The first assumption requires balls $B(x,r)$ in the space to have mass comparable $r^Q$. The second assumption is a certain estimate concerning the $Q$-modulus of path families (see 
Section \ref{defsection}) resembling change of variables by polar coordinates. These assumptions lead to strong results and can often be verified assuming purely geometric conditions on the space. Heinonen and Koskela showed that the Loewner condition is equivalent to a suitable \emph{Poincar\'e inequality}. They also proved that the infinitesimal QC condition and the QS condition are equivalent under the above assumptions, at least locally. 

We now come back to the uniformization problem. After previous results by Semmes \cite{Se1} and David and Semmes \cite{DS}, Bonk and Kleiner \cite{BK} gave a satisfactory answer in the case $Y$ Ahlfors $2$-regular and homeomorphic to $\mathbb{S}^2$. Namely, they proved that under these assumptions $Y$ is QS equivalent to $\mathbb{S}^2$ if and only if $Y$ is \emph{linearly locally contractible}. This is a geometric condition which in particular rules out cusp-like behavior, see Section \ref{quasisym}. This result has been extended in several consequent works cf. \cite{BK2}, \cite{MeWi}, \cite{Wil}, \cite{Wil2}. There the Ahlfors regularity condition is combined with varying geometric conditions on the space $Y$. 

In higher dimensions, the uniformization problem does not have a satisfactory answer even for Ahlfors regular spaces. Examples by Semmes \cite{Se2} show that the result of Bonk and Kleiner mentioned above does not generalize to dimension $3$. Heinonen and Wu \cite{HeWu} and Pankka and Wu \cite{PW} gave further examples of geometrically nice spaces without QS 
parametrizations. 

In this paper we take a slightly different approach to the uniformization problem in dimension two. We would like to find minimal hypotheses under which a result resembling the classical uniformization theorem as much as possible could be proved. This 
means giving up the geometric conditions such as Ahlfors regularity and linear local contractibility, and instead of QS maps seek for parametrizations by conformal or QC maps which do not in general have good global 
properties. 

There are two main reasons for using such an approach. First, while the geometric conditions are good tools to work with, assuming them and Ahlfors regularity in particular is too restrictive in many situations. Secondly, if one knows the existence of QC 
parametrizations in general spaces, then one can try to upgrade their properties using QC invariants together with whatever conditions the underlying spaces satisfy. 

We consider metric spaces $X$ homeomorphic to $\R^2$. Also, we work with the Hausdorff $2$-measure and assume that it 
is locally finite on $X$. This is natural since the Hausdorff measure is related to the metric in $X$, but also to plane topology via coarea estimates and separation properties. This guarantees that QC maps in $X$ are closely related to the metric and topology. 
Under these minimal assumptions, we define conformal and QC maps $f:X \to \R^2$ using the \emph{geometric definition}. This is a standard definition of quasiconformality involving conformal modulus of path families, see Section \ref{defsection}. 

QC maps between general metric spaces are often defined using the \emph{metric definition}, see Section \ref{quasisym}. The advantage of the geometric definition is that it automatically gives a QC invariant that can be used to prove estimates in the presence of geometric or other conditions. This is not true with the metric definition which in general implies very few properties by itself. In the case of Ahlfors regular Loewner spaces the two definitions of quasiconformality coincide (\cite{HeKo}, \cite{Ty2}). See \cite{BKR} and \cite{Willi} for more general results concerning the equivalence between different definitions. 

The uniformization problem now asks for conditions on $X$ under which there exists a QC map $f:X \to \R^2$. It follows from the results mentioned above that such a map exists if $X$ is Ahlfors $2$-regular and linearly locally contractible. However, as discussed above, it is of great interest to consider more general spaces that do not satisfy such strong conditions. One could hope that a QC map always exists under the minimal assumptions that $X$ be homeomorphic to $\R^2$ with locally finite Hausdorff $2$-measure. This is not true, however, as shown in Example \ref{decomp}. 

Our main result, Theorem \ref{main}, gives a necessary and sufficient condition called \emph{reciprocality}: whenever $Q \subset X$ is a topological square, let $M_1$ be the modulus of all paths joining two opposite sides in $Q$, and $M_2$ the modulus of all paths joining the other two sides in $Q$. Then we require that $M_1\cdot M_2$ is bounded from above by $\kappa$ and below by $\kappa^{-1}$, with constant $\kappa$ depending only on $X$. We also assume that the modulus of a point is always zero, in a suitable sense. 

A basic exercise in classical QC theory shows that planar rectangles satisfy the reciprocality condition with constant $1$. Applying the 
Riemann mapping theorem, or arguing directly, one sees that this holds for all Jordan domains in the plane. Then it is easy to deduce that reciprocality is necessary for the existence of a QC map $f: X \to \R^2$. Theorem \ref{main} shows that it is also sufficient. Reciprocality of a general space $X$ implies that $X$ cannot be too ``squeezed" and concentrated too much around a small set of zero Hausdorff measure, cf. Example \ref{decomp}. 

Methods applying reciprocality in connection with quasiconformality have previously been used in Euclidean spaces and also in more general situations, cf. \cite{Ca}, \cite{Sch}, although they usually do not appear explicitly. Indeed, reciprocality is connected to the fact that conjugate functions can be defined for harmonic functions. Also related is the fact that capacities are dual to the moduli 
of separating hypersurfaces, cf. \cite{Ge3}, \cite{Z}, \cite{Der}, \cite{Der2}. In this paper we show that the reciprocality condition can be isolated and applied to prove uniformization results in a very general setting.  

The reciprocality condition is much weaker than Ahlfors regularity. Although it is sometimes difficult to determine whether the condition holds, it can be verified in several important cases. In Theorem \ref{rectifi}, we show that reciprocality holds 
if the measures of balls $B(x,r)$ are bounded from above by a constant times $r^2$, without assuming further geometric conditions on $X$. Consequently, such spaces admit QC parametrizations. In Theorem \ref{mainmini} we show that if a QC 
parametrization exists, it can be always chosen to have dilatation bounded from above by $2$. 

As an application of our results, we can reprove the QS uniformization theorem of Bonk and Kleiner discussed above. Indeed, it follows from the theory of Heinonen and Koskela that QC maps between Ahlfors regular, linearly locally contractible spaces are QS. Now Theorem \ref{rectifi} gives 
a QC map even without the connectivity condition, so under its presence the quasiconformality can be ``upgraded" to 
quasisymmetry.  

\newpage


\subsection{Definitions} \label{defsection}

Throughout this paper, $X$ denotes a metric space homeomorphic to $\rk$. It then follows that the Hausdorff $2$-measure $\hausk$ of every ball 
$B \subset X$ is positive (see Remark \ref{nakuttaja}). In this paper we always assume that $\hausk(B)$ is also finite whenever $\overline{B} \subset X$ is compact. Notice that $X$ is not assumed to be complete or proper. 

\begin{definition}
\label{moddef} Let $\Gamma$ be a family of continuous paths in $X$. The (conformal) modulus of $\Gamma$ is 
$$
\modu(\Gamma)= \inf_{\rho} \int_X \rho^2 \, d\hausk, 
$$
where the infimum is taken over all admissible functions for $\Gamma$, i.e., all non-negative Borel functions $\rho$ satisfying 
$$
\int_{\gamma} \rho \, ds \geq 1 
$$
for all locally rectifiable $\gamma \in \Gamma$. 
\end{definition}

If $E, F, G \subset X$, we denote by $\Delta(E,F;G)$ the family of all continuous paths joining $E$ and $F$ in $G$, 
and $\modu(E,F;G):=\modu\Delta(E,F;G)$. 

\begin{definition}
\label{qcdef}
Let $\Omega' \subset X$ be a domain, and $f:\Omega' \to \Omega \subset \rk$ a homeomorphism. We say that $f$ and $f^{-1}$ are $K$-\emph{quasiconformal}, or $K$-QC (conformal if $K=1$), if 
$$
K^{-1} \modu(\Gamma) \leq \modu(f\Gamma) \leq K \modu(\Gamma)
$$ 
for every path family $\Gamma$ in $\Omega'$.  Here $f\Gamma= \{f \circ \gamma: \, \gamma \in \Gamma \}$. 
\end{definition}


We will abuse terminology by calling an injective map $f$ QC if $f$ is a QC homeomorphism onto its image. 
QC maps between metric spaces are usually defined using the so-called metric definition, see Remark \ref{intti}. It turns out that in the setting of this paper the so-called geometric definition given above is more natural. 

\begin{definition}
\label{rec} 
We say that $X$ is $\kappa$-\emph{reciprocal}, if the conditions $(1)$-$(3)$ hold: if $Q \subset X$ is homeomorphic 
to a closed square, let $\zeta_1,\ldots, \zeta_4$ be the boundary edges in cyclic order. Then the moduli of 
opposite edges satisfy
\begin{eqnarray}
\label{ylaraja}
\modu(\zeta_1,\zeta_3;Q) \cdot \modu(\zeta_2,\zeta_4;Q) &\leq& \kappa, \quad \text{and} \\
\label{alaraja}
\modu(\zeta_1,\zeta_3;Q) \cdot \modu(\zeta_2,\zeta_4;Q)&\geq& \kappa^{-1}. 
\end{eqnarray}  
If $a \in X$ and $X \setminus B(a,R) \neq \emptyset$,  
then 
\begin{equation}
\label{nollamoduli} 
\lim_{r \to 0}\modu(\overline{B}(a,r),X \setminus B(a,R);\overline{B}(a,R)) =0. 
\end{equation}
We say that $X$ is reciprocal if $X$ is $\kappa$-reciprocal for some $\kappa$. 
\end{definition}

It follows from the Riemann mapping theorem, or can be proved directly, that simply connected domains in $\mathbb{R}^2$ are $1$-reciprocal, as well as smooth surfaces. $\mathbb{R}^2$ equipped with a non-Euclidean norm is always $\kappa$-re\-ci\-pro\-cal for some $\kappa$, and $1$-reciprocal if and only if the norm is induced by an inner product. We discuss further examples in the next sections. 

\subsection{Main Results} \label{resultsection}
The main result of this paper reads as follows. 
\begin{theorem} \label{main}
There exists a QC homeomorphism $f:X \to \Omega \subset \R^2$ if and only if $X$ is reciprocal. 
\end{theorem}
There are spaces $X$ for which the conditions of Theorem \ref{main} are not satisfied, see Example \ref{decomp}. 
The necessity of reciprocality for the existence of a QC parametrization follows directly from the $1$-reciprocality of 
Euclidean plane domains and the definition of quasiconformality. Sufficiency is the actual content of Theorem \ref{main}. 

Combining results from Sobolev and Lipschitz analysis in metric spaces, the measurable Riemann mapping theorem, and John's 
theorem on symmetric convex bodies, one can give a universal bound for the QC dilatation in Theorem \ref{main}. 

\begin{theorem} \label{mainmini}
There exists a QC homeomorphism $f:X \to \Omega \subset \R^2$ if and only if there exists a 
$2$-QC homeomorphism $f_0:X \to \Omega \subset \R^2$. If moreover $X \subset R^N$ for some $N \geq 2$, then $2$ can be replaced 
by $1$. 
\end{theorem}


The constant $2$ in Theorem \ref{mainmini} is not best possible. The best constant for the space $X=(\R^2,||\cdot||_{\infty})$ 
is $\pi/2$, see Example \ref{norm}. This suggests that $\pi/2$ may also be the sharp constant in the theorem. See Section \ref{bensajumala} for further discussion. It follows from Theorems \ref{main} and \ref{mainmini} that if $X$ is reciprocal then $X$ is always $4$-reciprocal and if moreover $X \subset \R^N$ then $X$ is $1$-reciprocal. 

Theorems \ref{main} and \ref{mainmini} can be applied to the class of spaces satisfying upper Euclidean mass bounds. 

\begin{theorem}
\label{rectifi}
Suppose there exists $C_U>0$ such that 
\begin{equation}
\label{upperbound}
\hausk(B(x,r)) \leq C_U r^2  
\end{equation}
 for every $x \in X$ and $r>0$. Then $X$ is reciprocal.  
\end{theorem}

The proofs of Theorems \ref{main}, \ref{mainmini} and \ref{rectifi} show that Theorem \ref{rectifi} remains true if \eqref{upperbound} is assumed for balls inside compact subsets $E$ of $X$, such that the constant $C_U$ is allowed to depend on $E$. Several examples of reciprocal spaces can be constructed that do not satisfy \eqref{upperbound} even locally. Theorems \ref{main}, \ref{mainmini} and \ref{rectifi} hold also when $X$ is homeomorphic to the Riemann sphere $\mathbb{S}^2$, with obvious modifications. 

Bonk and Kleiner \cite{BK} gave an excellent characterization for quasispheres among the topological spheres satisfying 
\eqref{upperbound}. Theorems \ref{main} and \ref{rectifi} yield a new proof to their result. 

\begin{corollary}[\cite{BK}, Theorem 1.1]
\label{bonkkleiner}
Assume that $Y$ is homeomorphic to $\mathbb{S}^2$ and satisfies \eqref{upperbound}. Then 
there exists a QS homeomorphism $f:Y \to \mathbb{S}^2$ if and only if $Y$ is linearly locally contractible.  
\end{corollary}

See Section \ref{quasisym} for the definitions of quasisymmetry and linear local contractibility. Again, the actual content 
of Corollary \ref{bonkkleiner} is the existence of the required quasisymmetric map. Corollary \ref{bonkkleiner} is quantitative: $f$ can be chosen to be $\eta$-quasisymmetric with $\eta$ depending only on $C_U$ and the 
linear local contractibility constant. In contrast to Theorem \ref{main}, it is clear that Corollary \ref{bonkkleiner} does not hold 
with a universal quasisymmetry function $\eta$. 

\subsection{Organization of the paper}
In Section \ref{exsection} we give two examples illustrating Theorems \ref{main} and 
\ref{mainmini}. In the first example we construct a surface $X$ that cannot be parametrized by a QC map. This is done 
by fixing a Cantor set of positive Lebesgue measure in $\R^2$, and choosing a continuous weight vanishing on the Cantor set. Taking 
the path metric with respect to this weight yields a non-reciprocal space $X$. In the second example we consider $\R^2$ 
equipped with the $L^{\infty}$-norm, and find the best possible dilatation for QC maps between this space and the Euclidean plane. 

Theorem \ref{main} is proved in Sections \ref{pathsection}--\ref{globalsection}. First, in Section \ref{pathsection} we apply coarea estimates to find paths with positive modulus in $X$. Both the results and the methods in this section are frequently applied in the following sections. 

We construct the map in Theorem \ref{main} in several steps. We first show the existence of a QC map 
in a given topological square $Q \subset X$. We start in Section \ref{minimizer} by defining the real part $u$ of $f$. Applying Heinonen and Koskela's notion of upper gradients, we show that $u$ can be defined as an energy minimizer among functions taking value $0$ on a fixed edge of $\partial Q$ and value $1$ on the opposite edge. In $\R^2$ this would mean finding the harmonic function with minimal energy under such boundary conditions. We also prove a maximum principle for $u$ that later allows us to develop its main properties. The results in this section hold in great generality, and at this point we do not assume any of the reciprocality conditions. 

In Section \ref{continofu} we apply the maximum principle, together with conditions \eqref{alaraja} and \eqref{nollamoduli}, to prove 
continuity of the function $u$ in $Q$. Moreover, in Section \ref{setsofu} we show that under these conditions almost every level set 
of $u$ is a simple curve. This helps us define a conjugate function $v$ for $u$. Indeed, our method for defining $v$ in $\R^2$ 
would simply be integrating $|\nabla u|$ over the level sets of $u$. It turns out that a similar approach works also in our generality, 
although the actual definition is more involved. In Section \ref{conjug} we carry out the construction of $v$ and prove continuity. 
The map $f=(u,v)$ then maps $Q$ onto a rectangle $[0,1] \times [0,M_1]$, where $M_1$ depends on $Q$. 

Once we have constructed the map $f$, we need to show that it is QC. In particular, we need to establish some analytic properties for 
$f$. Section \ref{condo} is the first step in this direction. There we apply a dyadic decomposition of the image to show that $f$ maps 
sets of measure zero to sets of measure zero, and that the change in area is what corresponds to $|\nabla u|^2$ in $\R^2$. 
This leads to a change of variables formula that by itself does not imply quasiconformality of $f$ but plays a role in the proof. The first application of the formula appears in Section \ref{homeosection} where we prove that $f$ is a homeomorphism. 

To prove quasiconformality of $f$, we need to show the validity of the modulus inequalities 
$\modu(\Gamma) \leq K \modu(f\Gamma)$ 
and $\modu(f \Gamma) \leq K\modu(\Gamma)$. These depend on the analytic (Sobolev) regularity of $f$ and $f^{-1}$, respectively. To prove 
the regularity of $f$, we introduce in Section \ref{vari} a modification of conformal modulus, called variational modulus. Although 
the variational modulus is not as easy to work with as the conformal modulus, it has the advantage of being exactly the dual 
of conformal modulus, in a suitable sense. In Section \ref{regoff}, we use this duality together with the reciprocality conditions to 
prove regularity of $f$, and, consequently, the first of the modulus inequalities. It is worth noticing that this is the only step in the proof 
where condition \eqref{ylaraja} is assumed. 

We complete the proof of quasiconformality of $f$ in Section \ref{regoffmiinus1}, by showing regularity of the map $f^{-1}$ and the second modulus inequality. To prove Theorem \ref{main}, we exhaust the space $X$ with squares $Q$ as above, and give normal 
family arguments to show the existence of a QC map from the whole space $X$ as a limit of maps $f$ constructed above; we do this in Section \ref{globalsection}. 

We prove Theorem \ref{mainmini} in Section \ref{bensajumala}. In contrast to other parts of this paper, which are mostly elementary 
and self-contained, here we rely on results from different areas. We apply the differentiability results of Kirchheim \cite{Ki}, the measurable Riemann mapping theorem, and John's theorem on convex bodies to find a QC map in $X$ with small dilatation. 

In Section \ref{massboundsection} we prove Theorem \ref{rectifi} by checking that spaces satisfying \eqref{upperbound} satisfy the reciprocality conditions. In Section \ref{quasisym} we consider quasisymmetric maps and apply Theorems \ref{main} and 
\ref{rectifi} to prove Corollary \ref{bonkkleiner}. Finally, in Section \ref{huri} we briefly discuss the absolute continuity properties of QC 
maps in the current generality, as well as the reciprocality conditions.

\vskip 10pt
\noindent
{\bf Acknowledgement.}
We are grateful to the anonymous referee for the careful reading of the manuscript and the thoughtful comments that lead to several improvements. We thank Mario Bonk, Changyu Guo, Jeff Lindquist, Dimitrios Ntalampekos, Martti Rasimus and Vyron Vellis for their comments and corrections. 
\newpage


\section{Examples} \label{exsection}

We first introduce some basic notation and terminology. If $Y=(Y,d)$ is a metric space, $k \in \{1,2\}$, and $E \subset Y$, the \emph{Hausdorff $k$-measure} $\haus^k(E)$ of $E$ is 
$$
\haus^k(E) = \lim_{\delta \to 0} \inf \Big\{\sum_{j=1}^{\infty} a_k \operatorname{diam}(A_j)^k : \, E \subset \bigcup_{j=1}^{\infty}A_j, \, 
\operatorname{diam}(A_j) < \delta  \Big\}, 
$$
where $a_1=1$ and $a_2=\pi/4$. $\haus^2$ coincides with the Lebesgue measure $|\cdot|$ in $\R^2$. We always assume that $X=(X,d)$ is homeomorphic to 
$\R^2$ and that the Hausdorff $2$-measure of every compact $E \subset X$ is finite. If $x \in X$ and $r>0$, we denote $B(x,r)=\{y \in X: d(x,y)<r\}$, and 
$S(x,r)=\{y \in X: d(x,y)=r\}$. We denote the image of a continuous path $\gamma$ by $|\gamma|$. We call a continuous, injective path $\gamma:[a,b] \to Y$ simple. Moreover, a simple curve is the 
image of a simple path. We call a connected set a domain if it is open,  and a continuum if it is compact. A continuum is non-trivial if it contains more than one point. 

In this section we give the following examples to illustrate the sharpness of our main results. 

\begin{example}
\label{decomp} 
Let $\mathcal{C} \subset \R^2$ be a Cantor set defined as follows: At the first step we divide 
the unit square $Q_0=[0,1]^2$ to four congruent subsquares with disjoint interiors. 
For each of these subsquares $\hat{Q}$, we choose a square $Q$ with the same center as $\hat{Q}$ 
and with sidelength $(1-a_1)/2$. Then we remove everything except for the four squares $Q$. 

At the second step, we repeat the process with the unit square replaced by each of the squares 
$Q$ remaining after the first step.  Continuing this way, after $n$ steps we have $4^n$ squares remaining, each of sidelength 
$$
2^{-n}\prod_{j=1}^n (1-a_j). 
$$
Taking the intersection of all the remaining squares gives the Cantor set $\mathcal{C}$. 
We choose the sequence $(a_j)$ such that $\mathcal{C}$ has positive Lebesgue measure. 

Now let $0 \leq \omega \leq 1$ be a continuous function in 
$\R^2$ such that $\omega(x)=0$ if and only if $x \in \mathcal{C}$ and $\omega=1$ near infinity. Define $d=d_{\omega}$ in $\R^2$ by setting 
$$
d(x,y)= \inf \int_{\gamma_{x,y}} \omega \, ds,  
$$
where the infimum is taken over all rectifiable paths in $\R^2$ joining $x$ and $y$. We check that 
$d$ is a metric in $\R^2$. First, since $0 \leq \omega \leq 1$, $d(x,y)$ is finite for every $x$ and $y$. 
Also, the triangle inequality follows directly from the definition. It remains to show that $d(x,y)>0$ if $x \neq y$. 

Let $x \neq y$ be points in $\R^2$. If $x \notin \mathcal{C}$ or $y \notin \mathcal{C}$, then there exists an $\epsilon >0$ such that $\omega \geq \epsilon$ on some disc $B(x,\delta)$ or $B(y,\delta)$, with $0< \delta < |x-y|/2$. Consequently, 
$$
d(x,y)\geq \epsilon \delta. 
$$
If both $x,y \in \mathcal{C}$, then there is some step $n$ remaining square $Q$ in the construction of $\mathcal{C}$ such that $x \in Q$ but $y \notin Q$. It follows that there is a slightly larger square $\hat{Q}$ with same center as $Q$, such that 
$$
(\hat{Q} \setminus Q) \cap \mathcal{C} = \emptyset. 
$$
Since $\omega$ is continuous and positive on $\hat{Q} \setminus Q$, there exists $\epsilon>0$ such that for every path $\gamma$ joining the two boundary components, 
$$
\int_\gamma \omega \, ds \geq \epsilon. 
$$
Consequently, 
$$
d(x,y) \geq \epsilon >0. 
$$ 
We conclude that $d$ is a metric on $\R^2$. Moreover, $d(x,y)\leq |x-y|$ for all $x,y \in \R^2$, so 
the identity map $I:\R^2 \to (\R^2,d)$ is a homeomorphism and $d$ is a length metric. Recalling that $\omega=1$ near infinity and applying the Hopf-Rinow theorem, we conclude that $(\R^2,d)$ is a geodesic metric space. 

We have shown that $(\R^2,d)$ is a geodesic metric space homeomorphic to $\R^2$. The $1$-Lipschitz continuity of the identity map $I$ also shows that the Hausdorff $2$-measure 
$\hausk_d$ in the space $(\R^2,d)$ is locally finite. In fact, 
$$
\int_A g \, d\hausk_d =\int_A g \omega^2 \, dx
$$ 
for all Borel measurable $A \subset \R^2$ and $g \geq 0$. Here and in what follows $dx$ refers to integration with respect to Lebesgue measure in $\mathbb{R}^2$. We show that there are no QC maps from $(\R^2,d)$ into $\R^2$ by proving that $(\R^2,d)$ is not reciprocal. Indeed, let $M>0$. Since $\mathcal{C}$ has density points, we can choose a square $Q$ such that $|Q \setminus \mathcal{C}| \leq M^{-1}|Q|$. Without loss of generality, $Q=[0,1]^2$. In $(\R^2,d)$, we give a lower bound for the modulus $\modu(\Gamma_1)$ of all paths joining the vertical edges of $Q$ in $Q$. Namely, if $\rho$ is an admissible function, we integrate $\rho$ over the horizontal segments of height $t$; 
$$
1 \leq \int_{[0,1]\times \{t\}} \rho \omega \, ds. 
$$
We then integrate over $t$ and apply H\"older's inequality to get 
\begin{eqnarray*}
1 &\leq& \int_{Q} \rho \omega \, dx = \int_{Q \setminus \mathcal{C}} \rho \omega \, dx \leq 
|Q \setminus \mathcal{C}|^{1/2} \Big( \int_{Q} \rho^2 \omega^2 \, dx\Big)^{1/2} \\ 
&\leq& M^{-1/2} \Big( \int_{Q} \rho^2 \, d\hausk_d \Big)^{1/2}.  
\end{eqnarray*}
Minimizing over $\rho$, we get $\modu(\Gamma_1) \geq M$. Similarly, if $\Gamma_2$ is the family of paths joining the horizontal edges, we get $\modu(\Gamma_2) \geq M$. Therefore, 
$$
\modu(\Gamma_1) \cdot \modu(\Gamma_2) \geq M^2. 
$$
Letting $M \to \infty$, we conclude that $(\R^2,d)$ is not reciprocal. 
\end{example}


\begin{example}
\label{norm}
We equip $\R^2$ with the $\ell^{\infty}$-norm $||(x_1,x_2)||_{\infty}=\max\{|x_1|,|x_2|\}$. If $\hausk$ denotes Hausdorff measure on 
$(\R^2,||\cdot||_\infty)$ and $|\cdot|$ the Lebesgue $2$-measure, then $\hausk(A)=\pi|A|/4$ for every Borel set $A \subset \R^2$, see \cite[Lemma 6]{Ki}. We now claim that the identity map $f : (\R^2,||\cdot||_\infty) \to \R^2$ is $\pi/2$-QC, where it is understood that the image is equipped with Euclidean norm $||\cdot||$. 
We have 
$$
L_f=L_f(x):=\limsup_{r \to 0} \sup_{||x-y||_{\infty} \leq r}\frac{||f(x)-f(y)||}{||x-y||_{\infty}}=\sqrt{2} \text{ and } J_f=J_f(x)= \frac{4}{\pi} 
$$
for every $x \in \R^2$, for the maximal stretching $L_f$ and volume derivative $J_f$ of $f$. A standard change of variables argument now shows that 
\begin{equation}
\label{triv}
\modu(\Gamma) \leq \frac{L_f^2}{J_f} \modu(f \Gamma) =\frac{\pi}{2} \modu(f \Gamma)
\end{equation}
whenever $\Gamma$ is a path family in $(\R^2, ||\cdot||_{\infty})$. Indeed, if $\rho$ is an admissible function for $f\Gamma$, then 
the function $L_f (\rho\circ f)$ is admissible for $\Gamma$, and moreover 
$$
\int L_f^2 (\rho \circ f)^2 \, \dhausk \leq \frac{L_f^2}{J_f} \int \rho^2 \, dx. 
$$
Since this holds true for all admissible functions $\rho$, \eqref{triv} follows. Similarly, we see that 
$\modu(f\Gamma) \leq \frac{4}{\pi} \modu(\Gamma)$, 
so 
$$
\frac{2}{\pi} \modu(\Gamma) \leq \modu(f \Gamma) \leq \frac{4}{\pi} \modu(\Gamma) 
$$
for every path family $\Gamma$. We conclude that $f$ is $\pi/2$-QC. 

We next show that there are no $K$-QC maps with $K < \pi/2$. Denote by $\phi$ the counterclockwise rotation of $\R^2$ by $\pi/4$, and let $Q=\phi([0,1]^2)$. We give a lower bound for the modulus 
$\modu(\Gamma_1)$ in 
$(\R^2, ||\cdot||_{\infty})$ of the family of paths $\Gamma_1$ joining 
$\phi(\{0\}\times [0,1])$ and $\phi(\{1\}\times [0,1])$ in $Q$. Let $\rho$ be an admissible function. Then 
$$
1 \leq \int_{\phi(\{t\}\times [0,1])} \rho \, d\haus^1. 
$$
Integrating over $t$ and applying the co-area formula \cite[Theorem 9.4]{AK} (with the co-area factor of $L^{\infty}$), we see that 
\begin{eqnarray*}
1 &\leq& \int_0^1  \int_{\phi(\{t\}\times [0,1])} \rho \, d\haus^1 \, dt = \frac{4}{\sqrt{2}\pi} \int_{Q} \rho \, d\hausk \\
&\leq&  \frac{4}{\sqrt{2}\pi}\Big( \int_{Q} \rho^2 \, d\hausk\Big)^{1/2} \hausk(Q)^{1/2}= \frac{\sqrt{2}}{\sqrt{\pi}} \Big( \int_{Q} \rho^2 \, d\hausk\Big)^{1/2}. 
\end{eqnarray*}
Minimizing over $\rho$ gives $\modu(\Gamma_1) \geq \pi/2$. Similarly, if $\Gamma_2$ is the family of paths joining the other two sides of $Q$, then $\modu(\Gamma_2) \geq \pi/2$. Hence, if $f:Q \to \R^2$ is $K$-QC, then using the $1$-reciprocality of $\R^2$ we get 
$$
\frac{\pi^2}{4} \leq \modu(\Gamma_1) \cdot \modu(\Gamma_2) \leq K^2 \modu(f\Gamma_1) 
\cdot \modu(f \Gamma_2) = K^2, 
$$
i.e., $K \geq \pi /2$. 
\end{example}


\section{Existence of rectifiable paths} \label{pathsection}
Recall that we assume that $X$ is homeomorphic to $\R^2$ and has locally finite $2$-measure. In this section we show that under these mild conditions one can find large families of rectifiable paths in $X$ (see \cite{Se3} for much deeper results along these lines). 
We will later prove qualitative estimates, such as continuity, using such families. We will frequently use the following results. These are \cite[Proposition 15.1]{Se3} and \cite[Proposition 3.1.5]{AT} (slightly modified), respectively. 
 
\begin{proposition}
\label{findcurves}
Let $x,y \in X$ be given, $x \neq y$. Suppose that $E \subset X$ is a continuum with $\haus^1(E) < \infty$ and $x,y \in E$. Then there is an $L>0$, $L \leq \haus^1(E)$, and an injective $1$-Lipschitz 
mapping $\gamma:[0,L] \to X$ such that $\gamma(t)\in E$ for all $t$, $\gamma(0)=x$, $\gamma(L)=y$, and 
$\haus^1(\gamma(F))=\haus^1(F)$ for all measurable sets 
$F \subset [0,L]$. 
\end{proposition}

\begin{proposition}
\label{coarea}
Let $A \subset X$ be Borel measurable. If $m:X \to \mathbb{R}$ is $L$-Lipschitz and $g:A \to [0,\infty]$ Borel measurable, 
then 
$$
\int_{\mathbb{R}} \int_{A \cap m^{-1}(t)} g(s) \, d\haus^1(s)\, dt \leq  \frac{4 L}{\pi}  \int_A g(x) \, d\hausk(x). 
$$
\end{proposition}

We next show that the family of paths joining two continua always has positive modulus. We need the following topological lemma, cf. 
\cite[IV Theorem 26]{Moo}. 

\begin{lemma}
\label{separates}
Let $U,V \subset \R^2$ be disjoint continua, and suppose that a compact set $F \subset \R^2 \setminus (U \cup V)$ separates $U$ and $V$ in $\R^2$. Then $F$ contains a continuum $G$ separating $U$ and $V$ in $\R^2$.  
\end{lemma}

\begin{remark}
\label{nakuttaja}
Let $x \in X$. Then, by Lemma \ref{separates}, there exists $r_0>0$ such that $\haus^1(S(x,r))>0$ for almost every $0<r<r_0$. Applying Proposition \ref{coarea} with 
$m=d(\cdot,x)$, we see that $\hausk(B)>0$ for every ball $B \subset X$. See \cite{HW} for further connections between topological dimension and Hausdorff measures. 
\end{remark}

\begin{proposition}
\label{continuum}
Let $\alpha$ and $\beta$ be two nontrivial continua in a topological closed square $Q \subset X$. Then 
$$
\modu(\alpha,\beta;Q)>0. 
$$
\end{proposition}

\begin{proof}
We first assume that both $\alpha$ and $\beta$ lie in the interior of $Q$, henceforth denoted by $\operatorname{int} Q$. 
Fix points $a \in \alpha$ and $b \in \beta$, and a continuous path $\eta:[0,1]\to \operatorname{int} Q$ joining $a$ and $b$. Let 
$m(x)=\operatorname{dist}(x,|\eta|)$, where $|\eta|$ is the image of $\eta$. Then $m$ is $1$-Lipschitz. Moreover, there exists $\epsilon >0$ such that 
$F_t:=m^{-1}(t) \subset Q$ and $F_t$ separates $\partial Q$ and $|\eta|$ for every $0<t<\epsilon$. Applying Proposition \ref{coarea} to $m$ and $g=1$, we see that $\haus^1(F_t)$ is finite for almost every $t$. Therefore, since $Q$ is homeomorphic to a planar square, Lemma \ref{separates} shows that $F_t$ contains a continuum $G_t$ which also separates. Since $\alpha$ and $\beta$ are nontrivial continua, there exists $0<\epsilon'<\epsilon$ such that for every $0<t<\epsilon'$ there are points $a_t \in \alpha \cap G_t$ and $b_t \in \beta \cap G_t$. Applying Proposition \ref{findcurves}, we find for almost every $0<t<\epsilon'$ a rectifiable, injective path 
$\gamma_t$ joining $a_t$ and $b_t$ in $G_t$. Denote by $\Gamma$ the family of all such $\gamma_t$. Then 
$$
\Gamma \subset \Delta(\alpha, \beta;Q). 
$$

Now let $g: Q \to [0,\infty]$ be admissible for $\Gamma$. Then, applying Proposition \ref{coarea} and H\"older's inequality, we have 
$$
\epsilon' \leq \int_0^{\epsilon'} \int_{\gamma_t} g \, ds \, dt \leq \frac{4}{\pi} \int_Q g(x)\, d\hausk(x)
\leq \frac{4}{\pi}\hausk(Q)^{1/2} \Big(\int_Q g(x)^2 \, d\hausk(x) \Big)^{1/2}. 
$$
Since the estimate holds for all admissible functions $g$, we conclude that 
$$
\modu(\alpha,\beta;Q) \geq \modu(\Gamma) \geq \Big(\frac{\pi \epsilon'}{4\haus^2(Q)^{1/2}}\Big)^2 >0. 
$$
If $\alpha$ touches the boundary of $Q$ but $\beta$ does not, then we modify the proof as follows: if $\alpha$ contains a point in $\operatorname{int}Q$, then we can find a subcontinuum in the interior and the proof above applies. Otherwise, $\alpha$ contains a topological line segment $I \subset \partial Q$. Now, we can choose the point $a$ to be the center of $I$, and we can choose 
$\eta:[0,1] \to \operatorname{int}Q \cup\{a\}$ such that the $\epsilon$-neighborhood of $|\eta|$ does not intersect $\partial Q \setminus I$ 
when $\epsilon$ is small enough. Now the proof above applies. We proceed similarly if both $\alpha$ and $\beta$ touch the boundary of $Q$. 
\end{proof}





\section{Energy minimizer $u$ on a topological square}
\label{minimizer}

In this section we define a suitable energy minimizing, ``harmonic" function $u$ in our general setting. We also develop some basic properties for $u$. Later, we define a ``conjugate function" $v$ of $u$, and show that, under our reciprocality assumption, the resulting map $f=(u,v)$ is QC. 

Let $\Omega \subset X$. Recall that a Borel function $g \geq 0$ is an \emph{upper gradient} of a function $u$ in $\Omega$, if 
\begin{equation}
\label{upperg}
|u(y)-u(x)| \leq  \int_{\gamma} g \, ds
\end{equation}
for every $x$ and $y \in \Omega$ and every locally rectifiable path $\gamma$ joining $x$ and $y$ in $\Omega$. Here by joining we mean that both $x$ and $y \in |\gamma|$. Also, we agree that the left term in \eqref{upperg} equals $\infty$ if $|u(x)|=\infty$ or $|u(y)|=\infty$. We say that $g$ is a weak upper gradient of $u$, if there exists a path family $\Gamma_0$ with modulus zero such that \eqref{upperg} holds for every $x$ and $y$ and every $\gamma \notin \Gamma_0$. Similarly, we say that a property holds for almost every path in a path family $\Gamma$, if there exists $\Gamma_0 \subset \Gamma$ of modulus zero such that the property holds for all 
$\gamma \in \Gamma \setminus \Gamma_0$. Furthermore, we say that a Borel function $\rho$ is weakly admissible for $\Gamma$, if 
the integral of $\rho$ over $\gamma$ is at least $1$ for almost every $\gamma \in \Gamma$. 

We now construct the function $u$. Let $Q \subset X$ be homeomorphic to a closed square in $\rk$, and $\zeta_1,\ldots, \zeta_4$ the boundary edges as in \eqref{ylaraja} and \eqref{alaraja}. At this point we do not assume any of the reciprocality conditions. 
We consider the modulus 
$$
M_1=\modu(\zeta_1,\zeta_3; Q). 
$$
A standard method now shows that there exists a weakly admissible function realizing $M_1$. More precisely, let $(\rho^j)$ be a minimizing sequence of admissible functions. Then, after passing to a subsequence, $\rho^j$ converges to $\rho \in L^2(Q)$ weakly in $L^2$. Moreover, by Mazur's lemma \cite[Page 19]{HKSTbook}, there exists a sequence $(\rho_k)$ of convex combinations of the $\rho^j$; 
$$
\rho_k = \sum_{j=1}^{N(k)} \lambda_j \rho^j, \quad \sum_{j=1}^{N(k)} \lambda_j=1, \quad \lambda_j \geq 0, 
$$
such that $\rho_k \to \rho$ strongly in $L^2$. 

Now it follows by Fuglede's lemma \cite[Page 131]{HKSTbook} that 
\begin{equation}
\label{fuglede}
\int_{\gamma}  \rho_k \, ds \to \int_{\gamma} \rho \, ds <\infty
\end{equation}
for almost every $\gamma$ in $Q$. In particular, 
\begin{equation}
\label{kellonelja}
\int_{\gamma} \rho \, ds \geq 1 
\end{equation}
for almost every $\gamma$ joining $\zeta_1$ and $\zeta_3$ in $Q$, so 
$$
\int_Q \rho^2 \, d\hausk =M_1. 
$$


We would now like to define the function $u$ by integrating the minimizing function $\rho$ over paths. This is possible although some technicalities arise. Denote by $\Gamma_0$ the family of paths in $Q$ that have a subpath for which \eqref{fuglede} does not hold. Then $\modu(\Gamma_0)=0$. 

We will be working with paths that do not belong to the exceptional family $\Gamma_0$. For instance, we show in Lemma \ref{kipu} that the upper gradient inequality \eqref{upperg} holds for the function $u$, weak upper gradient $\rho$, and all paths $\gamma$ outside $\Gamma_0$. Since $\rho$ is integrable on such paths $\gamma$, it follows that $u$ will be absolutely continuous there. The subpath property in the definition of $\Gamma_0$ is given to guarantee that paths outside $\Gamma_0$ can be concatenated succesfully. 

Define $u$ as follows: For $x \in Q$, first assume there exists 
$$
\gamma \in \Delta(\zeta_1,\zeta_3;Q) \setminus \Gamma_0
$$ 
such that some subpath $\gamma_x$ of $\gamma$ joins $\zeta_1$ and $x$. Then define  
\begin{equation}
\label{tuska}
u(x)= \inf_{\gamma_x} \int_{\gamma_x} \rho \, ds, 
\end{equation}
where the infimum is taken over all possible $\gamma$ and $\gamma_x$. If $u(x)$ cannot be defined this way for $x \in Q$, let 
$$
u(x)=\liminf_{y \in E, y \to x} u(y), 
$$
where $E$ is the set of points $y$ for which $u(y)$ is already defined. 


\begin{lemma}
\label{udefined}
The function $u:Q \to [0,\infty]$ is well-defined. 
\end{lemma}
\begin{proof}
We have to show that for every $x \in Q$ and every $\epsilon >0$ there exists $y \in B(x,\epsilon)$ such that $u(y)$ is defined by 
\eqref{tuska}. First notice that $B(x,\epsilon)\cap Q$ contains a non-trivial continuum $G$. Therefore, by Proposition \ref{continuum}, there exists a family $\Gamma$ of paths joining $\zeta_1$ and $G$ in $Q$, such that $\modu(\Gamma) >0$. Then Fuglede's lemma guarantees that for some $\gamma_a \in \Gamma$ \eqref{fuglede} holds for all subpaths of $\gamma_a$. Let $F$ be a non-trivial 
component of $|\gamma_a| \cap \overline{B}(x,\epsilon)$. Then, applying Proposition \ref{continuum} and Fuglede's lemma again gives a path $\gamma_b$ joining $F$ and $\zeta_3$ such that \eqref{fuglede} holds for all subpaths of $\gamma_b$. Now we can define 
$\gamma$ by concatenating a suitable subpath of $\gamma_a$ with $\gamma_b$. Then $\gamma$ 
joins $\zeta_1$ and $\zeta_3$ and $|\gamma|$ intersects $B(x,\epsilon)$. Moreover, all subpaths of $\gamma$ satisfy 
\eqref{fuglede}. Therefore, $u(y)$ can be defined by \eqref{tuska} for all $y \in |\gamma|$. 
\end{proof}

\begin{lemma}
\label{uudefined}
Let $\gamma' \notin \Gamma_0$ be a rectifiable path in $Q$. Then for every $x \in |\gamma'|$ there exists a path 
$\gamma \notin \Gamma_0$ joining $\zeta_1$ and $\zeta_3$ such that $x \in |\gamma|$. In particular, $u(x)$ is defined by 
\eqref{tuska}. 
\end{lemma}
\begin{proof}
The argument is similar to the previous lemma. Proposition \ref{continuum} gives path families $\Gamma_1$ and $\Gamma_2$ of positive modulus joining $\zeta_1$ and $|\gamma'|$, and $\zeta_2$ and $|\gamma'|$, respectively. Moreover, Fuglede's lemma gives paths 
$\gamma_a \in \Gamma_1 \setminus \Gamma_0$ and $\gamma_b \in \Gamma_2 \setminus \Gamma_0$. Now $\gamma$ can be defined by concatenating $\gamma_a$, $\gamma_b$, and a suitable subpath of $\gamma'$. That $\gamma \notin \Gamma_0$ follows because $\gamma_a$, $\gamma_b$ and $\gamma'$ all have the same property. 
\end{proof}


\begin{lemma}
\label{kipu}
The function $\rho$ is a weak upper gradient of $u$ in $Q$. In fact, \eqref{upperg} holds (with $\rho$) for all rectifiable paths $\gamma \notin \Gamma_0$.  
\end{lemma} 

\begin{proof}
Let $x$ and $y \in Q$. Since we only require the upper gradient inequality outside a set of modulus zero, we may assume that there is a rectifiable path $\gamma \notin \Gamma_0$ joining $x$ and $y$ in $Q$ . Then, by Lemma \ref{uudefined}, $u(x)$ and $u(y)$ are defined by \eqref{tuska}. We may assume that $u(y)>u(x)$. Then, by the definition of $u$,  
$$
u(y) \leq \inf_{\gamma_x}  \Big(\int_{\gamma} \rho \, ds + \int_{\gamma_x} \rho \, ds  \Big) \leq \int_{\gamma} \rho \, ds + u(x), 
$$
where the infimum is taken as in \eqref{tuska}. 
\end{proof}


We need the following auxiliary result to prove further properties for $u$. 

\begin{lemma}
\label{ykskolme}
Let $L>0$ and $\epsilon>0$, and denote the interior of $\{u>L\}$ by $E$. If $\eta$ is a rectifiable path in $\{u\geq L+\epsilon\}$, then 
$$
\haus^1(|\eta| \setminus E)=0. 
$$
\end{lemma}

\begin{proof}
First, let $y \in |\eta|$ and $0<r < \operatorname{diam}|\eta|/2$. Assume that $u(z) \leq L$ at some $z \in B(y,r/2)$. Then, by the definition of $u$, there exists a curve $\alpha$ joining 
$B(y,r/2)$ and $Q \setminus B(y,2r)$ such that $u \leq L$ everywhere on $\alpha$. Moreover, 
Lemma \ref{separates} implies that for every $r/2<s<r$ some continuum $C(s) \subset S(y,s)$ intersects both $|\eta|$ and $\alpha$. That is, there are $a_s, b_s \in C(s)$ such that 
$u(b_s)-u(a_s) \geq \epsilon$. Therefore, Proposition \ref{continuum} and Lemma \ref{findcurves} 
show that for almost every such $s$ there are rectifiable curves in $C(s)$ joining $a_s$ and $b_s$. 
Furthermore, for almost every such $s$, the upper gradient inequality gives 
$$
\epsilon \leq \int_{S(y,s)} \rho \, d\haus^1. 
$$
Integrating from $r/2$ to $r$ and applying Proposition \ref{coarea}, we have 
\begin{equation}
\label{ironm}
r \epsilon \leq  \frac{8}{\pi} \int_{B(y,r)} \rho \, d\hausk. 
\end{equation}

Now for every $\delta >0$ and every $y \in |\eta| \setminus E$ there exists $r_y< \delta$ 
such that \eqref{ironm} holds for every $r<r_y$. By the $5r$-covering lemma, among all such balls $B(y,r)$ we can find a finite or countable subcollection $\{B_j=B(y_j,r_j)\}$ such that the balls 
$B_j$ are pairwise disjoint and 
$$
 \Big( |\eta| \setminus E \Big) \subset \bigcup_j B(y_j,5r_j). 
$$
Applying \eqref{ironm} in all $B_j$ and summing the estimates gives  
$$
\haus_{5\delta}^1(|\eta| \setminus E) \leq 10 \sum_j r_j \leq \frac{80}{\epsilon \pi} \sum_j \int_{B_j} \rho \, d\hausk. 
$$
By the disjointness of the balls $B_j$, the sum on the right can be estimated from above by 
$$
\int_{N_{5\delta}(|\eta|)} \rho \, d\hausk, 
$$
where $N_{5\delta}(|\eta|)$ is the closed $5\delta$-neighborhood of $|\eta|$. Since $\eta$ is rectifiable, 
this integral converges to zero when $\delta \to 0$. Combining the estimates gives the claim. 
\end{proof}

\begin{lemma} 
\label{nollayksi}
We have $0\leq u(x) \leq 1$ for every $x \in Q$. 
\end{lemma}
\begin{proof}
Suppose to the contrary that $u(x_0) \geq 1+3\epsilon$ for some $\epsilon>0$ and $x_0 \in Q$. Then, by the definition of $u$, we find a curve $\alpha$ in the interior of $Q$ such that $u \geq 1+ 2\epsilon$ everywhere on $\alpha$. Proposition \ref{continuum} shows that 
\begin{equation}
\label{pepti}
\modu(\zeta_1,\alpha;Q) >0. 
\end{equation}
Given $\eta':[0,1] \to Q$ in $\Delta(\zeta_1,\alpha;Q) \setminus \Gamma_0$ such that 
$\eta'(0) \in \alpha$, let  
$$
t_0=\inf\{t: \, u(\eta'(t)) \leq 1+ \epsilon \} \quad \text{and}  \quad \eta=\eta'|[0,t_0].
$$ 
By the upper gradient inequality and the absolute continuity of $u$ on $\eta'$, 
$0<t_0<1$ and 
$$
\int_{\eta} \rho \, ds \geq \epsilon. 
$$
Combining with Lemma \ref{ykskolme}, we conclude that $\rho \chi_{E}/\epsilon$ is weakly admissible for $\Delta(\zeta_1,\alpha;Q)$, where $E$ is the interior of $\{u > 1\}$. 
In particular 
\begin{equation}
\label{ouwie}
\int_E \rho^2 \, d\hausk >0 
\end{equation}
by \eqref{pepti}. On the other hand, the function 
$$
\rho_0:= \rho \chi_{\{Q \setminus E\}} 
$$
is weakly admissible for $\Delta(\zeta_1,\zeta_3;Q)$ by the definition of $u$. But now \eqref{ouwie} gives 
$$
\int_Q \rho_0^2 \, d\hausk < \int_Q \rho^2 \, d\hausk. 
$$
This contradicts the minimizing property of $\rho$. The proof is complete. 
\end{proof}


We next establish a maximum principle. Let $\Omega \subset X$ be open. We denote 
$$
\partial_* \Omega=(\partial \Omega \cap Q) \cup (\overline{\Omega} \cap (\zeta_1 \cup \zeta_3)).   
$$

\begin{lemma}
\label{epatoivo}
Let $\Omega \subset X$ be open. Then 
$$
\sup_{x\in \Omega \cap Q} u(x)= \sup_{y \in \partial_* \Omega} u(y) \quad \text{and} \quad 
\inf_{x\in \Omega \cap Q} u(x)= \inf_{y \in \partial_* \Omega} u(y). 
$$
\end{lemma}

\begin{proof}
To prove the second equality, let $x_0 \in \Omega \cap Q$ and $u(x_0)=m$. Then there is 
$x \in \Omega \cap Q$ such that $u(x) \leq m +\epsilon$ is defined by \eqref{tuska}. Moreover, there exists a path $\gamma_x$ joining $\zeta_1$ and $x$ such that $u \leq m+2\epsilon$ on $|\gamma_x|$. But $|\gamma_x|$ must intersect $\partial_*\Omega$. The second equality follows. 

The proof of the first equality is similar to the proof of Lemma \ref{nollayksi}. Let 
$M=\sup_{y \in \partial_* \Omega} u(y) \leq 1$, and suppose there is $\delta >0$ such that $u(x)\geq M + 2 \delta$ for some 
$x \in \Omega \cap Q$. Then, by the definition of $u$, we can choose a curve $\alpha$ in $\Omega \cap Q$ such that $u > M+ \delta$ on $\alpha$. Applying Proposition \ref{continuum}, we see that $\modu(\alpha,\zeta_1;Q)>0$. Arguing as in the proof of Lemma \ref{nollayksi}, we see that 
$$
\int_{E} \rho^2 \, d\hausk >0, 
$$
where $E$ is the interior of $\Omega \cap \{u > M\}$. On the other hand, $\rho \chi_{ Q \setminus E}$ is weakly admissible for the minimizing problem, because $u \leq M$ on $\partial_* \Omega$. This contradicts the minimality of $\rho$. 
\end{proof}


\section{Continuity of $u$} \label{continofu}
Let $u$ be the function defined in Section \ref{minimizer}. 
In this section we show that $u$ is continuous, assuming conditions \eqref{alaraja} and \eqref{nollamoduli}. 

\begin{theorem}
\label{continuity}
Suppose that $X$ satisfies \eqref{alaraja} and \eqref{nollamoduli}. Then $u:Q \to [0,1]$ is continuous. Moreover, $u=0$ on the boundary component $\zeta_1$ and $u=1$ on the boundary component $\zeta_3$. 
\end{theorem}

The rest of this section is devoted to the proof of Theorem \ref{continuity}. We say that $D \subset X$ is a \emph{half-annulus}, if $D$ is homeomorphic to 
$$
 \{(s,\varphi): \, 1 \leq s \leq 2, \, 0 \leq  \varphi \leq \pi \} \subset \R^2,
$$
defined in polar coordinates. The boundary of $D$ consists of \emph{inner and outer circles}, and the two \emph{flat components}.  

\begin{lemma}
\label{dualnollamoduli}
Suppose $X$ satisfies \eqref{alaraja} and \eqref{nollamoduli}, and fix $x \in X$ and $R>0$. Moreover, let $r<R/2$ and assume that $D$ is a half-annulus with inner circle $T_r \subset B(x,r)$, outer circle $T_R \subset X \setminus B(x,R)$, and flat components $I$ and $J$. Then 
\begin{equation}
\label{needforcont}
\modu(I,J;D) \geq \Phi(r) \to \infty \quad \text{as } r \to 0, 
\end{equation}
where $\Phi$ depends on $r$, $R$, $\kappa$ and $x$. 
\end{lemma}

\begin{proof}
By condition \eqref{nollamoduli}, 
$$
\modu(T_R,T_r;D) \leq \modu(S(x,R),S(x,r);\overline{B}(x,R)) \leq \epsilon(x,r,R) \to 0  
$$
as $r \to 0$. On the other hand, by \eqref{alaraja}, 
$$
\modu(T_R,T_r;D) \cdot \modu(I,J;D) \geq \kappa^{-1}. 
$$
The lemma follows by combining the estimates. 
\end{proof}

\begin{remark}
For future reference, we note that Lemma \ref{dualnollamoduli} holds if the assumptions are replaced by assumption \eqref{upperbound}. 
See Section \ref{massboundsection} for further details. 
\end{remark}

\begin{lemma}
\label{boundcont1}
Suppose $X$ satisfies \eqref{needforcont}. Then $u$ is continuous in $(\zeta_1 \cup \zeta_2 \cup \zeta_4) \setminus \zeta_3$. 
\end{lemma}
\begin{proof}
Without loss of generality, $x \in (\zeta_1 \cup \zeta_2) \setminus (\zeta_3 \cup \zeta_4)$, otherwise we replace $\zeta_2$ by $\zeta_4$. We choose a topological closed disk $D' \subset X$ such that $x \in \operatorname{int}D'$. Moreover, we require that $D'$ does not intersect $\zeta_4$ or $\zeta_3$. Then the boundary circle $T'$ of $D'$ satisfies 
$T' \subset X \setminus  B(x,R)$ for some $R>0$. Let $r<R/2$, and choose another topological disk $D'' \subset B(x,r)$ containing $x$, with boundary circle $T''$. Then the two boundary circles and $\partial Q$ bound a half-annulus $A$ in $Q$. Denote by $T_R$ and $T_r$ the circular boundary components (the restrictions of $T'$ and $T''$, respectively), and by $I, J \subset \partial Q$ the flat components. Moreover, let 
$\gamma \in \Delta(I,J;A)$. Then, $\gamma$ and $\partial Q$ bound a domain $\Omega_{\gamma}$ in $Q$ containing $B(x,r) \cap Q$. Moreover, since $\gamma$ does not intersect $\zeta_4$, Lemma \ref{epatoivo} shows that 
\begin{eqnarray*}
\sup_{ y \in B(x,r) \cap Q} u(y) & \leq & \sup_{y \in \Omega_{\gamma}} u(y) \leq \sup_{y \in |\gamma| \cup \zeta_1} u(y),  \\
\inf_{ y \in B(x,r) \cap Q} u(y) & \geq & \inf_{y \in \Omega_{\gamma}} u(y) \geq \inf_{y \in |\gamma| \cup \zeta_1} u(y).  
\end{eqnarray*}
But $u=0$ on $\zeta_1$, so the first estimate above holds true without $\zeta_1$ on the last term. Also, if $\zeta_1$ intersects the boundary of $\Omega_{\gamma}$, then $\gamma$ intersects $\zeta_1$. So also the second estimate holds without $\zeta_1$. In other words, 
$$
\delta_r := \sup_{y,z \in B(x,r)} |u(y)-u(z)| \leq \sup_{y,z \in |\gamma|} |u(y)-u(z)|. 
$$
Since $\rho$ is a weak upper gradient of $u$, it follows that 
$$
\delta_r \leq \int_{\gamma} \rho \, ds
$$
for almost every $\gamma \in \Delta(I,J;A)$. Consequently, we have 
$$
\modu(I,J;A) \leq \delta_r^{-2} \int_Q \rho^2 \, d\hausk. 
$$
On the other hand, by \eqref{needforcont} we have 
$$
\modu(I,J;A) \geq \Phi(r) \to \infty \quad \text{as } r \to 0. 
$$
We conclude that $\delta_r \to 0$ as $r \to 0$, showing that $u$ is continuous at $x$. 
\end{proof}


\begin{lemma}
\label{oree}
Suppose $X$ satisfies \eqref{needforcont}. Then $u$ is continuous in $\zeta_3$ and equals $1$ there. 
\end{lemma}
\begin{proof}
Let $x \in \zeta_3$. In view of Lemma \ref{nollayksi}, it suffices to show that 
\begin{equation}
\label{frappe}
\liminf_{y \to x} u(y) \geq 1. 
\end{equation}
Without loss of generality, $x \notin \zeta_4$. If \eqref{frappe} does not hold, there exists $\epsilon>0$ and a sequence of points $x_j \to x$ in $Q$ such that $u(x_j)\leq 1-3\epsilon$ for 
every $j$. 

We choose a topological closed disk $D'$ such that $x \in \operatorname{int}D'$. Moreover, we require that $D'$ does not intersect $\zeta_4$. Then the boundary circle $T'$ of $D'$ satisfies 
$T' \subset X \setminus  B(x,R)$ for some $R>0$. Let $r<R/2$, and choose another topological disk $D'' \subset B(x,r)$ containing $x$, with boundary circle $T''$ such that $x \in D''$. 

By the definition of $u$ and Lemma \ref{findcurves}, there is a simple path $\eta \notin \Gamma_0$ joining $\zeta_1$ 
and $D''$ in $Q$ such that 
\begin{equation}
\label{lampsi}
\int_{\eta} \rho \, ds \leq 1-2 \epsilon. 
\end{equation}

We may assume that $\eta$ does not intersect $\zeta_3$, since otherwise \eqref{lampsi} violates \eqref{kellonelja}. Now $|\eta|$, $T'$, $T''$ and $\zeta_3$ bound a half-annulus $A$ with flat boundary components $I \subset |\eta|$ 
and $J \subset \zeta_3$. We claim that when $r$ is small enough there exists a path $\gamma \in \Delta(I,J;A) \setminus \Gamma_0$ satisfying 
\begin{equation}
\label{faarao}
\int_{\gamma} \rho \, ds < \epsilon. 
\end{equation}
Indeed, otherwise we would have 
$$
\modu(I,J;A) \leq \epsilon^{-2} \int_Q \rho^2 \, d\hausk. 
$$
But this contradicts \eqref{needforcont} when $r$ is small enough, so \eqref{faarao} holds. 

Concatenating $\gamma$ with a subpath of $\eta$ and applying \eqref{lampsi} and \eqref{faarao} now gives a path $\gamma' \notin \Gamma_0$ joining 
$\zeta_1$ and $\zeta_3$ in $Q$ such that 
$$
\int_{\gamma'} \rho \, ds \leq 1-\epsilon. 
$$
This contradicts \eqref{kellonelja}, and so \eqref{frappe} holds. The proof is complete. 
\end{proof}
Continuity of $u$ in the interior of $Q$ is proved using the methods above. However, the proof is more technical and we need an auxiliary lemma. 

\begin{lemma}
\label{sennus}
Suppose $X$ satisfies \eqref{needforcont}, and fix $x \in \operatorname{int}Q$. 
Moreover, suppose there is a simple, rectifiable path $\gamma:[0,1] \to \operatorname{int}Q$, $\gamma \notin \Gamma_0$ 
such that $\gamma(c)=x$ for some $0<c<1$. Then $u$ is continuous at $x$. 
\end{lemma}

\begin{proof}
Fix $\epsilon>0$. We choose a topological closed disk $D' \subset \operatorname{int} Q$ such that $x \in \operatorname{int}D'$. Then the boundary circle $T'$ of $D'$ satisfies $T' \subset X \setminus  B(x,R)$ for some $R>0$. Mapping $|\gamma|$ to a segment in $\R^2$ if necessary, we can choose $D'$ so that $|\gamma|$ separates $D'$ into two components $D'_1$ and $D'_2$. Also, 
since $\int_{\gamma} \rho \, ds < \infty$, we can choose $D'$ small enough such that 
$$
\int_\gamma \rho \chi_{D'} \, ds < \epsilon. 
$$
It then follows from the definition of $u$ and Lemma \ref{uudefined} that 
\begin{equation}
\label{hjk}
|u(y)-u(x)| \leq \epsilon \quad \text{for every } y \in |\gamma| \cap D'. 
\end{equation}
Let $r<R/2$, and choose another topological disk $D'' \subset B(x,r)$ with boundary circle $T''$, such that $x \in D''$. Again, we can choose $D''$ such that $|\gamma|$ separates $D''$ into two components $D''_1 \subset D'_1$ and $D''_2 \subset D'_2$. We control the oscillation of $u$ in $D''_1$ and $D''_2$ separately. Since the estimates are identical, we only consider the case $D''_1$.  

Now $D'_1$ contains a half-annulus $A_1$ bounded by $T'$, $T''$, and $|\gamma|$. The flat boundary components 
are $I,J \subset |\gamma|$. Then, if 
$$
u(z) \geq u(x)+2 \epsilon \quad \text{or} \quad u(z) \leq u(x)-2 \epsilon
$$ 
for some $z \in D''_1$, then Lemma \ref{epatoivo} and \eqref{hjk} yield 
$$
\sup_{a,b \in |\eta|} |u(a)-u(b)| \geq \epsilon 
$$
for every $\eta \in \Delta(I,J;A_1)$. Since $\rho$ is a weak upper gradient of $u$, we moreover have 
$$
\int_{\eta} \rho \, ds \geq \epsilon 
$$
for $\eta \notin \Gamma_0$, so 
$$
\modu(I,J;A_1) \leq \epsilon^{-2} \int_Q \rho^2 \, d\hausk. 
$$
This contradicts \eqref{needforcont}. Applying the same argument to $A_2$, we conclude that 
$$
\sup_{z \in D''} |u(z)-u(x)| \leq \epsilon(r) \to 0 \quad \text{as } r \to 0. 
$$
We conclude that $u$ is continuous at $x$. 
\end{proof}

\begin{lemma}
\label{nytkosejo}
Suppose $X$ satisfies \eqref{needforcont}. Then $u$ is continuous in $\operatorname{int}Q$. 
\end{lemma}
\begin{proof}
Fix $x \in \operatorname{int} Q$ and let $\epsilon>0$. Choose a topological disk $D \subset Q$ containing $x$, with boundary circle $T'$. Moreover, let $r>0$ and let $D'\subset B(x,r) \subset D$ be another disk containing $x$. Denote the boundary circle of $D'$ by $T'_r$. Then, by the definition of $u$, there exists 
a rectifiable path $\gamma' \notin \Gamma_0$ joining $\zeta_1$ and $D'$ such that 
\begin{equation}
\label{summert}
u(y) \leq u(x) + \epsilon
\end{equation} 
for every $y \in |\gamma'|$. Moreover, by Lemma \ref{findcurves}, we find a simple path $\gamma \notin \Gamma_0$ joining $\zeta_1$ and $D'$ with $|\gamma| \subset |\gamma'|$. By Lemma \ref{sennus}, $u$ is continuous on $|\gamma|$. We would like to repeat the argument used in the previous lemmas, applying the maximum principle, \eqref{summert} and Lemma \ref{epatoivo} in 
the domain bounded by $T'$, $T'_r$, and $|\gamma|$. But this domain is not a half-annulus, so Lemma \ref{epatoivo} does not apply directly. 

To correct this, notice that by the uniform continuity of $u$ on $|\gamma|$ and \eqref{summert} there is a neighborhood $V$ of $|\gamma| \cap D$ such that 
\begin{equation}
\label{laxe}
u(y) \leq u(x)+2\epsilon
\end{equation}
for all $y \in V$. We choose simple paths $I$ and $J$ in $V$ connecting $T'$ and $T'_r$ such that $|\gamma|$ separates $I$ and $J$ in $V$. 

Now $I$, $J$, $T'$ and $T'_r$ bound a half-annulus $A$, with flat boundary components $I$ and $J$. 
As before, the maximum principle and \eqref{laxe} imply that if $u(y) \geq u(x)+3 \epsilon$
for some $y \in D'$, then 
\begin{equation}
\label{armen}
\sup_{a,b \in |\eta|} |u(a)-u(b)| \geq \epsilon 
\end{equation}
for every $\eta$ joining $I$ and $J$ in $A$. Applying \eqref{armen} to all such paths, together with the weak upper gradient property of $\rho$, gives 
$$
\modu(I,J;A) \leq \epsilon^{-2} \int_Q \rho^2 \, d\hausk. 
$$
This contradicts \eqref{needforcont} when $r$ is small enough. We conclude that $u$ is continuous in $x$. 
\end{proof}

\begin{proof}[Proof of Theorem \ref{continuity}]
Combine Lemmas \ref{nollayksi}, \ref{boundcont1}, \ref{oree} and \ref{nytkosejo}. 
\end{proof}

\section{Level sets of $u$} \label{setsofu}

In this section we examine the properties of the level sets of $u$, and in particular show that almost every level set is a 
rectifiable curve. This helps us define the conjugate function $v$ in the next section. 

\begin{proposition}
\label{levelsets}
Suppose that $X$ satisfies \eqref{alaraja} and the minimizer $u$ satisfies the conclusions of Theorem \ref{continuity}. Then for 
$\haus^1$-almost every $t$ the level set $u^{-1}(t)$ is a simple rectifiable curve
$|\gamma_t|$ joining $\zeta_2$ and $\zeta_4$. 
\end{proposition}

We will later show that $u$ is the real part of a homeomorphism, so in particular $u^{-1}(t)$ is a simple curve for all $0<t<1$. The rest of this section is devoted to the proof of Proposition \ref{levelsets}. 

Recall our notation $M_1=\modu(\zeta_1,\zeta_3;Q)$. Moreover, for $0\leq s <t \leq 1$ we denote $A_{s,t}=\{x \in Q: \, s < u(x) < t \}$ and  
$$
M_{s,t}:= \modu(u^{-1}(s),u^{-1}(t);\overline{A}_{s,t}). 
$$

\begin{lemma}
\label{dividingmod}
Suppose $0\leq s <t \leq 1$, and that $u$ satisfies the conclusions of Theorem \ref{continuity}. Then 
\begin{equation}
\label{jeehee}
M_{s,t}= (t-s)^{-2} \int_{A_{s,t}} \rho^2 \, d\hausk =(t-s)^{-1}M_1. 
\end{equation}
\end{lemma}

\begin{proof}
First, we have 
$$
M_{s,t} \leq (t-s)^{-2} \int_{A_{s,t}} \rho^2 \, d\hausk, 
$$
since $(t-s)^{-1}\rho$ is weakly admissible. The reverse inequality also holds, since if there was an admissible function $g$ such that 
$$
\int_{A_{s,t}} g^2 \, d\hausk < (t-s)^{-2}\int_{A_{s,t}} \rho^2 d \hausk, 
$$
then 
$$
\rho'=\rho \chi_{\overline{A}_{0,s} \cup \overline{A}_{t,1}}+ (t-s) g \chi_{A_{s,t}} 
$$
would be weakly admissible for $\Delta(\zeta_1,\zeta_3;Q)$ (because $u=0$ in $\zeta_1$ and $u=1$ in $\zeta_3$), and 
$$
\int_Q (\rho')^2 \, d\hausk < \int_Q \rho^2 \, d\hausk.  
$$
This contradicts the minimizing property of $\rho$. Therefore, the first equality in \eqref{jeehee} holds. To prove the second equality, we denote  
$$
I_{s,t}:= \int_{A_{s,t}} \rho^2 \, d\hausk. 
$$
Let $\delta >0$, and 
$$
\rho_{\delta} = \frac{(1+\delta)\rho \chi_{A_{s,t}}+ \rho \chi_{Q \setminus A_{s,t}} }{1+\delta(t-s)}. 
$$
Then $\rho_{\delta}$ is weakly admissible for $\Delta(\zeta_1,\zeta_3;Q)$, and 
$$
\int_Q \rho_{\delta}^2 \, d\hausk = \frac{(1+\delta)^2I_{s,t} + M_1-I_{s,t}}{(1+\delta(t-s))^2}. 
$$
If $I_{s,t} < (t-s)M_1$, then the right term is strictly smaller than $M_1$ when $\delta>0$ is small enough. This contradicts the minimizing property of $\rho$. Similarly, if $I_{s,t}>(t-s)M_1$, we get a contradiction by the above argument, replacing $A_{s,t}$ with $Q \setminus A_{s,t}$. 
\end{proof}


\begin{lemma}
\label{simpcon}
Suppose $0<s<t<1$, and that $u$ satisfies the conclusions of Theorem \ref{continuity}. Then $A_{s,t}$ and $u^{-1}(t)$ are connected and simply connected sets connecting $\zeta_2$ and $\zeta_4$ in $Q$. 
Moreover, the sets $\zeta_2 \cap A_{s,t}$, $\zeta_2 \cap u^{-1}(t)$, $\zeta_4 \cap A_{s,t}$ and $\zeta_4 \cap u^{-1}(t)$ are all connected. 
\end{lemma}

\begin{proof}
First, if there is a simple loop $\gamma$ not contractible in $A_{s,t}$, then $\gamma$ bounds a domain $V \subset Q$ containing points outside $A_{s,t}$. But then $\partial V \subset A_{s,t}$. This violates the maximum principle, Lemma \ref{epatoivo}, and so $A_{s,t}$ must be simply connected. The same argument shows that $u^{-1}(t)$ is simply connected. 

Next, suppose $W$ is a connected component of $A_{s,t}$. We claim that $W$ has to intersect 
both $\zeta_2$ and $\zeta_4$. Notice that by the maximum principle, $W$ has to intersect either 
$\zeta_2$ or $\zeta_4$. We lose no generality by assuming that $W$ intersects $\zeta_2$. To show that $W$ also intersects $\zeta_4$, suppose to the contrary that this was not the case. 

Then, by Lemma \ref{separates}, there is a continuum $Y \subset \partial_* W$ separating $W$ and 
$\zeta_4$, where $\partial_*$ is as in Lemma \ref{epatoivo}. Now, if $s<u(x)<t$ at some point $x \in Y$, 
there is a neighborhood $B$ of $x$ such that $s<u<t$ everywhere on $B$. This contradicts the definition of $W$. Therefore, $u$ only takes values $s$ and $t$ on $Y$. But $Y$ is connected, so 
$u$ is constant on $Y$. On the other hand, $Y$ and $\zeta_2$ bound a domain in $Q$ that includes  $W$, and the maximum principle implies that $u$ equals either $t$ or $s$ everywhere in this domain 
This is a contradiction, since $W \subset A_{s,t}$ belongs to this domain. We conclude that $W$ intersects $\zeta_4$. 

Now let $V_1$ and $V_2$ be disjoint connected components of $A_{s,t}$. Then, since both separate $\zeta_1$ and $\zeta_3$, there exists $x \in Q \setminus A_{s,t}$ such that $A_{s,t}$ separates $x$ from both $\zeta_1$ and $\zeta_3$. This contradicts the maximum principle, Lemma \ref{epatoivo}. We conclude that $A_{s,t}$ is connected. To show that $u^{-1}(t)$ is connected and connects $\zeta_2$ and $\zeta_4$, it suffices to notice that the same holds for $\overline{A}_{t-1/j,t+1/j}$ and express $u^{-1}(t)$ as the intersection. 

The remaining claims can be proved by applying the maximum principle as in the previous paragraphs. We leave the details to the reader. 
\end{proof}


To prove Proposition \ref{levelsets}, we recall the compactness property of a family of paths with bounded length, and lower semicontinuity of path length under uniform convergence. The first property follows from the Arzela-Ascoli theorem, while the second property is a simple consequence of the definition of path length. 
\begin{lemma}
\label{lowersemi}
Let $\gamma_j:[0,1] \to Q$, $j \in \mathbb{N}$, be rectifiable paths with 
$$
A= \liminf_{j \to \infty} \ell(\gamma_j) < \infty.
$$
Then the paths $\gamma_j$ can be reparametrized so that the sequence of the re\-pa\-ra\-met\-rized paths has a subsequence converging uniformly to a rectifiable path $\gamma:[0,1] \to Q$ with $\ell(\gamma) \leq A$. 
\end{lemma}


The following differentiation result will be frequently applied, see \cite[Theorem 3.22] {Fol} for the proof. 
Suppose $A \subset Q$ is a Borel set. Moreover, suppose $\phi:A \to [0,\infty]$ is Borel measurable and integrable, and $\psi: A \to \mathbb{R}$ Borel measurable. Define 
$$
\varphi(B)=\int_{\psi^{-1}(B)} \phi \, d\hausk 
$$ 
for Borel sets $B \subset \mathbb{R}$. We say that $\varphi'(t)$ is the \emph{differential} of $\varphi$ at $t \in \mathbb{R}$, if 
$$
\varphi'(t)=\lim_{j \to \infty} \frac{\varphi((a_j,b_j))}{|b_j-a_j|}
$$
whenever $t \in (a_j,b_j)$ and $|b_j-a_j| \to 0$. 

\begin{lemma}
\label{volder}
Suppose $\varphi$ is defined as above. Then the differential $\varphi'(t)< \infty$ exists for almost every $t \in \R$ 
and defines a measurable function such that  
\begin{equation}
\label{roha}
\int_B \varphi'(t) \, dt \leq \varphi(B) 
\end{equation}
for all Borel sets $B \subset \R$. If moreover $\varphi(B)=0$ whenever $\haus^1(B)=0$, then equality holds in \eqref{roha}.

\end{lemma}




Towards the proof of Proposition \ref{levelsets}, we first show that almost every level set of $u$ 
has finite $1$-measure and contains a rectifiable path as in the statement of the proposition. 

\begin{lemma}
\label{mauko}
Suppose that $X$ satisfies \eqref{alaraja} and the minimizer $u$ satisfies the conclusions of Theorem \ref{continuity}. Then for $\haus^1$-almost every $t$ the level set $u^{-1}(t)$ has finite $\haus^1$-measure and contains a simple rectifiable curve $|\gamma_t|$ joining $\zeta_2$ and $\zeta_4$. 
\end{lemma}

\begin{proof}
We apply Lemma \ref{volder} with $\psi=u$, $\phi=1$, and choose $0<t<1$ such that $\varphi'(t)$ exists. Then $\hausk(u^{-1}(t))=0$, so $u^{-1}(t)$ does not have interior points. 

Let $h>0$ such that $[t-h,t+h]\subset (0,1)$. By Lemma \ref{simpcon}, $A_{t-h,t-h/2}$ and $A_{t-h/4,t-h/8}$ contain simple paths $\alpha$ and $\alpha'$, respectively, both joining $\zeta_2$ and $\zeta_4$. Let $D_{h}$ be the Jordan domain bounded by 
$\alpha$, $\alpha'$, $\beta$ $\beta'$, where $\beta$ is a subpath of $\zeta_2$ and $\beta'$ is a subpath of $\zeta_4$. Then 
\begin{equation}
\label{zeppe}
A_{t-h/2,t-h/4} \subset D_{h} \subset A_{t-h,t-h/8}. 
\end{equation}
Hence, by Lemma \ref{dividingmod}, 
$$
\modu(\alpha,\alpha'; D_{h}) \leq M_{t-h/2,t-h/4}=4h^{-1}M_1.  
$$
Here we use notation $M_{s,t}$ introduced before Lemma \ref{dividingmod}. Combining with \eqref{alaraja}, we have 
$$
\kappa^{-1} \leq \modu(\alpha,\alpha'; D_{h})\cdot \modu(\beta,\beta'; D_{h}) \leq 4h^{-1}M_1 \cdot \modu(\beta,\beta'; D_{h}), 
$$
i.e., 
$$
\modu(\beta,\beta'; D_{h}) \geq \frac{h}{4\kappa M_1}. 
$$ 
On the other hand, by \eqref{zeppe},  
$$
\modu(\beta,\beta'; D_{h}) \leq \ell^{-2}_h \hausk(A_{t-h,t+h})= \ell^{-2}_h \varphi((t-h,t+h)), 
$$
where $\ell_h$ is the length of a shortest path $\gamma_h$ joining $\beta$ and $\beta'$ in 
$\overline{D}_h$. 
Notice that $\gamma_h$ is simple, since otherwise we could find a shorter path inside 
$|\gamma_h|$ with the same property. 

Combining the estimates, we have 
\begin{equation}
\label{saalb}
\ell_h^2 \leq 4\kappa h^{-1}\varphi((t-h,t+h))M_1. 
\end{equation}
We take a sequence $h_j \to 0$. Then, by \eqref{saalb} and our choice of $t$, 
$$
\liminf_{j \to \infty} \ell_{h_j}^2 \leq 8\kappa \varphi'(t) M_1< \infty.  
$$
Hence, by Lemma \ref{lowersemi}, there is a subsequence of the simple paths $(\gamma_{h_j})$ converging uniformly to a rectifiable path $\tilde{\gamma}_{t,-}$. Moreover, by Lemma \ref{findcurves}, $|\tilde{\gamma}_{t,-}|$ contains a simple rectifiable path $\gamma_t$ joining $\zeta_2$ and $\zeta_4$ in $u^{-1}(t)$. This proves the second claim in the lemma. 

We found the path $\tilde{\gamma}_{t,-}$ as a limit of paths converging ``from left". With 
the same argument, replacing $t-h/q$ by $t+h/q$ everywhere, we get a sequence of simple rectifiable paths converging uniformly to a rectifiable path 
$\tilde{\gamma}_{t,+}$. Thus $\tilde{\gamma}_{t,+}$ is a limit of paths converging ``from right". 
Notice that both $|\tilde{\gamma}_{t,-}|$ and $|\tilde{\gamma}_{t,+}|$ 
are subsets of $u^{-1}(t)$. The first claim in the lemma follows if we can show that 
\begin{equation}
\label{zandig}
u^{-1}(t)=|\tilde{\gamma}_{t,-}| \cup |\tilde{\gamma}_{t,+}|. 
\end{equation}

Let $(\gamma_{k}^-)$ and $(\gamma_{k}^+)$ be the sequences of simple paths constructed above, such that 
$\gamma_{k}^- \to \tilde{\gamma}_{t,-}$ and $\gamma_{k}^+ \to \tilde{\gamma}_{t,+}$ uniformly 
as $k \to \infty$, and let $\Omega_k$ be the domain bounded by $|\gamma_k^-|$, $|\gamma_k^+|$, $\zeta_2$ and $\zeta_4$. Then $u^{-1}(t) = \cap_k \overline{\Omega}_k$. Since $u^{-1}(t)$ does not have interior points, it follows that for every $x \in u^{-1}(t)$ there is a sequence $(x_k^-)$ such that 
$x_k^- \in |\gamma_k^-|$ for every $k$ and $x_k^- \to x$, or a sequence $(x_k^+)$ such that 
$x_k^+ \in |\gamma_k^+|$ for every $k$ and $x_k^+ \to x$, or both. Thus, by the uniform convergence 
of the paths $\gamma_{k}^-$ and $\gamma_{k}^+$, $x \in |\tilde{\gamma}_{t,-}|$ or 
$x \in |\tilde{\gamma}_{t,+}|$. We conclude that \eqref{zandig} holds. The proof is complete. 
\end{proof}

\begin{proof}[Proof of Proposition \ref{levelsets}]
Again, we apply Lemma \ref{volder} with $\psi=u$, $\phi=1$, and choose $0<t<1$ such that $\varphi'(t)$ exists and the claims of Lemma \ref{mauko} hold. So $u^{-1}(t)$ contains a simple rectifiable path $\gamma_t$. We need to show that $u^{-1}(t)$ does not contain points outside $|\gamma_t|$. 

To prove this, it is convenient to use Euclidean coordinates. In other words, we now think of $d$ as a metric in $\R^2$. Then we may assume that $Q=[-1,1]^2$ and moreover that 
$|\gamma_t|=\{0\}\times [-1,1]$. Suppose there is a point $a \in u^{-1}(t) \setminus |\gamma_t|$. Then we may assume that 
$a \in \operatorname{int}Q$ removing, if necessary, at most countably many values of $t$ for which $u^{-1}(t)$ contains a non-trivial subcontinuum of $\partial Q$. 

Now we may assume that $a=(-1/2,0)$. Since $\haus^1(u^{-1}(t))<\infty$ by Lemma \ref{mauko},  Proposition \ref{findcurves} allows us to find a point $b \in |\gamma_t|$ and a simple path $\eta$, 
$|\eta| \subset u^{-1}(t)$, joining $a$ and $b$ in $Q$. Without loss of generality, $|\eta|=[-1/2,0] \times \{0\}$. 

Next let $I$ be the line segment $\{-1/4\} \times [-1,1]$. Then, when $h$ is small enough and $u^{-1}(t-h)$ contains a simple path $\gamma_{t-h}$, this path together with $I$ bounds a simply connected domain $U$, $a \in \partial U$, whose boundary 
is the union of a subcurve $J_1$ of $|\gamma_{t-h}|$, the subsegment 
$J'_0=[-1/2,-1/4]\times \{0\}$ of $|\eta|$, and two subsegments $J_2$ and $J_3$ of $I$. 

Now we slightly modify $U$ in order to have a Jordan domain $V \subset U$ to which condition \eqref{alaraja} can be applied. Since $u$ is continuous and $u=t$ on $J_0'$, we can choose a simple path $\eta'$ in $U$, depending on $h$,  close to $J'_0$ as follows: 
$u \geq t-h/2$ on $J_0=|\eta'|$ and $J_0$, $J_1$, $J_2$ and $J_3$ bound a Jordan domain $V$ such that $a \in \partial V$. What is important to 
us is that $J_0$ can be chosen so that there exists a constant $c>0$ not depending on $h$ such that whenever $\gamma$ is a path 
connecting $J_2$ and $J_3$ in $\overline{V}$, then $\ell(\gamma) \geq c$. Also, since $V \subset A_{t-h,t}$ and 
$$
\hausk(A_{t-h,t}) \leq 4 h\varphi'(t) 
$$ 
for $h$ small enough, we have 
\begin{equation}
\label{tura}
\modu(J_2,J_3;V) \leq 4c^{-2} h \varphi'(t)=4Ah, 
\end{equation}
where $A$ does not depend on $h$. Applying \eqref{alaraja} and \eqref{tura} shows that 
\begin{equation}
\label{broots}
\modu(J_0,J_1;V) \geq \frac{1}{4\kappa Ah}. 
\end{equation}
On the other hand, since $u=t-h$ on $J_1$ and $u \geq t-h/2$ on $J_0$, the function $2h^{-1}\rho$ is weakly admissible for 
$\Delta(J_0,J_1;V)$. Thus, by \eqref{broots} we have 
\begin{equation}
\label{real}
\frac{h}{16A \kappa} \leq \int_{V} \rho^2 \, d\hausk. 
\end{equation}

Notice that there is $\epsilon>0$ not depending on $h$ such that $\operatorname{dist}(I,|\gamma_t|) \geq \epsilon$. 
We claim that the function 
$$
\rho'=\rho \chi_{Q \setminus V} + \epsilon^{-1} h \chi_{A_{t-h,t}}
$$
is weakly admissible for $\Delta(\zeta_1,\zeta_3;Q)$. To see this, let $\gamma:[0,1] \to Q$ be a rectifiable path, 
$\gamma \notin \Gamma_0$, such that 
$\gamma(0) \in \zeta_1$ and $\gamma(1) \in \zeta_3$. Denote by $0<T<1$ the largest number such that 
$u(\gamma(T))=t-h$. 

Then, if $\gamma(S) \notin V$ for every $S >T$, we have   
$$
\int_{\gamma} \rho \chi_{Q \setminus V} ds \geq 1.  
$$
On the other hand, if $\gamma(S) \in V$ for some $S>T$, then  
$$
\int_{\gamma} \rho \chi_{Q \setminus V} ds \geq 1-h, 
$$
but also 
$$
\int_{\gamma}  \epsilon^{-1} h \chi_{A_{t-h,t}} ds \geq h 
$$
since a subpath of $\gamma$ joins $I$ and $\gamma_t$ in $A_{t-h,t}$. We conclude that $\rho'$ is indeed weakly admissible. 

Now, by \eqref{real} and H\"older's inequality, 
\begin{eqnarray*}
\int_Q (\rho')^2 \, d\hausk &=& 
\int_{Q\setminus  V} \rho^2 \, d\hausk + \epsilon^{-2} h^2 \hausk(A_{t-h,t}) +  
2\epsilon^{-1} h \int_Q \chi_{A_{t-h,t} \setminus V} \rho \, d\hausk \\
&\leq& -\frac{h}{16A\kappa} + \int_Q \rho^2 \, d\hausk + 4\epsilon^{-2} h^3 \varphi'(t) + 2\epsilon^{-1}M_1^{1/2}\varphi'(t)^{1/2}h^{3/2}.  
\end{eqnarray*}
We conclude that when $h$ is small enough, 
$$
\int_Q (\rho')^2 \, d\hausk < \int_Q \rho^2 \, d\hausk. 
$$
This contradicts the minimizing property of $\rho$. The proof is complete. 
\end{proof}

\section{Conjugate function $v$} \label{conjug}
In this section we construct a ``conjugate function" $v$ for our minimizing function $u$ and prove continuity. Then $f=(u,v)$ is the desired QC map in $Q$ if $X$ is reciprocal; this will be shown in the next sections. We note that the conjugate function is easier to find if $X$ is $1$-reciprocal. Indeed, if we construct $v$ as $u$ but replacing $\zeta_1$ and $\zeta_3$ with $\zeta_2$ and $\zeta_4$, respectively, then $f=(u,v):\operatorname{int} Q \to (0,1)^2$ is a conformal homeomorphism. In the general case of $\kappa$-reciprocal $X$ we have to work more to find $v$. The idea behind the construction is that $v$ should be defined integrating the minimizer $\rho$ over the level sets of $u$ in a suitable way.   



Recall the notation 
$$
A_{s,t}= \{ x \in Q: \, s <u(x)  < t\}. 
$$

\begin{lemma}
\label{techlevel}
Suppose $0 < t < 1$, and that the minimizer $u$ satisfies the conclusions of Theorem \ref{continuity}. Then for every $\epsilon>0$ there exists $h>0$ such that $A_{t-h,t}$ is contained in the $\epsilon$-neighborhood $N_{\epsilon}(u^{-1}(t))$ of $u^{-1}(t)$. 
\end{lemma}
\begin{proof}
If the claim is not true, then for some $0<t<1$ and $\epsilon >0$, $F_h=\overline{A}_{t-h,t} \setminus N_{\epsilon}(u^{-1}(t))$ is non-empty for all $h$. But the sets $F_h$ are nested and compact, and 
$$
\bigcap_{h>0} F_h= u^{-1}(t)  \setminus N_{\epsilon}(u^{-1}(t)) = \emptyset
$$ 
by the continuity of $u$. This is a contradiction since the intersection of the sets $F_h$ cannot be empty. 
\end{proof}


We denote 
$$
\mathbb{F}=\{0<t<1: \,  u^{-1}(t) \text{ is a simple curve } |\gamma_t| \}, 
$$
where $\gamma_t:[0,1] \to Q$, $\gamma_t(0) \in \zeta_2$, $\gamma_t(1) \in \zeta_4$, and 
$$
U=\{x \in Q: \, u(x)=t \text{ for some } t \in \mathbb{F}\}. 
$$
Recall that, by Proposition \ref{levelsets}, the set $\mathbb{F}$ has full $1$-measure in $(0,1)$. Now let $x =\gamma_t (T) \in U$ and 
denote 
$$
N_{\epsilon,T} (\gamma_t)= Q \cap \Big( \cup_{0 \leq s \leq T} B(\gamma_t(s),\epsilon)\Big). 
$$ 
We define $v: \overline{U} \to [0,M_1]$ such that 
\begin{equation}
\label{mahdo}
v(x)=\lim_{\epsilon \to 0} \liminf_{h \to 0} \int_{N_{\epsilon,T}(\gamma_t)\cap A_{t-h,t}} \frac{\rho^2}{h} \, \dhaus^2  
\end{equation}
when $x \in U$ and 
$$
v(x)=\liminf_{y \in U, y \to x} v(y) 
$$
when $x \in \overline{U} \setminus U$. That $0 \leq v \leq M_1$ follows from Lemma \ref{dividingmod} and Lemma \ref{techlevel}. 
Also, notice that $\zeta_1 \cup \zeta_3 \subset \overline{U}$. 


The following proposition allows us to extend $v$ to all of $Q$. Recall the notation $\partial_*$ from Lemma \ref{epatoivo}. 

\begin{proposition}
\label{constbound}
Suppose $X$ satisfies \eqref{alaraja} and \eqref{nollamoduli}. Let $V$ be a connected component of $Q \setminus \overline{U}$. Then $v$ is constant on $\partial_* V$. 
\end{proposition} 

\begin{proof}
Let $V$ be a connected component of $Q \setminus \overline{U}$. We will argue by contradiction, assuming that $v$ is not constant on $\partial_* V$. First, notice that there exists $0<t_0<1$ such that $V \subset u^{-1}(t_0)$, by Proposition \ref{levelsets}. Therefore, Lemma \ref{dividingmod} implies that $\rho(x)=0$ for almost every $x \in V$. 

Now let $a,b \in \partial_*V$, such that $v(a) \leq v(b) -8\mu$, $\mu >0$. Then, by the definition of $v$ we can find a radius $r>0$ and 
points $x \in B(a,r) \cap U$, $y \in B(b,r) \cap U$ such that $v(x)\leq v(y)-7\mu$. Let $x=\gamma_t(T)$ and $y=\gamma_s(S)$. Then we furthermore find $0<\epsilon<r$ and $h>0$ such that 
\begin{eqnarray}
\label{kaa1} v(x) &\geq&  \int_{N_{\epsilon,T}(\gamma_t)\cap A_{t-h,t}} \frac{\rho^2}{h} \, \dhaus^2  -\mu, \\
\label{kaa2} v(y) &\leq&  \int_{N_{\epsilon,S}(\gamma_s)\cap A_{s-h,s}} \frac{\rho^2}{h} \, \dhaus^2 +\mu. 
\end{eqnarray}
Moreover, by Lemma \ref{techlevel} we may choose $h$ so that $A_{t-h,t} \subset N_{\epsilon}(\gamma_t)$ and 
$A_{s-h,s} \subset N_{\epsilon}(\gamma_s)$. In particular, \eqref{kaa2} and Lemma \ref{dividingmod} give 
\begin{equation}
\label{kaa3}  \int_{A_{s-h,s} \setminus N_{\epsilon,S}(\gamma_s)} \frac{\rho^2}{h} \, \dhaus^2 \leq M_1 -v(y) +\mu \leq M_1-v(x)- 6\mu.
\end{equation}
We denote 
$$
\Omega=(N_{\epsilon,T}(\gamma_t)\cap A_{t-h,t}) \cup (A_{s-h,s} \setminus N_{\epsilon,S}(\gamma_s)). 
$$
Then combining \eqref{kaa1} and \eqref{kaa3} gives 
\begin{equation}
\label{karaka}
\int_\Omega \frac{\rho^2}{h} \, d\hausk \leq M_1-5 \mu. 
\end{equation}
Now choose points $p \in B(a,r) \cap V$, $q \in B(b,r) \cap V$, and a simple path $\eta$ joining $p$ and $q$ in $V$. Moreover, choose $\delta>0$ small enough such that $N_{\delta}(|\eta|) \subset V$. By condition \eqref{nollamoduli} we can choose $r>0$ small enough to begin with so that there are 
Borel functions $g_1$ and $g_2$ such that 
$g_1$ is admissible for $\Delta(\zeta_1,B(a,2r);Q)$, $g_2$ is admissible for $\Delta(\zeta_1,B(b,2r);Q)$, and 
\begin{equation}
\label{riisi}
\int_Q  (g_1 +g_2)^2 \, d\hausk \leq \frac{\mu^2}{M_1}. 
\end{equation}
We now define a function $g$ by setting 
$$
g= h^{-1} \rho \chi_{\Omega}+g_1+g_2+\delta^{-1}\chi_V. 
$$
This definition of $g$ is motivated by the fact that $g$ is weakly admissible for $\Delta(\zeta_1,\zeta_3;Q)$. Indeed, let 
$\gamma \in \Delta(\zeta_1,\zeta_3;Q) \setminus \Gamma_0$. If $|\gamma|$ intersects $B(a,2r)$ then 
$\int_{\gamma} g_1 \, ds \geq 1$. Similarly, if  $|\gamma|$ intersects $B(b,2r)$ then 
$\int_{\gamma} g_2 \, ds \geq 1$. If $|\gamma|$ intersects $|\eta|$, then $\int_\gamma \delta^{-1}\chi_V \, ds \geq 1$. Otherwise 
$\gamma$ either passes through $A_{t-h,t}$ or $A_{s-h,s}$, in which case $\int_{\gamma} \rho \chi_\Omega/h \, ds \geq 1$. 

Now, for $0<m<1$, also the function 
$$
w_m = (1-m)\rho + mg 
$$
is weakly admissible for $\Delta(\zeta_1,\zeta_3;Q)$. By the minimizing property of $\rho$, we have 
\begin{equation}
\label{oulu}
\int_{Q} \rho^2 \, d\hausk \leq  \int_Q w_m^2 \, d\hausk. 
\end{equation}
Differentiating in \eqref{oulu} with respect to $m$ and setting $m =0$ then gives 
\begin{equation}
\label{juusi}
\int_Q \rho^2 \, d\hausk \leq \int_Q \rho g \, d\hausk. 
\end{equation}
To conclude the proof, we show that this is a contradiction. Recall that $\rho(x)=0$ for almost every $x \in V$. In particular, 
\begin{equation}
\label{kaka}
\int_Q \rho \delta^{-1}\chi_V \, d\hausk=0.
\end{equation} 
Also, H\"older's inequality, \eqref{riisi} and the minimizing property of $\rho$ give 
\begin{equation}
\label{loka}
\int_Q \rho (g_1+g_2) \, d\hausk \leq \mu. 
\end{equation}
Therefore, combining the definition of $g$ with \eqref{karaka}, \eqref{kaka} and \eqref{loka} gives 
$$
\int_Q \rho g \, d\hausk \leq M_1-4 \mu = \int_Q \rho^2 \, d\hausk - 4 \mu. 
$$
This contradicts \eqref{juusi}. The proof is complete. 
\end{proof}

By Proposition \ref{constbound}, we can extend $v$ to all of $Q$: if $V$ is a connected component of $Q \setminus \overline{U}$ 
and $x \in V$, then $v(x)=v(y)$, where $y$ is any point on $\partial_* V$. 


\begin{proposition}
\label{vcont}
Suppose $X$ satisfies \eqref{alaraja} and \eqref{nollamoduli}. Then $v$ is continuous in $Q$. Moreover, $v(x)=0$ for every $x \in \zeta_2$, and $v(x)=M_1$ for every $x \in \zeta_4$.  
\end{proposition}

\begin{proof}
The proof is similar to the proof of Proposition \ref{constbound}. In view of the definition of $v$ and Proposition \ref{constbound}, 
it suffices to show that $v|\overline{U}$ is continuous. Fix $a \in \overline{U}$. Moreover, let $\mu, r >0$, $x,y \in B(a,r) \cap U$, and suppose 
$v(x) \leq v(y)-7 \mu$. Using the notation of Proposition \ref{constbound} for $v(x)$ and $v(y)$, we then conclude that \eqref{karaka} holds. This time we choose $r>0$ small enough and a Borel function $g_1$ which is admissible for $\Delta(\zeta_1,B(a,r);Q)$, such that \eqref{riisi} holds with $g_2$ removed from the estimate. Then we define 
$$
g= h^{-1} \rho \chi_{\Omega}+g_1,  
$$
and conclude as above that $g$ is weakly admissible for $\Delta(\zeta_1,\zeta_3;Q)$. This then leads to a contradiction 
precisely as in the proof of Proposition \ref{constbound}. So we conclude that, when $r>0$ is small enough, 
$v(y) \leq v(x) + 7 \mu$. Interchanging the roles of $x$ and $y$, and recalling the definition of $v$, we then have 
$$
\sup_{p,q \in B(a,r) \cap \overline{U}} |v(p)-v(q)| \to 0 \quad \text{as }  r \to 0. 
$$
We conclude that $v|\overline{U}$ is continuous at $a$, and furthermore that $v$ is continuous at every $b \in Q$ by the discussion above. The claims of the proposition now follow directly from the definition and continuity of $v$, Lemma \ref{dividingmod}, and Proposition \ref{constbound}. 
\end{proof}


We orient $X$ so that winding around $\partial Q$ starting from $\zeta_1$ and ending at $\zeta_4$ 
defines positive orientation. 

\begin{corollary}
\label{royh}
Suppose $X$ satisfies \eqref{alaraja} and \eqref{nollamoduli}. Then  
$$
f:=(u,v):Q \to [0,1]\times [0,M_1]
$$
is continuous and surjective. Moreover, for every $z \in (0,1)\times (0,M_1)$ we have 
$$
\operatorname{deg}(z,f,Q)=1, 
$$
where $\operatorname{deg}(z,f,Q)$ is the topological degree of $z$ with respect to $f$ and the domain $Q$. 
\end{corollary}

\begin{proof}
First, $f$ is continuous by Theorem \ref{continuity} and Proposition \ref{vcont}. Also, $f$ maps $\partial Q$ onto 
$\partial ([0,1]\times [0,M_1])$ winding around once with positive orientation, so the topological degree $\operatorname{deg}(z,f,Q)=1$ for all $z \in (0,1)\times (0,M_1)$. In particular, $f$ is surjective 
(see \cite[Chapter II]{RR}). 
 
\end{proof}


\section{Change of variables with $f=(u,v)$} 
\label{condo}
In order to prove quasiconformality, we need to establish analytic properties for $f$. In this section we prove a change of 
variables formula by employing decompositions of the rectangle $[0,1]\times [0,M_1]$ and the corresponding preimages. 

We decompose the interval $[0,1]\times [0,M_1]$ as follows. We first choose $k_0 \in \mathbb{Z}$ and $2^{-1} <m_1\leq 1$ 
such that 
\begin{equation}
\label{kankara}
m_1 2^{k_0}=M_1. 
\end{equation}

Let $k \in \mathbb{N}$, $k \geq -k_0+2$, and consider the 
rectangles 
\begin{equation}
\label{changes}
R(i,j,k)=[i2^{-k},(i+1)2^{-k}]\times[j2^{-k}m_1,(j+1)2^{-k}m_1],  
\end{equation}
where $0\leq i \leq 2^{k}-1$, $0\leq j \leq 2^{k+k_0}-1$. Then two rectangles either coincide or have disjoint interiors, and the union of all the rectangles covers all of $[0,1]\times[0,M_1]$. We denote 
$$
Q(i,j,k)=f^{-1}(R(i,j,k)), 
$$
and by $\tilde{Q}(i,j,k)$ the interior of $Q(i,j,k)$. Also, when $(i,j,k)$ is fixed, and $0\leq s <t \leq 1$, we use the notation  
$$
A_{s,t}(i,j,k)=A_{s,t} \cap \tilde{Q}(i,j,k). 
$$




\begin{lemma} 
\label{tasan}
Suppose $X$ satisfies \eqref{alaraja} and \eqref{nollamoduli}. Then we have 
\begin{equation}
\label{tasa}
\int_{Q(i,j,k)} \rho^2 \, d\haus^2 = 2^{-2k}m_1  \quad \text{and} \quad \int_{\partial Q(i,j,k)} \rho^2 \, d\haus^2=0
\end{equation}
for every $(i,j,k)$ as above.  
\end{lemma} 

\begin{proof}
Fix $k \geq -k_0+2$. We claim that 
\begin{equation}
\label{ooho}
\int_{\tilde{Q}(i,j,k)} \rho^2 \, d\haus^2 \geq  2^{-2k}m_1
\end{equation}
for every $(i,j)$. Suppose to the contrary that 
\eqref{ooho} does not apply for some $(i,j)$. Setting 
$$
\varphi(E)=\int_{u^{-1}(E) \cap \tilde{Q}(i,j,k)}  \rho^2 \, d\hausk, 
$$
we get a set function for which Lemma \ref{volder} can be applied. Since \eqref{ooho} does not hold, Lemma \ref{volder} shows that there exists a set $G \subset (i2^{-k},(i+1)2^{-k})$ of positive $1$-measure such that for every $t \in G$ 
\begin{equation}
\label{laski}
\lim_{h \to 0} \int_{A_{t-h,t}(i,j,k)} \frac{ \rho^2}{h} \, d\hausk < 2^{-k}m_1. 
\end{equation}
In particular, for some $t \in G$ the level set $u^{-1}(t)$ is a simple curve $\gamma_t$. 
By the definition of $Q(i,j,k)$, we find $a=\gamma_t(T)$ and $b=\gamma_t(S)$ such that 
$v(a)=j2^{-k}m_1$, $v(b)=(j+1)2^{-k}m_1$, and $\gamma_t(q) \in Q(i,j,k)$ for every $T \leq q \leq S$.  
The definition of $v$ now implies 
$$
2^{-k}m_1=v(b)-v(a) \leq \lim_{h \to 0} \int_{A_{t-h,t}(i,j,k)} \frac{ \rho^2}{h} \, d\hausk, 
$$
contradicting \eqref{laski} (the detailed proof of the last inequality involves the argument used in the proof of Proposition \ref{constbound} and is left to the reader). We conclude that \eqref{ooho} holds. 

Now we can apply \eqref{kankara} and \eqref{ooho} to all $(i,j)$ to estimate
\begin{eqnarray*}
M_1 &=& \sum_{i=0}^{2^k-1} \sum_{j=0}^{2^{k+k_0}-1} 2^{-2k}m_1  \\ 
&\leq & 
\sum_{i=0}^{2^k-1} \sum_{j=0}^{2^{k+k_0}-1} \int _{\tilde{Q}(i,j,k)} \rho^2 \, d\haus^2 
\leq \int_Q \rho^2 \, d\hausk=M_1.
\end{eqnarray*}
This gives \eqref{tasa}. The proof is complete. 

\end{proof}

Applying Lemma \ref{tasan} gives the desired change of variables formula. 

\begin{proposition}
\label{huuto}
Suppose $X$ satisfies \eqref{alaraja} and \eqref{nollamoduli}. If 
$g: [0,1]\times[0,M_1] \to [0,\infty]$ is Borel measurable, then 
$$
\int_{[0,1]\times[0,M_1]} g(y) \, dy =\int_{Q} g(f(x)) \rho^2(x) \, d\hausk(x). 
$$
\end{proposition}

\begin{proof}
By Monotone Convergence, we may assume that $g$ is bounded. Let first $g_k$ be of the form 
\begin{equation}
\label{icicles}
g_k=\sum_j a_j \chi_{R_j}, 
\end{equation}
where $a_j \geq 0$ and $R_j$ is a rectangle of the form \eqref{changes} for every $j$, such that the rectangles have disjoint interiors. Then \eqref{tasa} gives 
\begin{eqnarray*}
\int_{[0,1]\times [0,M_1]} g_k(y) \, dy &=& \sum_j a_j |R_j|=\sum_j \int_{Q_j} a_j \rho^2(x)\, d\hausk(x) \\
&=& \int_Q g_k(f(x))\rho^2(x)\, d\hausk(x). 
\end{eqnarray*}

In general, the bounded Borel function $g$ can be approximated in $L^1([0,1]\times [0,M_1])$ by continuous functions and furthermore by uniformly bounded functions $g_k$ of the form \eqref{icicles} such that 
$$
g_k(y)\to g(y) \quad \text{for almost every } y \in [0,1]\times [0,M_1].  
$$
Now it suffices to show that
$$
\lim_{k \to \infty} \int_Q g_k(f(x))\rho^2(x)\, d\hausk(x) = \int_Q g(f(x))\rho^2(x)\, d\hausk(x). 
$$
We claim that the set $E= \{x \in Q: \, g_k(f(x)) \to g(f(x)) \}$ satisfies 
\begin{equation}
\label{aggae}
\int_{Q \setminus E} \rho^2(x) \, d\hausk(x)=0. 
\end{equation}
The proposition follows from \eqref{aggae}, the definition of $E$, and Dominated Convergence. 

To prove \eqref{aggae}, notice that $|f(Q\setminus E)|=0$. Hence, applying 
\eqref{tasa} again shows that, given $\epsilon>0$, the set $f(Q \setminus E)$ can be covered by rectangles $R_\ell$ of  the form \eqref{icicles} such that 
$$\int_{Q \setminus E} \rho^2(x) \, d\hausk(x)  \leq \sum_{\ell} \int_{Q_\ell} \rho^2(x) \, d\hausk(x) \\ 
= \sum_{\ell} |R_\ell|   < \epsilon.  
$$
So \eqref{aggae} follows. The proof is complete. 
\end{proof}

\begin{remark}
In the next section we show that $f$ is one-to-one. Combining this with Proposition \ref{huuto} 
shows that $f$ satisfies Condition $(N)$: If $E \subset Q$ and $\hausk(E)=0$, then $|f(E)|=0$. 
To see this, apply the change of variables formula to the function $g=\chi_{f(E)}$. 
\end{remark}



\section{Invertibility of $f$} \label{homeosection}

In this section we show that the map $f:Q \to [0,1] \times [0,M_1]$ is a homeomorphism. In particular, $v$ is then defined by \eqref{mahdo} in every $x \in Q$. 

\begin{proposition}
\label{homeo}
Suppose $X$ satisfies \eqref{alaraja} and \eqref{nollamoduli}. Then $f:Q \to [0,1]\times[0,M_1]$ is a homeomorphism. 
\end{proposition}

The rest of this section is devoted to the proof of Proposition \ref{homeo}. We first show that $f^{-1}(z)$ does not 
contain non-trivial continua for any $z \in [0,1]\times [0,M_1]$. 





Suppose $F \subset f^{-1}(z)$ is a non-trivial continuum. There exists a non-trivial continuum $E \subset \operatorname{int} Q$ such that $E \cap f^{-1}(z)= \emptyset$. 

We first give an estimate resembling a lower modulus bound. Let $\tau:[0,1]\to Q$ be a simple path such that 
$\tau(T) \in E$ if and only if $T=0$ and $\tau(T)\in F$ if and only if $T=1$. If there exists $a \in F \cap \operatorname{int} Q$, 
then we choose $\tau$ so that $\tau(1)=a$. Otherwise $F$ is a simple curve in $\partial Q$, and we choose $\tau$ so that 
$\tau(1)$ is not a boundary point of $F$ in $\partial Q$.  

Consider the distance function 
$$
\psi(x)=\operatorname{dist}(x,|\tau|). 
$$
Now $\psi$ is $1$-Lipschitz, and there exists $\delta >0$ such that $\psi^{-1}(\delta)$ intersects both $E$ and $F$. 
Moreover, by Lemma \ref{separates} we can choose a connected component $G$ of 
$$
X \setminus (E \cup F \cup |\tau| \cup \psi^{-1}(\delta))
$$ 
such that $|\tau| \subset \partial G$. Notice that $G \subset \psi^{-1}((0,\delta)) \cap Q$. 

By Proposition \ref{findcurves} and Proposition \ref{coarea}, for almost every $0<s<\delta$ the preimage $\psi^{-1}(s)$ 
contains a simple rectifiable path 
$$
\beta_s:[0,1] \to \psi^{-1}(s), \quad \beta_s \notin \Gamma_0, 
$$
in $\overline{G}$ such that $\beta_s(T) \in E$ if and only if $T=0$ and $\beta_s(T)\in F$ if and only if $T=1$ (recall that $\Gamma_0$ is an exceptional path family of modulus zero). Then, if $0<s-h<s$, there exists a unique component $V_{s-h,s} \subset G$ of 
$$
X \setminus (E \cup F \cup |\beta_s| \cup \psi^{-1}(s-h)) 
$$
such that $|\beta_s| \subset \partial V_{s-h,s}$. 

For the rest of this section, we use the notation $B'(r)=f^{-1}(B(z,r))$. 

\begin{lemma}
\label{floyd}
Suppose $X$ satisfies \eqref{alaraja} and \eqref{nollamoduli}. Fix $0<s<\delta$ as above. 
Moreover, let $R>r>0$ such that $B'(2R) \cap E = \emptyset$. Then 
\begin{equation}
\label{ovela}
\frac{1}{4} \log_2 \frac{R}{r} \leq \int_{|\beta_s|} \frac{\rho\chi_{Q\setminus B'(r)}}{|f-z|} \, d\haus^1 +\liminf_{h \to 0} \int_{V_{s-h,s}} \frac{\rho \chi_{Q \setminus B'(r)}}{h{|f-z|}}\, d\hausk. 
\end{equation}
\end{lemma}

\begin{proof}
We may assume that there exists $L \in \mathbb{N}$ such that $R=2^Lr$. Let $j \in \mathbb{N}$, $j \leq L$. Now denote 
\begin{eqnarray*}
m_j &=& \max\{T \in [0,1]: \, \beta_s(T) \in A(j+1)\},  \\
M_j &=& \min\{S \in [0,1]: \, \beta_s(S) \in A(j-1)\}, 
\end{eqnarray*}
and $\gamma_j=\beta_s|[m_j,M_j]$. Here 
$$
A(j) = f^{-1}(\overline{B}(z,2^{j+1}r)\setminus B(z,2^jr)). 
$$
Then $m_1 \leq M_1 \leq m_2 \leq M_2....$, and 
$$
|f(\beta_s(m_j))-f(\beta_s(M_j))| = 2^{j}r. 
$$
Therefore, either 
\begin{equation}
\label{unyt}
|u(\beta_s(m_j))-u(\beta_s(M_j))| \geq 2^{j-1}r, 
\end{equation}
or 
\begin{equation}
\label{vnyt}
|v(\beta_s(m_j))-v(\beta_s(M_j))| \geq 2^{j-1}r. 
\end{equation}
If \eqref{unyt} occurs, then, since $\beta_s \notin \Gamma_0$, \eqref{upperg} is satisfied with $\gamma_j$. In other words, 
\begin{equation}
\label{uupee}
2^{j-1}r \leq \int_{\gamma_j} \rho \, ds= \int_{|\gamma_j|} \rho \, d\haus^1.  
\end{equation}
We claim that if \eqref{vnyt} occurs, then 
\begin{equation}
\label{veepee}
2^{j-1}r \leq \liminf_{h \to 0} \int_{V_{s-h,s}\cap A(j)} \frac{\rho \chi_{Q \setminus B'(r)}}{h}\, d\hausk.
\end{equation}
Suppose for the moment that \eqref{veepee} holds. Then, applying \eqref{uupee} and \eqref{veepee}, we have 
\begin{eqnarray*}
1 & \leq &  \int_{|\gamma_j|} \frac{\rho}{2^{j-1}r} \, d\haus^1 +
\liminf_{h \to 0} \int_{V_{s-h,s}\cap A(j)} \frac{\rho \chi_{Q \setminus B'(r)}}{h2^{j-1}r}\, d\hausk \\
& \leq & 4  \Big (\int_{|\gamma_j|} \frac{\rho}{|f-z|} \, d\haus^1 +
\liminf_{h \to 0} \int_{V_{s-h,s}\cap A(j)} \frac{\rho \chi_{Q \setminus B'(r)}}{h|f-z|}\, d\hausk \Big). 
\end{eqnarray*}
Summing over $j$ and recalling $L=\log_2 R/r$ yields \eqref{ovela}. 

It remains to prove \eqref{veepee}, assuming that \eqref{vnyt} holds. The argument is almost identical to what we have already seen in 
the proof of Proposition \ref{constbound}. Fix $\epsilon >0$. Recall that $v$ is constant on every component of $Q \setminus \overline{U}$ by the remark after Proposition \ref{constbound}, where $U$ is as defined in Section \ref{conjug}. Therefore, taking a subpath of $\gamma_j$ if necessary, we may assume that 
$$
\gamma_j(m_j), \gamma_j(M_j) \in \overline{U}. 
$$
Now we find $0 < a,b <1$ and points $x_a,x_b \in U$ such that 
$u^{-1}(a)$ and $u^{-1}(b)$ are simple paths $\gamma_a$ and $\gamma_b$ joining $\zeta_2$ and $\zeta_4$, and
$$
x_a =\gamma_a(T) \in B(\beta_s(m_j),\epsilon), \quad x_b =\gamma_b(S) \in B(\beta_s(M_j),\epsilon). 
$$
From now on the argument proceeds exactly as the proof of Proposition \ref{constbound}, so we only recall the main points. 
We choose a small $h>0$, and define a weakly admissible function for $\Delta(\zeta_1,\zeta_3;Q)$ as follows. First, near $\gamma_a$ and $\gamma_b$, we apply the function $\rho$ the same way as in the definition of $v(x_a)$ and $v(x_b)$. Then, near 
$\gamma_j$ we use the function $\psi/h$. Finally, in $Q \setminus (B(x_a,\epsilon) \cup B(x_b,\epsilon))$ we apply 
\eqref{nollamoduli} to construct an admissible function $p_\epsilon$ for $\Delta (\partial Q, B(x_a,\epsilon) \cup B(x_b,\epsilon);Q)$ such 
that the integral of $p_{\epsilon}^2$ is small. We take the sum $g$ of these functions, and test the minimizing property of $\rho$ for 
$M_1=\modu(\zeta_1,\zeta_3;Q)$ with the function $(1-\alpha) \rho + \alpha g$.  Taking $\epsilon \to 0$ and $\alpha \to 0$, we 
arrive at \eqref{veepee}. 
\end{proof}


\begin{lemma}
\label{jencri}
For every $z \in [0,1] \times [0,M_1]$ the set $f^{-1}(z)$ is totally disconnected. 
\end{lemma}

\begin{proof}
We prove that the continuum $F$ above is trivial. Fix $R>0$ as in Lemma \ref{floyd}.  For the moment, we denote $\phi=\frac{\rho\chi_{Q\setminus B'(r)}}{|f-z|}$, and 
$$
\varphi((s-h,s)):= \int_{V_{s-h,s}} \phi\, d\hausk. 
$$
Then applying the coarea inequality, Proposition \ref{coarea} to the distance function $\psi$ (recall that $|\beta_s| \subset \psi^{-1}(s)$), we have 
\begin{equation}
\label{hhe}
\int_{0}^{\delta} \int_{|\beta_s|}\phi \, d\haus^1 \, ds \leq \frac{4}{\pi} \int_{Q \setminus B'(r)} \frac{\rho}{|f-z|}\, d\hausk. 
\end{equation}
On the other hand, $\varphi$ extends to a set function with differential $\varphi'(s)$ at almost every $s$, and 
$$
\int_0^\delta \varphi'(s) \, ds \leq \varphi((0,\delta))
$$
by Lemma \ref{volder}. Therefore, 
\begin{equation}
\label{orma}
\int_0^{\delta} \liminf_{h \to 0}  \int_{V_{s-h,s}} \frac{\phi}{h}\, d\hausk \, ds \leq  \int_{Q \setminus B'(r)} \frac{\rho}{|f-z|}\, d\hausk. 
\end{equation}
Combining \eqref{ovela}, \eqref{hhe} and \eqref{orma}, we have 
$$
\frac{\delta \log_2 \frac{R}{r}}{4}\leq 3  \int_{Q \setminus B'(r)} \frac{\rho}{|f-z|}\, d\hausk.
$$ 
Furthermore, applying H\"older's inequality and Proposition \ref{huuto}, we have 
\begin{eqnarray*}
\frac{\delta \log_2 \frac{R}{r}}{12} & \leq &  \hausk(Q)^{1/2} \Big( \int_{Q\setminus B'(r)} \frac{\rho^2}{|f-z|^2} \, d\hausk\Big)^{1/2}\\
 & \leq & \hausk(Q)^{1/2} \Big( \int_{([0,1]\times [0,M_1])\setminus B(z,r)} \frac{1}{|y-z|^2} \, dy \Big)^{1/2} \\
 & \leq &2 \pi \hausk(Q)^{1/2} \Big(\log_2 \frac{\max\{2,2M_1\}}{r} \Big)^{1/2}. 
\end{eqnarray*}
But this is a contradiction when $r \to 0$. The lemma follows. 
\end{proof}

\begin{proof}[Proof of Proposition \ref{homeo}]
We know that $f$ is continuous and surjective by Corollary \ref{royh}. We use the notation of Section \ref{conjug}; 
$$
\mathbb{F}=\{t \in [0,1]: \, u^{-1}(t)  \text{ is a simple curve }  |\gamma_t| \}, 
$$
and 
$$
U=\{x \in Q: u(x) \in \mathbb{F} \}. 
$$
Then for every $t \in \mathbb{F}$ the conjugate function $v$ is increasing on $|\gamma_t|$ by construction, so 
$f^{-1}(t,s)$ is a continuum for all $s$. But then $f^{-1}(t,s)$ has to be a point by Lemma \ref{jencri}. 
We conclude that if $z=(t,s)$, $t \in \mathbb{F}$, then $f^{-1}(t)$ is a point. It remains to prove that 
the same property holds when $t \notin \mathbb{F}$. Fix such $t$, and $z=(t,s)$. 

We now claim that $f^{-1}(z)$ contains a point $x_z$ with the following property: for every $\epsilon >0$ the 
$x_z$-component $V$ of $f^{-1}B(z,\epsilon)$ contains points $a,b \in U$ such that $u(a)<t$ and $u(b)>t$. 
Indeed, we know that the set $u^{-1}(t)$ separates $A=\{u <t\}$ and $B= \{ u >t\}$ in $Q$. Consider the set 
$$
C= \overline{A} \cap \overline{B}. 
$$
Recall that $u^{-1}(t)$ does not contain interior points by Proposition \ref{constbound} and Lemma \ref{jencri}. 
Therefore $C$ is non-empty and also separates $A$ and $B$ in $Q$, since $\overline{A} \setminus C$ and $\overline{B} \setminus C$ 
are open and disjoint in $Q$. We conclude that $C$ contains a continuum joining $\zeta_2$ and $\zeta_4$, 
so in particular $v$ takes all values between $0$ and $M_1$ in $C$. The claim follows. 

Notice that by Proposition \ref{constbound} and Lemma \ref{jencri}, $U$ is dense in $Q$. Suppose there exists a 
point $x_0 \in f^{-1}(z)$, $x_0 \neq x_z$, and let $\epsilon >0$. Then by the density of $U$, the $x_0$-component $W$ 
of $f^{-1}B(z,\epsilon)$ contains a point  $a_0 \in U$ such that $u(a_0)<t$ or $u(a_0)>t$. Without loss of generality, assume that 
$u(a_0)<t$. Connecting $a_0$ to $x_0$ in $W$ and $a$ to $x_z$ in $V$, we may assume that $u(a_0)=u(a)=t_0<t$. 

We now have components $V$ and $W$ of $f^{-1}B(z,\epsilon)$, and points $a \in V$ and $a_0 \in W$, such that 
$$
f(a)=(t_0,r) \in B(z,\epsilon), \quad f(a_0)=(t_0,s) \in B(z,\epsilon). 
$$
Recall that the restriction of $f$ to $|\gamma_{t_0}|$ is injective. In particular, $a$ and $a_0$ can be connected 
by a subcurve $\eta$ of $|\gamma_{t_0}|$ such that $f(\eta) \subset B(z,\epsilon)$. So we conclude that in fact $V=W$. 
But this is a contradiction when $\epsilon$ is small enough, since $f^{-1}(z)$ is not connected by Lemma \ref{jencri}. 
The proof is complete. 


\end{proof}

\section{Variational modulus} \label{vari}
In the next section we prove perhaps the most intricate property of our map $f$, showing that under the reciprocality assumption the function $C \rho$ is a weak upper gradient of $f$ when $C$ is large enough. To accomplish this, we now introduce a modified version of 
the conformal modulus called variational modulus, and prove a reciprocality result connecting the variational modulus to conformal modulus. The variational modulus appears, though implicitly, in the work of Gehring \cite{Ge3} in Euclidean space, where it coincides with conformal modulus. 

Let $Q^0 \subset X$ be homeomorphic to a closed square, with boundary segments $\zeta_1^0,\zeta_2^0,\zeta_3^0,\zeta_4^0$. Denote by $\Lambda$ the family of simple paths joining $\zeta_2^0$ and $\zeta_4^0$ in $Q^0 \setminus \zeta^0_1$. Fix $\gamma \in \Lambda$, and let $N_\epsilon(\gamma)$ be the closed $\epsilon$-neighborhood of $|\gamma|$. Then, when $\epsilon>0$ is small 
enough, $N_{\epsilon}(\gamma) \cap \zeta^0_1=\emptyset$. We denote by $F(\gamma)$ the connected component 
of $Q^0 \setminus |\gamma|$ containing $\zeta_1^0$. Moreover, we denote 
\begin{eqnarray*}
\Gamma_\epsilon(\gamma)&:=& \Delta(|\gamma|,F(\gamma) \setminus \operatorname{int} N_{\epsilon}(\gamma);F(\gamma) \cap N_\epsilon(\gamma)),  \\
\mathcal{F}_\epsilon(\gamma)&:=& \{g: g \text{ is weakly admissible for }\Gamma_\epsilon(\gamma) \}. 
\end{eqnarray*}

We say that a Borel function $H \geq 0$ is \emph{$V$-admissible} for $\Lambda$, if 
\begin{equation}
\label{erjantai}
\inf_{\gamma \in \Lambda} \liminf_{\epsilon \to 0} \inf_{g \in \mathcal{F}_\epsilon(\gamma)} \int_{Q^0} g H\, d\haus^2 \geq 1, 
\end{equation}
and define the \emph{variational modulus} $\overline{\modu}(\Lambda)$ by 
$$
\overline{\modu}(\Lambda) = \inf_{H} \int_{Q^0} H^2 \, d\hausk, 
$$
where the infimum is taken over all $V$-admissible functions $H$. 

\begin{lemma}
\label{aap}
Suppose $X$ satisfies \eqref{alaraja} and \eqref{nollamoduli}. Then 
$$
\modu(\zeta^0_1,\zeta^0_3;Q^0) \cdot \overline{\modu}(\Lambda)=1. 
$$
\end{lemma}

\begin{proof}
We first prove 
\begin{equation}
\label{lut}
\modu(\zeta^0_1,\zeta^0_3;Q^0) \cdot \overline{\modu}(\Lambda) \geq 1.  
\end{equation}
Let $u_0$ be the minimizing function in $Q^0$ constructed exactly as $u$ in Section \ref{minimizer}, with minimizing weak upper gradient $\rho_0$. Moreover, let $H$ be $V$-admissible for $\Lambda$, and $\gamma_t=u_0^{-1}(t)$, $0<t<1$. Denote 
$$
\varphi((s,t))=\int_{A_{s,t}}H\rho_0 \, d\hausk, 
$$
where 
$$
A_{s,t}=\{x \in Q^0: s<u_0(x)<t\} 
$$
as before. By Lemma \ref{volder} and the $V$-admissibility of $H$, 
$$
1 \leq \varphi'(t)= \lim_{h \to 0} \frac{\varphi((t-h,t))}{h} 
$$
exists for almost every $0<t<1$, and moreover 
\begin{equation}
\label{aaraa}
1 \leq \int_0^1\varphi'(t) \, dt \leq \varphi((0,1))=\int_{Q^0} H\rho_0 \, d\hausk. 
\end{equation}
Recall also that 
$$
\int_{Q^0} \rho_0^2 \, d\hausk =\modu(\zeta_1^0,\zeta_3^0;Q^0). 
$$
Therefore, \eqref{lut} follows from \eqref{aaraa} by applying H\"older's inequality and minimizing over $H$.

To conclude the proof, we show that  
\begin{equation}
\label{happan}
\modu(\zeta^0_1,\zeta^0_3;Q^0) \cdot \overline{\modu}(\Lambda) \leq 1.  
\end{equation}

We claim that the function $\rho_0 \cdot \modu(\zeta^0_1,\zeta^0_3;Q^0)^{-1}$ is $V$-admissible for $\Lambda$. This immediately implies 
\eqref{happan}. Let $\gamma$, 
$\epsilon$ and $g \in \mathcal{F}_{\epsilon}(\gamma)$ be as in \eqref{erjantai}. Then, $g$ is in particular weakly admissible for 
$\Delta(\zeta^0_1,\zeta^0_3;Q^0)$. Applying the minimizing property of $\rho_0$, we have 
$$
\modu(\zeta^0_1,\zeta^0_3;Q^0)= \int_{Q^0} \rho_0^2 \, d\hausk \leq \int_{Q^0} ((1-s)\rho_0+sg)^2 \, d\hausk, \quad 0<s<1. 
$$
Subtracting the left integral from both sides of the inequality and letting $s \to 0$, we have 
$$
\modu(\zeta^0_1,\zeta^0_3;Q^0) \leq \int_{Q^0} g \rho_0 \, d\hausk, 
$$
verifying our claim. 
\end{proof}

We can apply the variational modulus to estimate conformal modulus in the rectangles $Q(i,j,k)$ defined in Section \ref{condo}. 
Recall that assuming reciprocality means that we assume the conditions \eqref{ylaraja}, \eqref{alaraja} and \eqref{nollamoduli}. 
Lemma \ref{applvari} is the only step in the proof of Theorem \ref{main} where condition \eqref{ylaraja} is needed. 

\begin{lemma}
\label{applvari}
Suppose $X$ is reciprocal, and let $f:Q \to [0,1]\times[0,M_1]$ be the mapping constructed in the previous sections. Let $0 \leq a_1<b_1 \leq 1$, 
$0 \leq a_2 < b_2 \leq M_1$, and 
\begin{eqnarray*}
Q^0 & = & \{x \in Q: f(x) \in [a_1,b_1] \times [a_2,b_2]\}, \\ 
\zeta_1^0 & = & \{x \in Q: f(x) \in \{a_1\} \times [a_2,b_2] \}, \\
\zeta_2^0 & = & \{x \in Q: f(x) \in  [a_1,b_1] \times \{a_2\} \}, \\
\zeta_3^0 & = & 	\{x \in Q: f(x) \in \{b_1\} \times [a_2,b_2] \}, \\ 
\zeta_4^0 & = & \{x \in Q: f(x) \in [a_1,b_1]\times \{b_2\} \}. 
\end{eqnarray*}
Then 
\begin{eqnarray}
\label{kjk1} \modu(\zeta^0_1,\zeta^0_3;Q^0) & = & \frac{b_2-a_2}{b_1-a_1}, \quad \text{and } \\ 
\label{kjk2}  \frac{b_1-a_1}{\kappa(b_2-a_2)} & \leq & \modu(\zeta^0_2,\zeta^0_4;Q^0)  \leq  \frac{\kappa(b_1-a_1)}{b_2-a_2}. 
\end{eqnarray}
\end{lemma}

\begin{remark}
Notice that in the case $[a_1,a_2]=[0,1]$, $[b_1,b_2]=[0,M_1]$, Lemma \ref{applvari} follows from Proposition \ref{homeo} and conditions \eqref{ylaraja} and \eqref{alaraja}. The main content of the lemma is the second inequality in \eqref{kjk2}, as we will see in Section \ref{regoff}. 
\end{remark}
\begin{proof}

We claim that 
\begin{equation}
\label{capsule}
\overline{\modu}(\Lambda) \leq \frac{b_1-a_1}{b_2-a_2}. 
\end{equation}
The lemma follows from \eqref{capsule}. Indeed, $\geq$ in \eqref{kjk1} follows from \eqref{capsule} and Lemma \ref{aap}, and 
$\leq$ from Proposition \ref{huuto} and the fact that $\rho/(b_1-a_1)$ is weakly admissible for $\Delta(\zeta_1^0,\zeta_3^0;Q^0)$. 
The estimates in \eqref{kjk2} then follow directly from \eqref{kjk1} and conditions \eqref{ylaraja} and \eqref{alaraja}. 

To prove \eqref{capsule}, we again apply the same argument as in the proof of Proposition \ref{constbound}. Namely, 
let $\gamma:[0,1]\to Q^0 \in \Lambda$, $\epsilon>0$, and $g \in \mathcal{F}_{\epsilon}(\gamma)$ as in the definition of $\overline{\modu}(\Lambda)$. Then 
$$
f(\gamma(0))=(t,a_2), \quad f(\gamma(1))=(s,b_2), \quad a_1 \leq t,s \leq b_1.
$$  
Also, $\gamma(0)=\gamma_{t}(T)$ and $\gamma(1)=\gamma_{s}(S)$ for some $T$ and $S$. 
Here $|\gamma_t|=u^{-1}(t)$ as before. We claim that 
\begin{equation}
\label{kivela}
\int_{Q^0} g \rho \, d\hausk \geq b_2-a_2 - \mu, \quad \mu \to 0 \text{ as } \epsilon \to 0, 
\end{equation}
i.e., that $\rho/(b_2-a_2)$ is $V$-admissible. This implies \eqref{capsule} by Proposition \ref{huuto}. 

Fix $\mu>0$. Then, by the definition of $v$ and Lemma \ref{techlevel}, we can choose $\epsilon>0$ and $h>0$ small enough such that
$A_{t-h,t}\subset N_{\epsilon}(\gamma_{t})$, and 
\begin{equation}
\label{euro}
a_2 \geq \int_{N_{\epsilon,T}(\gamma_t)\cap A_{t-h,t}} \frac{\rho^2}{h} \, d\hausk - \frac{\mu}{4}. 
\end{equation}
Similarly, we can assume $A_{s-h,s}\subset N_{\epsilon}(\gamma_{s})$, and 
\begin{equation}
\label{jeuro}
M_1-b_2 \geq  \int_{A_{s-h,s} \setminus N_{\epsilon,S}(\gamma_s)} \frac{\rho^2}{h} \, d\hausk - \frac{\mu}{4}. 
\end{equation}
We denote 
$$
\Omega = (N_{\epsilon,T}(\gamma_t)\cap A_{t-h,t}) \cup (A_{s-h,s} \setminus N_{\epsilon,S}(\gamma_s)). 
$$
Combining \eqref{euro} and \eqref{jeuro} then gives 
\begin{equation}
\label{ironma}
\int_{\Omega} \frac{\rho^2}{h}\, d\hausk \leq M_1+a_2-b_2+\frac{\mu}{2}. 
\end{equation}

By condition \eqref{nollamoduli}, when $\epsilon$ is small enough, we can choose an admissible function $p$ for 
$\Delta(\zeta_1,B(\gamma(0),2\epsilon)\cup B(\gamma(1),2\epsilon);Q)$ such that 
\begin{equation}
\label{oppu}
\int_{Q} p\rho \, d\hausk \leq \Big(\int_Q p^2 \, d\hausk\Big)^{1/2}M_1^{1/2} \leq \frac{\mu}{2}. 
\end{equation} 
Recall that $M_1=\modu(\zeta_1,\zeta_3;Q)$. Now 
$$
\overline{g}= h^{-1}\rho \chi_{\Omega} + p+ g\chi_{Q^0} 
$$
is admissible for $\Delta(\zeta_1,\zeta_3;Q)$, so testing the minimizing property of $\rho$ with 
$(1-m)\rho + m \overline{g}$, $m \to 0$, it follows that 
\begin{equation}
\label{oaar}
M_1=\int_Q \rho^2 \, d\hausk \leq \int_Q \overline{g} \rho \, d\hausk=
 \int_{\Omega}  \frac{\rho^2}{h} \, d\hausk +  \int_{Q} p\rho \, d\hausk 
 + \int_{Q^0} g\rho \, d\hausk. 
\end{equation}
Combining \eqref{oaar} with \eqref{ironma} and \eqref{oppu} gives \eqref{kivela}. The proof is complete. 
\end{proof}

\section{Regularity of $f$} \label{regoff}
By Proposition \ref{homeo}, our map $f:Q \to [0,1]\times [0,M_1]$ is a homeomorphism, assuming \eqref{alaraja} and \eqref{nollamoduli}. In this section we show that if we also assume condition \eqref{ylaraja}, then we have one of the modulus inequalities required for quasiconformality. 

Upper gradients for maps are defined similarly to upper gradients of functions. We say that a Borel function $g \geq 0$ is an \emph{upper gradient} of a map $F:(Y,\operatorname{d}_Y) \to (Z,\operatorname{d}_Z)$ between metric spaces, if 
\begin{equation}
\label{geneug}
\operatorname{d}_Z(F(a),F(b)) \leq \int_{\gamma} g \, ds 
\end{equation}
for every $a, b \in Y$ and every locally rectifiable path $\gamma$ joining $a$ and $b$ in $Y$. If moreover $Y$ is equipped with locally 
finite $\hausk$-measure, then $g$ is a weak upper gradient of $F$ if there exists an exceptional set $\Gamma'$ of modulus zero such that \eqref{geneug} holds for every $\gamma \notin \Gamma'$. 

Furthermore, if $g \in L^2(Y)$ is a weak upper gradient of $F$, then there exists an exceptional set $\Gamma''$ of modulus zero such that 
if $h\geq 0$ is a Borel function in $Z$, then 
\begin{equation}
\label{iisti}
\int_{F \circ \gamma} h \, ds \leq \int_{\gamma} (h \circ F) g \, ds 
\end{equation}
for every $\gamma \notin \Gamma''$. See \cite{HKSTbook} for the proof of this property and more information on upper gradients and absolute continuity. 

\begin{proposition}
\label{ugv}
Suppose $X$ is $\kappa$-reciprocal. Then $2000 \kappa^{1/2} \rho$ is a weak upper gradient of $f$. 
\end{proposition}

\begin{remark}
Proposition \ref{ugv} and Lemma \ref{kipu} imply in particular that $f$ belongs to the Newtonian Sobolev space $N^{1,2}(Q,\R^2)$. See \cite{HKSTbook} for the theory of Sobolev spaces in metric measure spaces. 
\end{remark}

Before proving Proposition \ref{ugv}, we apply it to prove the modulus inequality discussed above. 
\begin{corollary}
\label{modineq1}
Suppose $X$ is $\kappa$-reciprocal. Then 
$$
\modu(\Gamma) \leq  4\cdot 10^6 \kappa \modu(f \Gamma)
$$ 
for every path family $\Gamma$ in $Q$. 
\end{corollary}

\begin{proof}
Let $g$ be an admissible function for $f\Gamma$. By Proposition \ref{ugv}, the function $\rho'=2000\kappa^{1/2} \rho$ is a weak upper gradient of $f$. Therefore, for almost every rectifiable path $\gamma \in \Gamma$,  
$$
1 \leq \int_{f \circ \gamma}  g \, ds \leq \int_{\gamma} (g \circ f) \rho' \, ds 
$$
by \eqref{iisti}. Thus $(g\circ f) \rho' $ is weakly admissible for $\Gamma$. By Proposition \ref{huuto}, we have 
\begin{eqnarray*}
\int_{[0,1]\times[0,M_1]} g^2 \, dx &=& \int_Q (g \circ f)^2 \rho^2 \, d\hausk \\
&\geq& \frac{1}{4\cdot 10^6 \kappa} \int_Q (g \circ f)^2 (\rho')^2 \, d\hausk \geq 
\frac{1}{4\cdot 10^6 \kappa} \modu(\Gamma). 
\end{eqnarray*}
Taking infimum over admissible functions $g$ gives the claim. 
\end{proof}


\begin{proof}[Proof of Proposition \ref{ugv}]
We use the notation $R(i,j,k)$ and $Q(i,j,k)$ introduced in Section \ref{condo}. We fix $k$ and denote 
$$
\hat{Q}(i,j,k)= \bigcup_{|i'-i| \leq 1, |j'-j|\leq 1} Q(i',j',k). 
$$
In other words, $\hat{Q}(i,j,k)$ is the preimage under $f$ of a rectangle $\hat{R}(i,j,k)$ with the same center as $R(i,j,k)$, so that 
$\hat{R}(i,j,k)$ is a scaled copy of $R(i,j,k)$ with scaling factor $3$, except when $Q(i,j,k)$ intersects $\partial Q$. We will consider four subsets of $\hat{Q}(i,j,k)$ together with their boundaries. We denote 
\begin{eqnarray*}
P_1(i,j,k) & = & \bigcup_{|j'-j|\leq 1}Q(i-1,j',k), \quad P_2(i,j,k)  =  \bigcup_{|i'-i|\leq 1}Q(i',j-1,k) \\ 
P_3(i,j,k) & = & \bigcup_{|j'-j|\leq 1}Q(i+1,j',k), \quad P_4(i,j,k)  =  \bigcup_{|i'-i|\leq 1}Q(i',j+1,k).  
\end{eqnarray*}
Then the union of the sets $P_{\ell}$ forms a topological annulus around $Q(i,j,k)$, except when $Q(i,j,k)$ intersects $\partial Q$. 

The rectangles $f(P_{\ell}(i,j,k))$ have two opposite sides three times as long as the two other sides; we say that the long boundary curves of 
$P_{\ell}$ are the preimages of the longer sides. We denote by 
$$
\Gamma_{\ell}=\Gamma_{\ell}(i,j,k) 
$$
the family of rectifiable paths joining the long boundary curves of $P_{\ell}(i,j,k)$ in $P_{\ell}(i,j,k)$. Then Lemma \ref{applvari} gives 
$$
\modu(\Gamma_{\ell}) \leq 3 \kappa 
$$ 
for all $\ell$. Therefore, we can choose a weakly admissible function $\nu_{\ell}(i,j,k):Q \to [0,\infty]$ for $\Gamma_{\ell}$ such that 
\begin{equation}
\label{ammma}
\int_{P_{\ell}(i,j,k)} \nu_{\ell}(i,j,k)^2 \, d\hausk \leq 6 \kappa. 
\end{equation}
We now define $\sigma_k:Q \to [0,\infty]$, 
$$
\sigma_k = 2^{-k} \sum_{i=0}^{2^{k}-1} \sum_{j=0}^{2^{k+k_0}-1} \sum_{\ell = 1}^4 \nu_{\ell}(i,j,k)\chi_{P_{\ell}(i,j,k)}. 
$$
Notice that if $x \in \operatorname{int} Q(i,j,k)$ for some $(i,j)$, then there are at most $8$ triples $(i',j',\ell)$ such that 
$x \in P_{\ell}(i',j',k)$. Therefore, applying \eqref{ammma} and Lemma \ref{tasan}, we see that 
\begin{equation}
\label{saanparky}
\int_{Q(i,j,k)} \sigma_k^2 \, d\hausk \leq 384 \kappa \cdot 2^{-2k} \leq 768 \kappa \int_{Q(i,j,k)} \rho^2 \, d\hausk. 
\end{equation}
In particular, the sequence $(\sigma_k)$ is bounded in $L^2(Q)$, so there exists a subsequence $(\sigma_{k_n})$ converging 
weakly to $\sigma \in L^2$. Furthermore, by Mazur's lemma, there exists a sequence $(\hat{\sigma}_n)$ of convex combinations of the 
functions $\sigma_{k_n}$ converging to $\sigma$ strongly in $L^2$. Notice that \eqref{saanparky} then holds with $\sigma_k$ replaced by $\sigma$ and for every $(i,j,k)$. 

Now, if $\Omega \subset Q$ is open in $Q$, then we can take a ``Whitney decomposition" 
of $f(\Omega)$ and cover it with the sets $R(i,j,k)\subset f(\Omega)$ with disjoint interiors (and varying $k$). Then $\Omega$ is covered by the corresponding sets $Q(i,j,k) \subset \Omega$, and applying Lemma \ref{tasan} and \eqref{saanparky} with $\sigma$ 
gives 
$$
\int_{\Omega} \sigma^2 \, \dhausk \leq 768 \kappa \int_{\Omega} \rho^2 \, d\hausk. 
$$
Since this holds for every open $\Omega \subset Q$, we conclude that 
$$
\sigma(x) \leq \sqrt{768 \kappa} \rho(x) 
$$
for $\hausk$-almost every $x \in Q$. So the proposition follows if we can show that $64 \sigma$ is a weak upper gradient of $f$. 

First, notice that since the functions $\nu_{\ell}(i,j,k)$ are weakly admissible for the path families $\Gamma_{\ell}(i,j,k)$, we can choose an exceptional set $\hat{\Gamma}$ of zero modulus such that the following holds: whenever 
$\gamma$ contains a subpath $\tilde{\gamma}$ in $\Gamma_{\ell}(i,j,k) \setminus \hat{\Gamma}$, then 
$$
\int_{\gamma} \nu_{\ell}(i,j,k)\chi_{P_{\ell}(i,j,k)}\, ds  \geq 1. 
$$

Now fix $k \geq 1$ and a non-constant $\gamma:[0,1] \to Q$, $\gamma \notin \hat{\Gamma}$. Then, if 
\begin{equation}
\label{jolly}
|\gamma| \cap Q(i,j,k) \neq \emptyset, \quad \gamma(0), \gamma(1) \notin \hat{Q}(i,j,k), 
\end{equation}
we have 
\begin{equation}
\label{sdf}
\int_{\gamma} \sigma_k \chi_{\hat{Q}(i,j,k)} \, ds \geq 2^{-k}m_1. 
\end{equation}
Indeed, if \eqref{jolly} holds, then there exists $\ell \in \{1,2,3,4\}$ and a simple $\gamma_{\ell} \in \Gamma_{\ell}(i,j,k)$ such that 
$|\gamma_{\ell}| \subset |\gamma|$, so \eqref{sdf} holds by the weak admissibility of $\nu_{\ell}(i,j,k)$ and the definition of $\sigma_k$. 

On the other hand, the triangle inequality gives 
\begin{eqnarray*}
|f(\gamma(1))-f(\gamma(0))| &  \leq & \sum_{|\gamma| \cap Q(i,j,k) \neq \emptyset} \max_{x,y \in Q(i,j,k)}|f(y)-f(x)| \\
& \leq & 2^{2-k} \operatorname{card}\{(i,j): \, |\gamma| \cap Q(i,j,k) \neq \emptyset \}. 
\end{eqnarray*}
Therefore, applying \eqref{sdf} and recalling the bounded overlap of the sets $\hat{Q}(i,j,k)$, we have 
\begin{eqnarray*}
|f(\gamma(1))-f(\gamma(0))| & \leq & 8 \lim_{k \to \infty}  \sum_{|\gamma| \cap Q(i,j,k) \neq \emptyset} 
\int_{\gamma} \sigma_k \chi_{\hat{Q}(i,j,k)} \, ds  \\
& \leq & 64 \lim_{k \to \infty} \int_{\gamma} \sigma_k \, ds. 
\end{eqnarray*}
Finally recall that, by Fuglede's lemma, 
$$
\int_{\gamma} \sigma \, ds  =\lim_{k \to \infty} \int_{\gamma} \sigma_k \, ds 
$$
for almost every $\gamma$. We conclude that 
$$
64 \sigma \leq 64 \cdot \sqrt{768\kappa} \rho \leq 2000 \kappa^{1/2} \rho
$$
is a weak upper gradient of $f$. The proof is complete. 
\end{proof}


\section{Regularity of $f^{-1}$ and quasiconformality of $f$} \label{regoffmiinus1}
In this section we conclude the proof of the quasiconformality of $f$. 

\begin{theorem}
\label{joo1}
Suppose $X$ is $\kappa$-reciprocal. Then $f:\operatorname{int} Q \to (0,1) \times (0,M_1)$ is 
a $ 8 \cdot 10^6 \kappa^2$-QC homeomorphism. 
\end{theorem}

Theorem \ref{joo1} follows from Proposition \ref{homeo}, Corollary \ref{modineq1}, and Corollary \ref{modineq2} below. 
In this section we prove Sobolev regularity for the inverse map $f^{-1}$. This leads to the last modulus inequality in 
the definition of quasiconformality, finishing the proof of Theorem \ref{joo1}. 

We formulate the next results in slightly greater generality than what is needed to prove Theorem \ref{joo1}. Notice that the results 
can be applied to our map $f$, thanks to Proposition \ref{homeo} and Corollary \ref{modineq1}. 

Suppose $F:\Omega \to \Omega' \subset \R^2$ is a homeomorphism, $\Omega \subset X$ a domain. If $y=(y_1,y_2) \in \Omega'$, we define 
$$
J_{F^{-1}}(y)= \limsup_{r \to 0} \frac{\hausk(F^{-1}(R(y,r)))}{4r^2}, 
$$
where $R(y,r)=[y_1-r,y_1+r]\times [y_2-r,y_2+r]$. We will use the following facts from real analysis 
(cf. \cite[Theorem 2.12]{Ma}, \cite[Theorem 3.22]{Fol}): if $h \geq 0$ is a Borel function 
in $X$, then 
\begin{equation}
\label{arsa}
\int_{\Omega'} (h \circ F^{-1})  J_{F^{-1}} \, dy \leq \int_{\Omega} h \, d\hausk. 
\end{equation}
Also, if $F$ is our map $f$, then (see Proposition \ref{huuto})
$$
J_{f^{-1}}(y)=(\rho(f^{-1}(y)))^{-2} 
$$
for Lebesgue almost every $y \in (0,1) \times (0,M_1)$.


\begin{proposition}
\label{sobofmiinus}
Suppose $X$ satisfies \eqref{alaraja} with constant $\kappa$, and let $F:\Omega \to \Omega' \subset \R^2$ be a homeomorphism, 
$\Omega \subset X$ a domain. If there exists $K$ such that 
\begin{equation}
\label{taaskoo}
\modu(\Gamma) \leq K\modu(F\Gamma)
\end{equation} 
for every path family $\Gamma$ in $\Omega$, then $(2\kappa K J_{F^{-1}})^{1/2}$ is a weak upper gradient of $F^{-1}$. 
\end{proposition}


\begin{corollary}
\label{modineq2}
Suppose $X$ and $F$ are as in Proposition \ref{sobofmiinus}. Then  
\begin{equation}
\label{uusihata}
\modu(F\Gamma) \leq 2\kappa K \modu(\Gamma) 
\end{equation}
for every path family $\Gamma$ in $\Omega$. 
\end{corollary}

\begin{proof}
Let $\Gamma$ be a path family in $\Omega$, and let $h$ be admissible for $\Gamma$. Then, by Proposition \ref{sobofmiinus} and \eqref{iisti}, 
$$
(2\kappa K)^{1/2} (h\circ F^{-1}) J_{F^{-1}}^{1/2}
$$
is weakly admissible for $F\Gamma$, and moreover by \eqref{arsa} 
$$
\modu(F \Gamma) \leq 2 \kappa K  \int_{\Omega'} (h\circ F^{-1})^2 J_{F^{-1}}\, dy \leq 2\kappa K \int_{\Omega} h^2 \, d\hausk. 
$$
Minimizing over $h$ gives \eqref{uusihata}. The proof is complete. 
\end{proof}


The rest of this section is devoted to the proof of Proposition \ref{sobofmiinus}. The basic idea for the proof is classical in 
QC mapping theory, see \cite[Theorem 31.2]{Vabook}. However, here we replace the classical geometric conditions by the reciprocality condition \eqref{alaraja}. 


We say that a continuous function $w: \Omega' \to \mathbb{R}$ is \emph{ACL}, if $w$ is absolutely continuous on $\haus^{1}$-almost every line segment parallel to the coordinate axes. 
Notice that if $w$ is ACL, then it has partial derivatives at almost every point, defining the gradient $\nabla w$. 

For the rest of this section, we denote $H=F^{-1}$. Let $a \in X$, and denote 
$$
H_{a}(y)=\operatorname{dist}(H(y),a). 
$$ 
We will use the following fact, cf. \cite[Theorems 7.1.20 and 7.4.5]{HKSTbook} for the proof: If there exists a Borel function $g \in L^2(\Omega')$ such that, for every $a \in X$, $H_a$ is ACL and 
\begin{equation}
\label{usihata}
|\nabla H_a(y)| \leq g(y) \quad \text{for almost every }  y \in \Omega', 
\end{equation}
then $g$ is a weak upper gradient of $H$.


\begin{proof}[Proof of Proposition \ref{sobofmiinus}]
In view of the previous discussion it suffices to show that $H_a$ is ACL and that the function 
\begin{equation}
\label{aatu}
g=(2\kappa KJ_{F^{-1}})^{1/2}
\end{equation}
satisfies \eqref{usihata} for every $a \in X$. Notice that by assumption it suffices to consider the restriction of $H$ 
to an arbitrary cube $Q' \subset \Omega'$, and that by scaling and translating $\Omega'$ if necessary we may assume $Q'=[0,1]^2$.  
We denote 
$$
\varphi(G):= \hausk(H(G \times [0,1])), \quad G \subset [0,1]. 
$$
Recall from Lemma \ref{volder} that 
$$
\varphi'(t)=\lim_{h \to 0}\frac{\varphi([t-h,t+h])}{2h}  
$$
exists and is finite for almost every $0<t<1$. We fix such a $t$, and let 
$$
E \subset I_t=\{(t,b):\, 0 \leq b \leq 1\} 
$$
with $\haus^1(E) < \epsilon$. We will prove 
\begin{equation}
\label{mustasuk}
\haus^1(H(E)) \leq C \epsilon^{1/2}, 
\end{equation}
where $C$ may depend on $t$ but not on $E$. This suffices for the ACL-property since $H$ is a homeomorphism and since we can apply the same argument to the horizontal segments. 

We may assume that $E$ is a Borel set. Furthermore, since $H(E)$ is an increasing limit of compact subsets, we may assume 
that $E$ is compact. Now there exists $\delta>0$ and a covering 
$$
I_j=\{\{t\} \times [a_j,b_j]\}_{j=1}^L, \quad b_j-a_j=\delta, 
$$ 
of $E$ by segments with pairwise disjoint interiors such that $L\delta <\epsilon$. 
Indeed, since $E$ is compact, given a small open cover of $E$, there exists a subcover $\{I_j\}$, $j=1, \ldots, p$, 
consisting of open intervals. Finally, for a large enough integer $\ell$, we cover $\cup_j I_j$ with dyadic intervals 
of length $2^{-\ell}$, increasing the measure only slightly.

For $\nu>0$, denote 
\begin{eqnarray*}
\Gamma_j(\nu) &=&\Delta([t-\nu,t+\nu]  \times \{a_j\},[t-\nu,t+\nu] \times \{b_j\};T_j(\nu)) \quad \text{and}       \\
\Lambda_j(\nu)&=&\Delta(I_{t-\nu}|[a_j,b_j],I_{t+\nu}|[a_j,b_j];T_j(\nu)), 
\end{eqnarray*}
where $T_j(\nu)=[t-\nu,t+\nu] \times[a_j,b_j]$, and 
$$
I_{t \pm \nu}|[a_j,b_j]=\{(t \pm \nu,a) \in I_t: \, a_j \leq a \leq b_j\}. 
$$ 
By Lemma \ref{lowersemi}, for every $\alpha>0$ there exists $\nu < \delta$ such that 
\begin{equation}
\label{color}
\ell(\gamma) \geq (1-\alpha) \ell(H(I_t|[a_j,b_j])) \quad \text{for every } \gamma \in H \Gamma_j(\nu). 
\end{equation}
For now we choose $\alpha=1/2$. Also, choosing $\nu$ smaller if necessary we may assume that  
$$
\varphi([t-\nu,t+\nu])\leq 4 \nu \varphi'(t). 
$$
Now, the moduli of $\Gamma_j(\nu)$ and $\Lambda_j(\nu)$ are easy to calculate. Combining with assumption \eqref{taaskoo}, 
we then have 
$$
\modu(H\Lambda_j(\nu)) \leq K \modu(\Lambda_j(\nu))=\frac{K\delta}{2 \nu}. 
$$
By condition \eqref{alaraja}, 
$$
\modu(H\Lambda_j(\nu)) \geq \frac{1}{\kappa \modu(H\Gamma_j(\nu))}. 
$$
Moreover, by \eqref{color}, the constant function $2 /\ell(H(I_t|[a_j,b_j]))$ is admissible for $H\Gamma_j(\nu)$, so 
$$
\modu(H\Gamma_j(\nu)) \leq \frac{4\hausk(H(T_j(\nu))) }{\ell(H(I_t|[a_j,b_j]))^2}. 
$$
Combining the estimates, we get 
\begin{equation}
\label{myohtarv}
\ell(H(I_t|[a_j,b_j]))^2 \leq \frac{2\kappa K\delta \hausk(H(T_j(\nu)))}{\nu}. 
\end{equation}
Summing over $j$ and applying H\"older's inequality yields  
\begin{eqnarray*}
\nonumber \haus^1(H(E)) &\leq& \sum_{j=1}^L \ell(H(I_t|[a_j,b_j])) 
\leq \Big( \frac{2\kappa K\delta}{\nu}  \Big)^{1/2} \sum_{j=1}^L \Big(\hausk(H(T_j(\nu))) \Big)^{1/2} \\
 &\leq&  \Big(\frac{2L \kappa K \delta}{\nu} \Big)^{1/2}  \Big(\sum_{j=1}^L \hausk(H(T_j(\nu)))  \Big)^{1/2}. 
\end{eqnarray*}
Recalling the disjointness of the interiors of the segments $I_j$ and that $L\delta < \epsilon$, we see that the right hand term is bounded by 
$$
\Big(\frac{4 \epsilon \kappa K  \varphi([t-\nu,t+\nu])}{\nu} \Big)^{1/2}  \leq 4 (\epsilon \kappa K \varphi'(t))^{1/2}, 
$$
so \eqref{mustasuk} follows. We conclude that $H$ is ACL. 

To conclude the proof we have to show that $H=F^{-1}$ satisfies \eqref{usihata} with the function $g$ in \eqref{aatu}. Let $y=(y_1,y_2) \in \Omega'$, and 
$$
Q_0=[y_1-\delta,y_1+\delta]\times [y_2-\delta,y_2+\delta] \subset \Omega'.  
$$
We denote $E_t=\{t\} \times [y_2-\delta,y_2+\delta]$. Let $a \in X$. Then, since $H$ is ACL, so is 
$H_a=\operatorname{dist}(\cdot,a)$. Now  
\begin{eqnarray}
\nonumber \left| \int_{y_1-\delta}^{y_1+\delta} \int_{y_2-\delta}^{y_2+\delta} \partial_2 H_a (t,s)\, ds \, dt \right|  & = & 
\left| \int_{y_1-\delta}^{y_1+\delta}H_a(t,y_2+\delta)-H_a(t,y_2-\delta)\, dt  \right| \\
\label{akraari} & \leq & \int_{y_1-\delta}^{y_1+\delta} \ell(H(E_t))\, dt. 
\end{eqnarray}
Notice that choosing $\alpha$ arbitrarily small in \eqref{color} and showing \eqref{myohtarv} with this sharper bound yields 
\begin{eqnarray*}
\ell(H(E_t))& \leq &
\lim_{\nu \to 0} \Big(\frac{2\kappa K\delta \hausk(H([t-\nu,t+\nu]\times [y_2-\delta,y_2+\delta]))}{2\nu} \Big)^{1/2} \\
& \leq & (2\kappa \delta K \varphi'(t))^{1/2}
\end{eqnarray*}
whenever $\varphi'(t)$ exists, where now 
$$
\varphi(G)= \hausk(H(G \times [y_2-\delta,y_2+\delta])). 
$$ 
Combining with \eqref{akraari}, H\"older's inequality and Lemma \ref{volder}, we have 
\begin{eqnarray*}
 \left| \int_{y_1-\delta}^{y_1+\delta} \int_{y_2-\delta}^{y_2+\delta} \partial_2 H_a (t,s)\, ds \, dt \right|
 &\leq & 2 (\kappa K)^{1/2} \delta \varphi((y_1-\delta,y_1+\delta))^{1/2} \\ 
 &= & 4 (\kappa K)^{1/2} \delta^2 \Big(\frac{\hausk(H(Q_0))}{4\delta^2}\Big)^{1/2}. 
\end{eqnarray*}
Dividing both sides by $4\delta^2$, taking $\delta \to 0$ and applying the Lebesgue differentiation theorem then gives 
$$
|\partial_2H_a(y)| \leq (\kappa K J_H)^{1/2}(y) 
$$
for almost every $y \in \Omega'$. Repeating the argument gives the same estimate for $\partial_1H_a$. Combining the estimates, 
we conclude \eqref{aatu}. 
\end{proof}



\section{Existence of QC maps $f_0: X \to \R^2$ and $\to \mathbb{S}^2$} \label{globalsection}
Theorem \ref{joo1} shows the existence of QC maps on subsets of a reciprocal space $X$. In this section we finish the proof 
of Theorem \ref{main} by showing the existence of a QC map on the whole space $X$. This is done by exhausting $X$ with a sequence of subsets for which Theorem \ref{joo1} can be applied, and then using normal family arguments. 

Recall that if $(F_j)$ is a sequence of $K$-QC maps $F_j: U\to V_j$ between planar domains containing $0$ and $1$ such that 
$F_j(0)=0$ and $F_j(1)=1$ for every $j$, then $(F_j)$ is a normal family and there exists a subsequence $(F_{j_{\ell}})$ converging locally uniformly to a $K$-QC map $F$, cf. \cite[20.5,21.3,37.2]{Vabook}. Also, notice that if $F$ and $G$ are $K_1$- and $K_2$-QC homeomorphisms, 
respectively, and if 
the composition $F\circ G$ is a well-defined homeomorphism, then $F\circ G$ is $K_1K_2$-QC; this follows from the 
definition of quasiconformality. 

Theorem \ref{main} is a direct consequence of the following. 

\begin{theorem}
\label{weakmain}
Suppose $X$ is $\kappa$-reciprocal. Then there is a $512 \cdot 10^{18}\kappa^6$-QC homeomorphism from $X$ onto 
either $\R^2$ or $\mathbb{D}$. 
\end{theorem}


\begin{proof}
Let $\{X_j\}$, $\overline{X}_j \subset X_{j+1}$ for all $j$, be an exhaustion of $X$ by open topological squares such that the closures 
$\overline{X}_j$ are closed topological squares. Moreover, fix $a, b \in X_1$, $a \neq b$. 

By Theorem \ref{joo1} and the Riemann mapping theorem, there exists for every $j \in \mathbb{N}$ a $8\cdot 10^6\kappa^2$-QC homeomorphism $f_j:X_j \to B_j$, where $B_j =B(0,r_j)\subset \R^2$ is a disk, normalized such that $f_j(a)=0$ and $f_j(b)=1$. We denote the inverse map by $h_j=f_j^{-1}:B_j \to X_j$. 

Now fix $k \in \mathbb{N}$, and let 
$$
g_j^k  = f_j \circ h_k:B_k \to B_j, \quad j \geq k. 
$$
Then $g_j^k$ is $64\cdot 10^{12}\kappa^4$-QC for all $j \geq k$. Moreover, $g_j^k(0)=0$ and $g_j^k(1)=1$. Thus $(g_j^k)_{j=k}^{\infty}$ is a normal family, so there exists a subsequence $(g_{j_k}^k)$ converging locally uniformly to a 
$64\cdot 10^{12}\kappa^4$-QC map $g_k:B_k \to \R^2$. It follows that 
$$
f_{j_k}|X_k =g_{j_k}^k \circ f_k \to g_k \circ f_k :X_k \to \R^2. 
$$ 
Taking a diagonal subsequence $(f_{\ell})$, we see that $(f_{\ell}|X_k)$ converges for every $k \in \mathbb{N}$ to 
a $512\cdot 10^{18}\kappa^6$-QC map. We conclude that the pointwise limit map $f: X \to \R^2$ is 
$512\cdot 10^{18}\kappa^6$-QC. Applying the Riemann mapping theorem if necessary, we see that the image 
$f(X)$ can be chosen to be either $\R^2$ or $\mathbb{D}$. 
\end{proof}


We now consider the case where $Y$ is homeomorphic to the Riemann sphere $\mathbb{S}^2$. The reciprocality conditions now easily generalize; we assume $\haus^2(Y)< \infty$, that \eqref{ylaraja} and \eqref{alaraja} hold for all topological squares in $\mathbb{S}^2$, and \eqref{nollamoduli} for all points and topological annuli. 

\begin{theorem}
\label{qcsphere}
With the above assumptions, there exists a $512\cdot 10^{18}\kappa^6$-QC homeomorphism $f:Y \to \mathbb{S}^2$. 
\end{theorem}
\begin{proof}
For $y_0 \in Y$, denote $X:= Y \setminus \{y_0\}$. Then $X$ satisfies the assumptions of Theorem \ref{weakmain}, so 
there exists a $512 \cdot 10^{18}\kappa^6$-QC map $f: X \to \R^2$. For $\epsilon >0$ small, consider 
$$
\Gamma_\epsilon = \Delta(\overline{B}(y_0,\epsilon), Y \setminus B(y_0,\operatorname{diam}(Y)/10);\overline{B}(y_0,\operatorname{diam}(Y)/10)\setminus 
B(y_0,\epsilon)). 
$$
Now $\modu(\Gamma_\epsilon) \to 0$ as $\epsilon \to 0$ by condition \eqref{nollamoduli}. Then also $\modu(f \Gamma_\epsilon) \to 0$ by the quasiconformality of $f$. Applying Proposition \ref{continuum}, we conclude that $f(X)$ does not have boundary in $\R^2$, and $f$ extends continuously, 
mapping $y_0$ to $\infty$ on the Riemann sphere. Moreover, the extension is $512 \cdot 10^{18}\kappa^6$-QC. The proof is complete. 
\end{proof}


\section{Minimizing dilatation} \label{bensajumala}


In this section we prove Theorem \ref{mainmini}. Note that the corresponding result also holds when $Y$ is homeomorphic to $\mathbb{S}^2$; this follows from the proof given below. 

The constant $2$ in Theorem \ref{mainmini} is not best possible. In fact, in view of Example \ref{norm} and the results of Behrend \cite{Beh} (see also \cite{Ba} and \cite{Bar}) on the area ratios of symmetric convex bodies, it is natural to ask if the sharp constant is $\pi/2$, or even if there always exists a QC map $f_0$ satisfying 
\begin{equation}
\label{paras}
\frac{2}{\pi} \modu(\Gamma) \leq \modu(f_0 \Gamma) \leq \frac{4}{\pi} \modu(\Gamma). 
\end{equation}
Both inequalities would be best possible by Example \ref{norm}. The results mentioned above together with the arguments in this section guarantee that there exists a 
QC map satisfying the first inequality in \eqref{paras}, and also there exists a QC map satisfying the second inequality. However, we do not know if a single map satisfies both inequalities. 


In the proof of Theorem \ref{mainmini}, we apply certain differentiability properties of Sobolev maps with values in metric spaces, together with the measurable Riemann mapping theorem and John's theorem on convex bodies. Instead of relying on the measurable Riemann mapping theorem, one could reprove it with the methods used in this paper. 

We now begin the proof of Theorem \ref{mainmini}. We will not consider the case $X \subset \mathbb{R}^N$ separately since it follows from the general arguments below. We assume $f$ is QC, and denote 
$$
h:= f^{-1}: \Omega \to X. 
$$
We will use some Lipschitz analysis. The following lemma is a special case of a statement concerning metric-valued Sobolev functions. See \cite[Theorem 8.1.49]{HKSTbook} for the proof. 


\begin{lemma}
\label{Lipschitzappr}
There exist measurable, pairwise disjoint sets $G_j$, $j=0,1,2,\ldots$, covering $\Omega$, such that $|G_0|=0$ and 
$h|G_j$ is $j$-Lipschitz continuous for all $j=1,2,\ldots$. 
\end{lemma}


Recall that every metric space $Z$ isometrically embeds to the Banach space $L^{\infty}(Z)$. Fix $j\geq 1$. Then $h|G_j$ can be extended to a Lipschitz map 
$$
h_j:\R^2 \to L^\infty(X). 
$$ 
By Kirchheim \cite[Theorem 2]{Ki}, 
$h_j$ is metrically differentiable; for almost every $x \in \R^2$ there exists a seminorm $MD(h_j,x)$ on $\R^2$ such that 
\begin{equation}
\label{valtioval}
||h_j(z)-h_j(y)||-MD(h_j,x)(z-y)=o(|z-x|+|y-x|). 
\end{equation}
We denote by $|MD(h_j,x)|$ the operator norm 
$$
\sup_{|z|=1} |MD(h_j,x)z|, 
$$
and define 
$$
g'(x)=\sum_j |MD(h_j,x)| \chi_{G_j}
$$

\begin{lemma}
\label{villeko}
\begin{itemize}
\item[(i)] The function $g'$ is a weak upper gradient of $h$.

\item[(ii)] $MD(h_j,x)$ is a non-zero norm for almost every $x \in \Omega$.  
\end{itemize}
\end{lemma}
\begin{proof}
The first claim follows from \cite[Proposition 6.3.22]{HKSTbook}. Towards the second claim, recall that $J_h$ denotes the volume derivative of $h$. Then, by Proposition \ref{huuto}, $J_h(x)>0$ for almost every $x \in \Omega$. Now \eqref{valtioval} and a density 
point argument shows that $MD(h_j,x)$ has to be a non-zero norm at almost every $x \in \Omega$. 
\end{proof}


By Lemma \ref{villeko}, we can define in $\Omega$ a field $G=G_h$ of norms which are non-zero at almost every point, as follows. Let 
$$
G_x=MD(h_j,x)
$$
if $x$ is a point of metric differentiability for $h_j$ for which Lemma \ref{villeko} (ii) holds, and $G_x=0$ otherwise.

We would like to apply the measurable Riemann mapping theorem in order to make the distortion of $h$ smaller. To this end, recall that 
the unit ball in a norm in $\R^2$ is a symmetric convex body in the Euclidean plane. In the points $x$ where $G_x$ is a non-trivial norm, denote the unit ball by 
$$
C_x=\{y \in \R^2: G_x(y)\leq 1\}.
$$
Let $E_x$ be an ellipse, $E_x\subset C_x$, whose Lebesgue measure 
is maximized among all such ellipses. 
We can now define an ellipse field $\tilde{G}$ by setting 
$$
\tilde{G}_x=E_x 
$$
whenever defined, and $\tilde{G}_x=B(0,1)$ otherwise. Also, it follows from the construction that the complex dilatation associated to the ellipse field $\tilde{G}$ is measurable. Thus the measurable Riemann mapping theorem gives a QC homeomorphism 
$\nu:\Omega \to \Omega$ such that for almost every $x \in \Omega$ there exists some $r_x>0$ so that 
\begin{equation}
\label{dsiff}
D\nu(x)(E_x)=B(0,r_x)  
\end{equation}
We denote 
$$
H=h \circ \nu^{-1}. 
$$
Also, let $C'_x=D\nu(C_x)$. Then, since $D\nu(x)$ is linear, $B(0,r_x)$ has maximal Lebesgue area among all ellipses that are 
subsets of the symmetric convex set $C'_x$. 

We have now applied the measurable Riemann mapping theorem to construct a new map $H: \Omega \to X$. We decompose $H$ 
to Lipschitz pieces $H_j=H|D_j$ according to Lemma \ref{Lipschitzappr}. Then, replacing $h$ with $H$ in Lemma \ref{villeko}, 
we see that 
\begin{equation}
\label{aa1}
g(x)=\sum_j |MD(H_j,x)| \chi_{D_j}
\end{equation}
is a weak upper gradient of $H$. Moreover, since $\nu^{-1}$ is differentiable almost everywhere with non-zero Jacobian determinant, we can apply \eqref{dsiff} to estimate the dilatation of $H$.





Recall John's theorem (cf. \cite[Theorem 2.4.25]{HKSTbook}): if $V$ is a symmetric convex body in $\mathbb{R}^n$, and $D$ an ellipsoid contained in $V$ with maximal Lebesgue measure, then 
$$
D \subset V \subset \sqrt{n} D. 
$$
Combining John's theorem and the previous construction, we have 
\begin{equation}
\label{iita}
B(0,r_x) \subset C'_x \subset B(0, \sqrt{2} r_x). 
\end{equation}
Also, 
$$
C'_x:= \{y \in \R^2: MD(H_j,x)(y) \leq R_x \} 
$$
for some $R_x>0$. Then, by \eqref{iita}, 
\begin{equation}
\label{prinssi}
\frac{R_x^2}{r_x^2} \geq J_H(x)=\frac{\hausk_{MD}(C'_x)}{|C'_x|}=\frac{\pi R_x^2}{|C'_x|} \geq \frac{ R_x^2}{2r_x^2}, 
\end{equation}
where $\hausk_{MD}$ is the Hausdorff measure with respect to the norm $MD(H_j,x)$. That $\hausk_{MD}(C'_x)=\pi R_x^2$ is 
proved in \cite[Lemma 6]{Ki}. Also, recalling \eqref{dsiff}, we have 
$$
|MD(H_j,x)| = \frac{R_x}{r_x}. 
$$
Combining the estimates, we have 
$$
\frac{|MD(H_j,x)|^2}{J_H(x)}\leq 2. 
$$
This together with \eqref{aa1} and the proof of Corollary \ref{modineq1} gives the inequality 
$$
\modu(\Gamma)\leq 2\modu(H\Gamma) 
$$
for every path family $\Gamma$ in $\R^2$. 


For the reverse inequality, first notice that 
\begin{equation}
\label{vittuvittu}
\ell(MD(H_j,x)):=\inf_{|z|=1} |MD(H_j,x)z| \geq \frac{R_x}{\sqrt{2}r_x}.   
\end{equation}
Now let $\rho$ be admissible for $\Gamma$, and $\gamma \in \Gamma$. Removing an exceptional set of modulus zero if needed, we may assume 
that $H^{-1}$ is absolutely continuous on $H \circ \gamma$, and that the $H_j$ are differentiable with non-zero volume derivative $\haus^1$- almost everywhere on $|\gamma|$. Then we have  
$$
\int_{H \circ \gamma} \frac{\rho \circ H^{-1}}{\ell(MD(H_j,\cdot))\circ H^{-1}}\, ds \geq \int_\gamma \rho \, ds \geq 1, 
$$
showing that 
$$
 \frac{\rho \circ H^{-1}}{\ell(MD(H_j,\cdot))\circ H^{-1}}
$$
is admissible for $H \Gamma$. Applying Proposition \ref{huuto}, we have 
$$
\modu(H \Gamma) \leq \int_X  \frac{(\rho \circ H^{-1})^2}{\ell((MD(H_j,\cdot))\circ H^{-1})^2} \dhausk \leq 
\int_\Omega  \frac{\rho^2 J_H}{\ell(MD(H_j,\cdot))^2}\, dx. 
$$
On the other hand,  \eqref{prinssi} and \eqref{vittuvittu} imply 
$$
 \frac{J_H}{\ell(MD(H_j,\cdot))^2} \leq 2, 
$$
so 
$$
\modu(H \Gamma) \leq 2 \modu(\Gamma). 
$$
The proof of Theorem \ref{mainmini} is complete.


\section{Existence of QC maps under mass upper bound}
\label{massboundsection}

In this section we prove Theorem \ref{rectifi}. In other words, we show that the mass upper bound \eqref{upperbound} implies the three conditions of reciprocality. We prove each condition separately in Lemma \ref{ballmoduli}, Proposition \ref{upperboundforus} and Proposition \ref{lowerboundforus}, respectively. 

That \eqref{upperbound} implies \eqref{nollamoduli} is well-known, see \cite[Lemma 7.18]{Hei}. We give a proof for completeness.  
\begin{lemma}
\label{ballmoduli}
Suppose $X$ satisfies \eqref{upperbound}. Moreover, let $x \in X$ and $R>10 r>0$ such that $X \setminus B(x,R) \neq \emptyset$. Then 
$$
\modu(\overline{B}(x,r),X\setminus B(x,R); \overline{B}(x,R)) \leq 8C_U \log_2^{-1} \frac{R}{r}. 
$$  
In particular, \eqref{nollamoduli} holds. 
\end{lemma}

\begin{proof}
Define 
\begin{equation}
\label{eikojo}
g(y)=\frac{1}{d(y,x) \log_2 \frac{R}{r}}  
\end{equation}
when $r \leq d(y,x) \leq R$, and $g=0$ elsewhere. Then $g$ is admissible for 
$$
\Delta(\overline{B}(x,r),X \setminus B(x,R);\overline{B}(x,R)).  
$$
We denote $T=\lceil \log_2 R/r \rceil$. Then, applying \eqref{upperbound} yields 
\begin{eqnarray*}
& & \modu(\overline{B}(x,r),X \setminus B(x,R);\overline{B}(x,R)) \\  
&\leq& \int_X g^2 \dhausk \leq \log_2^{-2} \frac{R}{r} \sum_{j=1}^{T}
 \int_{\overline{B}(x,2^{j}r) \setminus B(x,2^{j-1}r)} d(y,x)^{-2} \dhausk(y) \\
&\leq& 4C_U \log_2^{-2} \frac{R}{r} \sum_{j=1}^{T} 1 \leq 8C_U \log_2^{-1}\frac{R}{r}. 
\end{eqnarray*}
\end{proof}

Now we notice that the continuity of the energy minimizer $u$ holds under condition \eqref{upperbound}. We need a slight modification of Proposition \ref{coarea}. 

\begin{proposition}
\label{coareaa}
Let $A \subset X$ be Borel measurable. Assume that $w:A \to \mathbb{R}$ is Lipschitz, 
and $g:A \to [0,\infty]$ continuous such that 
$$
|w(a)-w(b)| \leq  \Big(\sup_{c \in B(a,d(a,b))}g(c)\Big) d(a,b)  
$$
for every $a,b \in A$. If $h:A \to [0,\infty]$ is Borel measurable, then 
$$
\int_{\mathbb{R}} \int_{A \cap w^{-1}(t)} h(s) \, d\haus^1(s)\, dt \leq  \frac{4 }{\pi}  \int_A g(x)h(x) \, d\hausk(x). 
$$
\end{proposition}
Proposition \ref{coareaa} is proved almost exactly as Proposition \ref{coarea}. 

\begin{lemma}
\label{kovanalka}
Suppose $X$ satisfies \eqref{upperbound}. Then $u$ satisfies the conclusions of Theorem \ref{continuity}. 
\end{lemma}

\begin{proof}
The proof of Theorem \ref{continuity} shows that it suffices to establish \eqref{needforcont}. Let $x \in X$ and $R>0$ 
such that $X \setminus B(x,R) \neq \emptyset$. 
For $0<r<R$, consider the family $\Lambda_r$ of all rectifiable paths separating $B(x,r)$ and $X \setminus B(x,R)$ in $X$. Then 
\eqref{needforcont} follows if we can show that 
\begin{equation}
\label{nalka}
\modu(\Lambda_r) \to \infty \quad \text{as } r \to 0. 
\end{equation}
Consider the function 
$$
w(y)= \frac{\log \frac{R}{d(y,x)}}{\log \frac{R}{r}} \chi_{\overline{B}(x,R)\setminus B(x,r)}. 
$$
Then $w=0$ on $S(x,R)$, $w=1$ on $S(x,r)$, and $w$ is Lipschitz continuous. More precisely, \begin{equation}
\label{olisikopi}
|w(a)-w(b)| \leq  \Big(\sup_{c \in B(a,d(a,b))}g(c)\Big) d(a,b) 
\end{equation}
for every $a,b \in \overline{B}(x,R)\setminus B(x,r)$, where $g$ is the continuous function in \eqref{eikojo}. 
Notice that the level sets of $w$ separate $B(x,r)$ and $X \setminus B(x,R)$ in $X$. Moreover, by \eqref{olisikopi} together with Proposition \ref{coareaa}, and 
Proposition \ref{findcurves}, $w^{-1}(t)$ contains a separating rectifiable curve $\eta_t \in \Lambda_r$ for almost every $0<t<1$. 

Now let $h$ be admissible for $\Lambda_r$. Then, by \eqref{olisikopi} together with Proposition \ref{coareaa} applied to $w$, and H\"older's inequality, we have 
\begin{eqnarray*}
1  \leq  \int_0^1 \int_{\eta_t} h \, d\haus^1 \, dt \leq  \frac{4}{\pi} \int_Q h g \, d\hausk 
\leq \frac{4}{\pi} \Big(\int_Q h^2 \, d\hausk  \Big)^{1/2} \Big(\int_Q g^2 \dhausk \Big)^{1/2}. 
\end{eqnarray*}
By Lemma \ref{ballmoduli}, the latter integral tends to zero when $r \to 0$. Minimizing over $h$ 
gives \eqref{nalka}. The proof is complete. 
\end{proof}



We use a simple modification of the Hardy-Littlewood maximal function. 

\begin{lemma}
\label{maxx}
Let $g \in L^2(Q)$, and define 
$$
\mathcal{M}g(x)=\sup_{r>0} \frac{1}{\hausk(B(x,5r))} \int_{Q\cap B(x,r)} g(y) \, d\hausk (y). 
$$
Then 
$$
\int_Q \mathcal{M}g(y)^2 \, d\hausk(y) \leq 8 \int_Q g(y)^2 \, d\hausk(y). 
$$
\end{lemma}

The lemma is proved by slightly modifying the standard proof for the Hardy-Littlewood maximal function on doubling spaces. More precisely, one can follow the proof given in \cite[Theorem 2.2]{Hei} step by step, but when the doubling property is used there we apply our current definition of the maximal function instead. 

\begin{proposition}
\label{upperboundforus}
Suppose $X$ satisfies \eqref{upperbound}. Then $X$ satisfies \eqref{ylaraja}. 
\end{proposition}

\begin{proof}
Fix $Q$, the boundary paths $\zeta_1,\ldots, \zeta_4$, and the minimizer $\rho$ as in Section \ref{minimizer}. Recall  
$$
\int_Q \rho^2 \, d\hausk =\modu(\zeta_1,\zeta_3;Q)=M_1.
$$ 
We would like to show that $M_2=\modu(\zeta_2,\zeta_4;Q)\leq C /M_1$. In view of Lemma \ref{maxx}, it is sufficient to show that the function $C_1 (\mathcal{M}\rho)/M_1$ is admissible for $M_2$, for large enough $C_1$ depending only on the constant $C_U$ in \eqref{upperbound}. Let $\gamma$ be a rectifiable path in $Q$ joining $\zeta_2$ and $\zeta_4$, and $\epsilon>0$. We may assume that $\gamma$ is simple. Then, testing the modulus of $\Delta(\zeta_1,\zeta_3;Q)$ with the function 
$$
x \mapsto \epsilon^{-1} \operatorname{dist}(x,|\gamma|) \chi_{N_{\epsilon}(\gamma)}
$$
as in the proof of Proposition \ref{constbound}, we notice that 
\begin{equation}
\label{mede}
\int_{N_{\epsilon}(\gamma)} \rho \, d\hausk \geq \epsilon M_1. 
\end{equation}
Here $N_{\epsilon}(\gamma)$ is the closed $\epsilon$-neighborhood of $|\gamma|$ as before. We now choose a covering of $N_{\epsilon}(\gamma)$ by balls $B(x_j,5 \epsilon)$ such that each $x_j \in |\gamma|$ and the balls $B(x_j,\epsilon)$ are pairwise disjoint. By \eqref{mede} and \eqref{upperbound}, we have 
\begin{eqnarray*}
\epsilon M_1 &\leq& \sum_j \int_{B(x_j,5\epsilon)} \rho \, d\hausk = 
\sum_j \frac{\hausk(B(x_j,31\epsilon))}{\hausk(B(x_j,31\epsilon))} \int_{B(x_j,5\epsilon)} \rho \, d\hausk 
\\
&\leq& 961 C_U \epsilon^2 \sum_j  \inf_{x \in B(x_j,\epsilon)} \mathcal{M}\rho(x)  
\leq 961 C_U \epsilon \sum_j \int_{|\gamma|\cap B(x_j,\epsilon)} \mathcal{M}\rho \, d\haus^1 \\
&\leq & 961 C_U \epsilon \int_{\gamma} \mathcal{M}\rho \, ds. 
\end{eqnarray*}
We conclude that $961 C_U(\mathcal{M}\rho)/M_1$ is admissible for $M_2$, as desired. 
\end{proof}


To conclude the proof of Theorem \ref{rectifi}, we show that the lower bound \eqref{alaraja} follows from \eqref{upperbound}. The proof is an application of Proposition \ref{coarea} and the following estimate for the minimizer $u$. 

\begin{lemma}
\label{hest}
Suppose $B(x,2r) \subset Q$. Then 
$$
r\haus^1(u(B(x,r))) \leq  \int_{B(x,2r)} \rho \, d\hausk. 
$$
\end{lemma}
\begin{proof}
Applying Proposition \ref{coarea} with the distance function from $x$ shows that $\haus^1(S(x,s))< \infty$ for almost every $r<s<2r$. Similarly, the upper gradient inequality for $u$ and $\rho$ holds for every path $\eta$ with $|\eta| \subset S(x,s)$, for almost every $s$. Fix such $s$, and let $E_j(s)$ be a connected component of $S(x,s)$ such that $E_j(s)$ separates $X$.  Then $E_j(s)$ contains a curve $\gamma_j$ that bounds a domain $U_j(s)$. Moreover, $B(x,r) \subset \cup_j U_j(s)$, so 
$$
\haus^1(u(B(x,r))) \leq \sum_j \operatorname{diam} u(U_j(s)). 
$$
By the maximum principle (Lemma \ref{epatoivo}), 
$$
\operatorname{diam} u(U_j(s)) \leq \max_{y,z \in \gamma_j}|u(y)-u(z)| 
$$
for every $j$. On the other hand, the upper gradient inequality gives 
$$
\max_{y,z \in \gamma_j}|u(y)-u(z)| \leq \int_{\gamma_j} \rho \, d\haus^1. 
$$ 
Combining the estimates, we have 
$$
\haus^1(u(B(x,r))) \leq \sum_j \operatorname{diam} u(U_j(s)) \leq \int_{\cup_j \gamma_j(s)} \rho \, d\haus^1 
\leq \int_{S(x,s)} \rho \, d\haus^1. 
$$
Integrating over $s$ and applying Proposition \ref{coarea} again gives the desired estimate. 
\end{proof}


We need a version of the coarea inequality for our minimizer $u$. We will follow the proof of Proposition \ref{coarea} given in \cite[Proposition 3.1.5]{AT}, replacing the Lipschitz condition assumed there with Lemma \ref{hest}. 

\begin{proposition}
\label{minicoarea}
Suppose $X$ satisfies \eqref{upperbound}. Let $g:Q \to [0,\infty]$ be a Borel function. Then the function 
$t \mapsto  \int_{u^{-1}(t)} g \, d\haus^1$ is measurable, and 
\begin{equation}
\label{bowie}
\int_0^1\int_{u^{-1}(t)} g \, d\haus^1 \, dt \leq 8000 C_U \int_Q g (\mathcal{M}\rho) \, d\hausk. 
\end{equation}
\end{proposition}
\begin{proof}
Fix $\ell \in \mathbb{Z}$, and denote $E=\{x \in Q: 2^{\ell} < g(x) \leq 2^{\ell+1}\}$. Then, for $j \in \mathbb{Z}$, let 
$$
E_j=\{x \in E: 2^j < \mathcal{M}\rho(x) \leq 2^{j+1}\}. 
$$
It then suffices to show that 
\begin{equation}
\label{teero}
\int_0^1 \haus^1(u^{-1}(t) \cap E_j) \, dt \leq 4000C_U \int_{E_j} \mathcal{M}\rho \, d\hausk. 
\end{equation}

Let $\epsilon >0$ and choose a finite or countable covering of $E_j$ 
by balls $B_i=B(x_i,r_i)$, $x_i \in E_j$, such that $2r_i<\epsilon$ for every $i$, and 
\begin{equation}
\label{vittupaat}
\sum_i r_i^2 \leq 10 \hausk(E_j). 
\end{equation}
We denote $\lambda B_i=B(x_i,\lambda r_i)$. Notice that removing $\partial Q$ does not affect the left side of \eqref{bowie}, so we may and will assume that $2B_i \subset Q$ for every $i$. Now, by Lemma \ref{hest}, 
$$
\sum_i r_i \haus^1(u(B_i)) \leq \sum_i \int_{2B_i} \rho \, d\hausk. 
$$
On the other hand, by \eqref{upperbound} and \eqref{vittupaat},  
\begin{eqnarray*}
\sum_i \int_{2B_i} \rho \, d\hausk &\leq& 100 C_U \sum_i r_i^2 \frac{1}{\hausk(10 B_i)} \int_{2B_i} \rho \, d\hausk \\
&\leq& 100 C_U \sum_i r_i^2 \mathcal{M}\rho(x_i) \leq 1000 C_U 2^{j+1} \hausk(E_j) \\ 
&\leq &2000 C_U \int_{E_j}\mathcal{M} \rho \, d\hausk. 
\end{eqnarray*}
Combining the estimates yields 
\begin{equation}
\label{rokka}
\sum_i r_i \haus^1(u(B_i)) \leq 2000 C_U \int_{E_j}\mathcal{M} \rho \, d\hausk.
\end{equation}
We now define 
$$
g_{\epsilon}(t)=\sum_i r_i \chi_{u(B_i)}(t).  
$$
Integrating over $t$ and taking the integral inside the sum then yields 
$$
\int_0^1 g_{\epsilon}(t) \, dt \leq \sum_i r_i \haus^1(u(B_i)). 
$$
On the other hand, by the definition of $\epsilon$-content,   
$$
\haus^1_{\epsilon}(u^{-1}(t)\cap E_j) \leq 2g_{\epsilon}(t) 
$$
for every $0<t<1$. Combining the estimates with \eqref{rokka} yields 
$$
\int_0^1 \haus_{\epsilon}^1(u^{-1}(t)\cap E_j)\, dt \leq 4000 C_U \int_{E_j}\mathcal{M} \rho \, d\hausk 
$$
(measurability follows by standard real analysis). Letting $\epsilon \to 0$, \eqref{teero} follows by monotone convergence. 
\end{proof}


\begin{proposition}
\label{lowerboundforus}
Suppose $X$ satisfies \eqref{upperbound}. Then $X$ satisfies \eqref{alaraja}. 
\end{proposition}
\begin{proof}
By Proposition \ref{minicoarea} applied to the constant function $1$, we know that $\haus^1(u^{-1}(t)) < \infty$ for almost every $t$. 
Also, by Lemma \ref{kovanalka} and Lemma \ref{simpcon}, $u^{-1}(t)$ is connected for all $t$. Therefore, by Proposition \ref{findcurves}, 
$u^{-1}(t)$ contains a simple path $\gamma_t$ joining $\zeta_2$ and $\zeta_4$ in $Q$ for almost every $t$. 
Now, if $g$ is admissible for $\modu(\zeta_2,\zeta_4;Q)$, then $\int_{\gamma_t} g \, d\haus^1 \geq 1$
for almost every $0<t<1$. Integrating over $t$ and applying Proposition \ref{minicoarea}, we have 
$$
1 \leq \int_0^1 \int_{\gamma_t} g \, d\haus^1 \, dt \leq 8000C_U \int_Q g (\mathcal{M}\rho) \, d\hausk. 
$$
Moreover, by H\"older's inequality and Lemma \ref{maxx}, 
\begin{eqnarray*}
\int_Q g (\mathcal{M}\rho) \, d\hausk &\leq& 3 \Big(\int_Q g^2 \, d\hausk\Big)^{1/2}\cdot \Big( \int_Q \rho^2 \, d\hausk \Big)^{1/2}\\
&=& 3 \Big(\int_Q g^2 \, d\hausk \Big)^{1/2} \cdot \modu(\zeta_1,\zeta_3;Q)^{1/2}. 
\end{eqnarray*}
Minimizing over $g$ gives the claim. 
\end{proof}


\section{Existence of QS maps} \label{quasisym}

In this section we prove Corollary \ref{bonkkleiner} as an application of Theorems \ref{main} and \ref{rectifi}. Recall that Corollary \ref{bonkkleiner} is proved in \cite{BK} using different methods. Theorem \ref{rectifi} can be seen as a generalization of Corollary \ref{bonkkleiner}. In \cite{BK} another generalization of Corollary \ref{bonkkleiner} is given for quasisymmetric (QS) maps in general, possibly fractal, topological spheres. Wildrick \cite{Wil} extended Corollary \ref{bonkkleiner} to surfaces homeomorphic to $\mathbb{R}^2$. 


\begin{definition}
\label{soad}
Suppose $\eta:[0,\infty)\to [0,\infty)$ is a homeomorphism. A homeomorphism $F:(Y,d) \to (Z,d')$ between metric spaces is 
\emph{$\eta$-quasisym\-met\-ric}, if 
\begin{equation}
\label{ankkk}
\frac{d'(F(x_1),F(x_2))}{d'(F(x_1),F(x_3))} \leq \eta(t)  \quad \text{whenever} \quad   \frac{d(x_1,x_2)}{d(x_1,x_3)} \leq t   
\end{equation}
for distinct points $x_1,x_2,x_3$. 
\end{definition}

\begin{remark}
\label{intti}
\begin{enumerate}
\item The \emph{metric definition of quasiconformality} requires for \eqref{ankkk} to hold with $t=1$ infinitesimally at every point $x_1 \in Y$. 
\item Notice that if $F$ is $\eta$-quasisymmetric, then the inverse $F^{-1}$ is $\eta'$-qua\-si\-sym\-met\-ric, where 
$$
\eta'(s)=\frac{1}{\eta^{-1}(\frac{1}{s})}. 
$$
\end{enumerate}
\end{remark}

\begin{definition}
A metric space $Y$ is \emph{$\lambda'$-linearly locally contractible}, if every ball $B(a,r)$ in $Y$ with radius 
$0<r<\operatorname{diam}(Y) /\lambda'$ is contractible inside $B(a,\lambda' r)$, i.e., there exists a continuous map 
$H:B(a,r) \times [0,1] \to B(a,\lambda' r)$ such that $H(\cdot,0)$ is the identity and $H(\cdot,1)$ is a constant map. 
\end{definition}


We use the chordal distance $d(a,b)=|a-b|$ in $\mathbb{S}^2$. We have now defined the concepts in the statement of Corollary \ref{bonkkleiner}. The method in the proof of Corollary \ref{bonkkleiner} assuming Theorem \ref{main} is nowadays standard in QC mapping theory and can be found 
in \cite{HeKo}. The argument is repeated here for completeness. We need three facts. First, if $E$ and $F$ are disjoint continua in $\mathbb{S}^2$ and 
$$
\operatorname{dist}(E,F) \leq T \min\{\operatorname{diam} E, \operatorname{diam} F\}, \quad 0<T<\infty, 
$$ 
then 
\begin{equation}
\label{loe}
\modu(E,F;\mathbb{S}^2) \geq \phi(T)>0, \quad \phi(T) \to \infty \text{ as } T \to 0. 
\end{equation}
This estimate is proved integrating a given admissible function over suitably chosen concentric circles intersecting both $E$ and $F$, 
and then integrating over the radius, cf. \cite[Section 10]{Vabook}, and \cite[Section 3]{HeKo}. Metric measure spaces satisfying  
\eqref{loe} are called \emph{Loewner spaces}, see \cite{HeKo}.  

Secondly, if $X$ is $\lambda'$-linearly locally contractible, then it satisfies the so-called $LLC$-conditions for all $\lambda>\lambda'$, cf. \cite{BK}: 
\begin{itemize}
\item[(1)] if $B(x,r)$ is a ball in $X$ and $a,b \in B(x,r)$, then there exists a continuum $E \subset  B(a,\lambda r)$ 
joining $a$ and $b$. 

\item[(2)] if $B(x,r)$ is a ball in $X$ and $a,b \in X \setminus B(x,r)$, then there exists a continuum $F \subset X \setminus B(a,r/\lambda)$ 
joining $a$ and $b$. 
\end{itemize}

Finally, applying Proposition \ref{coarea} and linear local contractibility as in Remark \ref{nakuttaja}, we get the lower bound
$$
C_L r^2 \leq \hausk (B(x,r)) \quad \text{whenever }  r \leq \operatorname{diam}(X), 
$$
 for measures of balls, where $C_L$ depends only on $\lambda'$. Combining with \eqref{upperbound}, we see that a space $X$ satisfying the conditions of Corollary \ref{bonkkleiner} is \emph{Ahlfors $2$-regular}.  In particular, $X$ is then a 
\emph{doubling metric space}. 

\begin{proof}[Proof of Corollary \ref{bonkkleiner}]
By Theorem \ref{rectifi} and the proof of Theorem \ref{qcsphere}, there exists a $2$-quasi\-con\-formal map $f:Y \to \mathbb{S}^2$. Then by Remark \ref{intti} it suffices to show that $f^{-1}$ is $\eta$-QS with $\eta$ depending only on $C_U$ and $\lambda'$. Moreover by Ahlfors regularity and a theorem of V\"ais\"al\"a, see \cite[Theorem 10.19]{Hei}, it suffices to show that $f^{-1}$ is 
\emph{weakly QS}, i.e., that \eqref{ankkk} holds with $t=1$. 

We first choose points $a_1,a_2,a_3 \in Y$ such that 
$$
d(a_i,a_j) \geq \frac{\operatorname{diam}(Y)}{ 2} \quad \text{for all } i \neq j. 
$$
We denote $b_j=f(a_j)$. Precomposing $f^{-1}$ with a M\"obius transformation, if necessary, we may then assume that 
$$
|b_i-b_j| \geq \frac{1}{4} \quad \text{for all } i,j. 
$$

Now take distinct points $y_1,y_2,y_3 \in \mathbb{S}^2$ such that 
\begin{equation}
\label{ade}
|y_1-y_2| \leq |y_1-y_3|. 
\end{equation}
We denote $x_k=f^{-1}(y_k)$, $k=1,2,3$. Then, by triangle inequality, 
$$
d(x_1,a_j)\geq \frac{\operatorname{diam}(Y)}{4}
$$ 
for at least two indices $j$. Among them we can then choose one of the indices, say $j=1$, such that also 
\begin{equation}
\label{eda}
|y_2-b_1| \geq \frac{1}{8}. 
\end{equation}

The $LLC$-conditions now guarantee the existence of a continuum 
$$
E \subset B(x_1,\lambda d(x_1,x_3))
$$ 
joining $x_1$ and $x_3$, and a continuum 
$$
F \subset Y \setminus B(x_1,A/\lambda), \quad A=\min\{d(x_1,x_2), \operatorname{diam}(Y)/4\}, 
$$ 
joining $a_1$ and $x_2$. 

Then, by \eqref{ade} and \eqref{eda}, the continua $f(E)$ and $f(F)$ satisfy the conditions in \eqref{loe} with $T=16$, so 
\begin{equation}
\label{ssj}
\phi(16) \leq \modu(f(E),f(F);\mathbb{S}^2) \leq 2 \modu(E,F;Y). 
\end{equation}
We may assume that 
$$
2\lambda d(x_1,x_3) \leq A/\lambda, 
$$
since otherwise there is nothing to prove. Then, by \eqref{upperbound} and Lemma \ref{ballmoduli}, 
\begin{eqnarray}
\nonumber \modu(E,F;Y) & \leq & \modu(\overline{B}(x_1,\lambda d(x_1,x_3)), Y \setminus B(x_1,A/\lambda); \overline{B}(x_1,A/\lambda)\\
\label{ssj2} & \leq &8 C_U \Big(\log \frac{A}{2\lambda^2 d(x_1,x_3)} \Big)^{-1}, 
\end{eqnarray}
where $C_U$ is the constant in \eqref{upperbound}. Combining \eqref{ssj} and \eqref{ssj2} gives 
$$
\frac{d(x_1,x_2)}{d(x_1,x_3)} \leq \frac{4A}{d(x_1,x_3)} \leq 8 \lambda^2 \exp(16C_U/\phi(16)). 
$$
We conclude that $f^{-1}$ is QS. 
\end{proof}





\section{Concluding remarks}
\label{huri}

We briefly discuss the absolute continuity properties of QC maps between $X$ and $\R^2$. It follows from Proposition \ref{huuto}, and the fact that planar QC maps satisfy Condition $(N)$, that every QC map $f:X \to \R^2$ 
satisfies condition $(N)$. One could hope for Condition $(N)$ to hold also for the inverse. Then it would follow from Lemma 
\ref{Lipschitzappr} that a reciprocal $X$ is always countably $2$-rectifiable. However, we show that this is not the case in general. 

\begin{proposition}
There exists a reciprocal $X \subset \R^3$ that is not countably $2$-rectifiable. In fact, $X$ satisfies \eqref{upperbound}. 
\end{proposition}

\begin{proof}
We only briefly describe the construction of $X$ and leave the details to the interested reader. We choose a self-similar Cantor set 
$\mathcal{C} \subset [0,1]^3$ such that 
\begin{equation}
\label{eidet}
C^{-1} r^2 \leq \haus^2(B(x,r)\cap \mathcal{C})\leq Cr^2
\end{equation}
for all $x \in \mathcal{C}$ and $0<r<1$, cf. \cite[pp. 65--67]{Ma}. Then, we construct a ``tree" consisting of tubes that follow the construction of the set $\mathcal{C}$. More precisely, each new step in the construction of $\mathcal{C}$ corresponds to a branching of the tree such that several tubes grow from every already existing tube. We can arrange the tubes such that the limiting set $X$ includes the whole set $\mathcal{C}$ so that $X$ is not rectifiable, and such that there is no overlapping so $X$ is homeomorphic to $\R^2$. Also, we can choose the area of each tube to be as small as we wish. Therefore, combining with \eqref{eidet} we can guarantee that the mass upper bound \eqref{upperbound} holds. Reciprocality then follows from Theorem \ref{rectifi}. 
\end{proof}


Our discussion is related to the so-called inverse absolute continuity problem for QS maps: 
if $f:X \to \R^2$ is QS, does $f$ satisfy condition $(N)$? See \cite{Ge1}, \cite{Ge2}, \cite{HS}, \cite{V}. There are several similar unsolved problems in QC mapping theory, see \cite{ABH} for an overview. From Theorem \ref{main} and the fact 
that planar QC maps preserve sets of measure zero, it follows that the answer is affirmative if $X$ is reciprocal. This fact 
can be also proved directly employing condition \eqref{ylaraja}, as we now demonstrate. 

\begin{proposition}
\label{invabs}
Suppose $X$ satisfies \eqref{ylaraja}, and let $f:X \to \R^2$ be QS. If $E \subset X$, then $\haus^2(E)=0$ if and 
only if $|f(E)|=0$. 
\end{proposition}


\begin{proof}
That $|f(E)|=0$ implies $\hausk(E)=0$ is well-known to hold even without assumption \eqref{ylaraja} by the works of Gehring, 
V\"ais\"al\"a and Tyson, see \cite{Ty}. Suppose 
$\hausk(E)=0$ and $|f(E)|>0$. Since $f$ is QS, by \cite{Ty} we know that there exists $K \geq 1$ such that 
$$
\modu(\Gamma) \leq K \modu(f^{-1} \Gamma) 
$$
for every path family $\Gamma$ in $\R^2$. Moreover, an examination of the proof given there shows that in fact 
\begin{equation}
\label{rampa}
\overline{\modu}(\Gamma)  \leq K \modu(f^{-1} \Gamma), 
\end{equation}
where $\overline{\modu}(\Gamma)$ is defined as $\modu(\Gamma)$ except for the definition of admissibility; we say that a Borel function $g$ is admissible for $\overline{\modu}(\Gamma)$ if $g$ is admissible for $\Gamma$ and $g=0$ almost everywhere on the set $f(E)$. 

Fix a density point $x_0$ of $f(E)$ and $\epsilon>0$. Then choose a square $Q=Q(x_0,r)$ such that 
\begin{equation}
\label{kyll}
\frac{|Q\setminus f(E)|}{4r^2} < \epsilon. 
\end{equation}
We may assume $x_0=0$. Let $\Gamma=\{\gamma_t\}$ be the family of horizontal segments joining the vertical sides of $Q$, and similarly let $\Lambda$ be the family of vertical segments joining the horizontal sides of $Q$. Then, if $g$ is admissible for $\overline{\modu}(\Gamma)$, Fubini's theorem gives 
$$
2r \leq \int_{-r}^r \int_{\gamma_t} g(s,t) \, ds dt= \int_Q g(x)\, dx. 
$$
Also, since $g=0$ almost everywhere in $f(E)$, \eqref{kyll} and H\"older's inequality give  
$$
 \int_Q g(x)\, dx \leq |Q\setminus f(E)|^{1/2}  \Big(\int_Q g(x)^2\, dx\Big)^{1/2} \leq 2 \epsilon r  \Big(\int_Q g(x)^2\, dx\Big)^{1/2}. 
$$
Combining the estimates and minimizing over admissible functions, we conclude 
$$
\overline{\modu}(\Gamma) \geq \epsilon^{-2}. 
$$
The same estimate holds with $\Gamma$ replaced by $\Lambda$. Now, by \eqref{rampa},  
$$
\epsilon^{-4} \leq \overline{\modu}(\Gamma) \cdot  \overline{\modu}(\Lambda) \leq K^2 
\modu(f^{-1}\Gamma) \cdot \modu(f^{-1}\Lambda). 
$$
But this contradicts \eqref{ylaraja} when $\epsilon>0$ is small enough. The proof is complete. 
\end{proof}


The following related question immediately arises from \cite{Ge1} and \cite{Ge2}. 
\begin{question}
\label{juhapekka}
Suppose $f:X \to \R^2$ is QS. Is $f$ QC (in the sense of Definition \ref{qcdef})?
\end{question}
Notice that a QS $f$ is automatically QC in the sense of the metric definition. The answer to Question \ref{juhapekka} is affirmative if in addition $X$ satisfies \eqref{upperbound}; this follows from \cite{HeKo} and also from Theorems 
\ref{main} and \ref{rectifi}. 

To finish, we discuss the three conditions in the definition of reciprocality. Although used only once in the proof of Theorem \ref{main}, we feel that the most important of the conditions is \eqref{ylaraja}. For instance, it is the failure of \eqref{ylaraja} that prevents the existence of a QC map in 
Example \ref{decomp}. 

\begin{question}
Does condition \eqref{ylaraja}, or some modification of it, imply \eqref{alaraja} and/or \eqref{nollamoduli}? 
\end{question}

It is not difficult to give examples of spaces that do not satisfy \eqref{nollamoduli}, but we do not know 
if such an example satisfying \eqref{ylaraja} exists. It can be proved that if $X$ satisfies \eqref{ylaraja}, then $\modu(\{x\},E;Q)=0$ for every $Q \in X$, $x \in Q$ and every compact $E \subset Q \setminus \{x\}$. Concerning condition \eqref{alaraja}, it seems that this condition should hold in great generality. 

\begin{question}
\label{aakkeli}
Does condition \eqref{alaraja} hold for all $X$? 
\end{question}


\vskip1cm 

\noindent Department of Mathematics and Statistics, University of
Jyv\"askyl\"a, P.O. Box 35 (MaD), FI-40014, University of Jyv\"askyl\"a, Finland.
\newline
{\it E-mail address:} {\bf kai.i.rajala@jyu.fi}

\end{document}